%% file: main.tex
\documentclass[12pt]{article}
\usepackage[utf8]{inputenc}

\usepackage{graphicx}
\usepackage{amsmath}
\usepackage{amsthm}
\usepackage{amssymb}

\usepackage{float}
\usepackage{caption}
\usepackage[hidelinks]{hyperref}
\usepackage{xcolor}
\usepackage{color}
\usepackage{mathtools}
\usepackage{dsfont}
\usepackage{comment}
\usepackage{enumitem}
\usepackage{cancel}
\usepackage[normalem]{ulem}
\usepackage{fullpage} 

\usepackage[vlined,nofillcomment]{algorithm2e}

\usepackage{stmaryrd} 
\usepackage{mathrsfs} 
\usepackage{multirow} 

\usepackage{xspace} 
\usepackage{graphbox} 
\usepackage{float} 

\newif\ifdraft
\drafttrue 
\newcommand{\draftonly}[1]{\ifdraft #1\fi}

\usepackage{todonotes} 

\DeclareMathOperator*{\arginf}{arginf}

\def\blue#1{{\color{blue}#1}}
\def\rd#1{{\color{red}#1}}

\def\red#1{{\color{red}#1}}

\newcommand{\black}[1]{{\color{black}#1}}
\newcommand{\becolblue}{\color{blue}}
\newcommand{\eecol}{\color{black}}
\def\argmin_#1{\underset{#1}{\mathrm{argmin\, }}}
\def\argmax_#1{\underset{#1}{\mathrm{argmax\, }}}

\newcommand{\noteb}[3]{
		{\colorbox{#2}{\bfseries\sffamily\scriptsize\textcolor{white}{#1}}}
		{\textcolor{#2}{\sf\small$\blacktriangleright$\textit{#3}$\blacktriangleleft$}}}

\newcommand{\olivier}[1]{\noteb{olivier}{white!30!blue}{#1}}

\def \E{\mathbb{E}}
\newcommand{\be}{\begin{eqnarray}}
\newcommand{\ee}{\end{eqnarray}}
\newcommand{\beno}{\begin{eqnarray*}}
\newcommand{\eeno}{\end{eqnarray*}}
\newcommand{\barr}[1]{\begin{array}{#1}}
\newcommand{\earr}{\end{array}}
\newcommand{\VECT}[1]{\left(\barr{c} #1 \earr\right)}

\newcommand{\N}{\mathbb{N}}
\newcommand{\R}{\mathbb{R}}
\newcommand{\eps}{\varepsilon}
\newcommand{\ma}{\alpha}
\newcommand{\mb}{\beta}

\newcommand{\mO}{\Omega}
\newcommand{\ml}{\lambda}
\newcommand{\ms}{\sigma}
\newcommand{\mt}{\theta}
\newcommand{\mG}{\Gamma}

\newcommand{\mT}{\Theta}
\newcommand{\cA}{{\mathcal A}} 
\newcommand{\cB}{{\mathcal B}}
\newcommand{\cC}{{\mathcal C}}
\newcommand{\cF}{{\mathcal F}}
\newcommand{\cG}{{\mathcal G}}
\newcommand{\cK}{{\mathcal K}}
\newcommand{\cL}{{\mathcal L}}
\newcommand{\cM}{{\mathcal M}}
\newcommand{\cN}{{\mathcal N}}
\newcommand{\cO}{{\mathcal O}}
\newcommand{\cT}{{\mathcal T}}

\newcommand{\cW}{{\mathcal W}}

\newcommand{\converge}{\rightarrow}
\newcommand{\implique}{\Rightarrow}
\newcommand{\equivalent}{\Leftrightarrow}
\newcommand{\conv}{\converge}

\definecolor{SpecGreen}{cmyk}{0.64,0,0.95,0.40}
\definecolor{SpecGray}{cmyk}{0.64,0,0.05,0.40}
\definecolor{apcolor}{rgb}{0.1,0.5,0.4} 

\newcommand{\Dt}{{\Delta t}}

\newcommand{\MODIF}[1]{{ #1}} 

 \newcommand{\COMMENTG}[1]{\draftonly{\small{\color{gray} #1}}}   

\newcommand{\ba}{{\bar a}}
\newcommand{\bb}{{\bar b}}

\newcommand{\supp}{supp}

\newcommand{\numA}{{\hat{\mathcal A}}} 
\newcommand{\numB}{{\hat{\mathcal B}}} 

\newcommand{\Meas}{{\mathcal M}} 
\newcommand{\disp}{\displaystyle}

\newtheorem{thm}{Theorem}[section]
\newtheorem{prop}[thm]{Proposition}

\newtheorem{lem}[thm]{Lemma}
\newtheorem{rem}[thm]{Remark}

\newtheorem{example}[thm]{Example}
\newtheorem{notation}[thm]{Notation}

\title{
  \MODIF{Representation results and error estimates for differential games with applications using neural networks}
}

\author{Olivier Bokanowski,%
\footnote{Université Paris Cité, Laboratoire Jacques-Louis Lions (LJLL), F-75013 Paris, France}
\footnote{Sorbonne Université, CNRS, LJLL, F-75005 Paris, France \texttt{olivier.bokanowski@u-paris.fr}}
\
Xavier Warin%
\footnote{EDF R\&D \& FiME, 91120 Palaiseau, France \texttt{xavier.warin@edf.fr}}
}

\newcommand{\ha}{{\hat a}}
\newcommand{\hb}{{\hat b}}

\newcommand{\hV}{{\hat V}}
\newcommand{\dt}{{\Delta t}}

\begin{document}

\date{Feb. 5, 2024; updated \today}

\maketitle

\begin{abstract}
We study deterministic optimal control problems for differential games with finite horizon. 
We propose new approximations of the strategies in feedback form and show error estimates and a convergence result 
of the value in some weak sense for one of the formulations. 
This result applies in particular to neural network approximations. 
This work follows some ideas introduced in Bokanowski, Prost and Warin (PDEA, 2023) 
for deterministic optimal control problems, yet with a simplified approach for the error estimates, 
which allows to consider a global optimization scheme instead of a time-marching scheme. 
We also give a new approximation result between the continuous and the semi-discrete optimal 
control value in the game setting, improving the classical convergence order $O(\dt^{1/2})$ to $O(\dt)$, 
under some assumptions on the dynamical system. 
Numerical examples are performed on elementary academic problems related to backward reachability, 
with exact analytic solutions given, as well as a two-player game in the presence of state constraints,
using stochastic gradient-type algorithms to deal with the minimax problem.
\end{abstract}



\medskip
{\bf Keywords}: 
{\small
differential games,
two-player games, 
neural networks, deterministic optimal control,
dynamic programming principle, 
Hamilton Jacobi Isaacs equation, 
front propagation, level sets, non-anticipative strategies.
}  


\input{sec-gen.tex} 
\input{sec-num.tex} 
\appendix               
\input{app1.tex}        
\input{app2.tex}        
\input{app-vex.tex}     

\input{sec-declarations.tex} 

\begin{small}
\bibliographystyle{abbrv}
\hypersetup{
    linkcolor=black, 
    citecolor=black, 
    urlcolor=black  }
\bibliography{bib_mtns,comp,other,bib_generalHJ,bib_games}
\end{small}
\end{document}

\pagebreak
\input{sec-step.tex}     
\input{app-dynprog.tex}      
\input{app-disc1.tex}    
\input{app-disc2.tex}    
\input{app-algo.tex}     
\input{app-algo-old.tex} 
\input{app-num.tex}     



\end{document}

%% file: sec-gen.tex
\newcommand{\hypo}{(H0)}
\newcommand{\hypf}{(H1)}
\newcommand{\hypG}{(H2)}
\newcommand{\hypPhi}{(H3)}
\newcommand{\hypXa}{(H4)}
\newcommand{\hypXb}{(H5)}
\newcommand{\hA}{{\hat \cA}}
\newcommand{\hB}{{\hat \cB}}
\newcommand{\hG}{{\hat \cG}}

\renewcommand{\obar}[1]{\overline{#1}}
\newcommand{\bbar}[1]{\overline{\overline{#1}}}
\newcommand{\barV} {\obar{V}}
\newcommand{\bbarV}{{\tilde V}}
\newcommand{\tV}{{\tilde V}}

\newcommand{\tJ}{{\tilde J}}
\newcommand{\tx}{{\tilde x}}
\newcommand{\tX}{{\tilde X}}
\newcommand{\talpha}{{\tilde\alpha}}
\newcommand{\tb}{{\tilde b}}

\newcommand{\mGT}{\mG_{(0,T)}}
\newcommand{\cS}{{\mathcal S}}
\renewcommand{\dt}{\tau}

\section{Introduction}

%
%


There exist different formulations of the value of a differential games, with notable contributions from Isaacs \cite{isa-65},
Fleming \cite{Fleming-64} (using piecewise constant controls),
as well as Friedman \cite{Friedman-71,Friedman-73-a}.
The definition proposed by Elliot and Kalton \cite{ell-kal-72}, employing non-anticipative strategies, 
has become widely accepted (see in particular \cite{Barron_Evans_Jensen_84} for the equivalence of the notions).

Non-anticipative strategies, according to this definition, 
can be interpreted as the optimal response of the first player (hereafter denoted as "$a$"), 
to an adverse control from the second player (denoted as "b"), utilizing only past knowledge. 

Building upon the viscosity approach introduced by Crandall and Lions in \cite{cra-lio-83},
the value is also a solution of an Hamilton-Jacobi-Issacs (HJI) equation for differential games~\cite{eva-sou-84}.
Recent developments in differential games with state constraints and their characterization by Hamilton-Jacobi equations are
discussed in Buckdan {\em et al.}~\cite{buc-car-qui-11};
see also Bettiol {\em et al.}~\cite{bet-car-qui-06}, Cardaliaguet {\em et al.}~\cite{car-qui-sai-00}, 
Bardi {\em et al.}~\cite{Bardi_Koike_Soravia_00}.

Additionally, Vinter {\em et al.}~\cite{vin-cla-jam-04,fal-kou-vin-12}, 
relying on Cardaliaguet {\em et al.}~\cite{car-pla-00}, explore the possibility to represent the
non-anticipative strategy 
as a function of the current state and of the adverse control [7]. 
This approach is also employed, for example, by Bardi and Soravia as illustrated in Eq.~(2.5) and Theorem 2.3 of their work \cite{bar-sor-91}.

In our work, we adopt a semi-discrete time setting, which is more straightforward to consider. 
Our primary objective is to revisit equivalent definitions of the corresponding semi-discrete values.

On the other hand, the classical error bound for the time-marching approximation 
of the value of a differential game
concerning its continuous limit is known to be of order $O(\dt^{1/2})$, 
where $\dt$ represents the time step mesh (Souganidis \cite{souganidis-85-b}, Crandall-Lions \cite{cra-lio-84} for finite difference schemes;
see also the textbook \cite{BarCap97}), and the proof relies on viscosity arguments.

In line with the approach presented in \cite{bok-gam-zid-22},
our contribution is to establish an error bound of order $O(\dt)$
under a separability assumption on the dynamics,
typically valid for pursuit-evasion games, and certain convexity assumptions.

When considering the numerical approximation of the value, 
pioneering works that employ a full-grid approach where first developed
for the approximation of the Hamilton-Jacobi-Isaacs (HJ) partial differential equation.
For games and practical implementations we mention the works of 
Bardi, Falcone and Soravia \cite{Bardi_Falcone_Soravia_94, Bardi_Falcone_Soravia_99},
particularly in connection with semi-Lagrangian schemes,
the finite difference approaches 
including Markov Chains approximations \cite{kus-dup-01}, 
finite difference schemes (including monotone schemes \cite{cra-lio-84},
semi-Lagrangian schemes as in~\cite{deb-jak-14, fal-fer-16},
ENO or WENO higher-order schemes \cite{Osher_Shu_91, Shu99}, 
finite element methods \cite{jensen-smears-2013} \cite{smears-suli-2016}, 
discontinuous Galerkin methods \cite{Hu_1999_SIAM_DG_FEM, Li_2005_AML_DG_HamiJaco}, 
Additionally, max-plus approaches have been explored \cite{aki-gau-lak-08}.

However, the utility of grid-based methods is constrained to low dimensions
due to the curse of dimensionality.
Consequently, alternative approaches have emerged, 
including sparse grid methods \cite{bok-gar-gri-klo-2013, garcke_kroner_2017},
tree structure approximation algorithms (such as~\cite{alla_falcone_Saluzzi_tree_2020}),
tensor decomposition methods \cite{dolgov_tensor_2021}, hierarchical approximations
\cite{akian2023multilevel}, radial basis function approaches \cite{fer-fer-jun-17},
and more.

Recently, neural network approximations have been developed to represent, optimize,
and address problems in average or large dimensions, often in conjunction with stochastic gradient algorithms.

It is worth noting that, in the realm of stochastic control, 
Deep Neural Network (DNN) approximations 
have previously found application in gas storage optimization, 
as in~\cite{barrera2006numerical}, where the neural network approximates the control
(specifically the quantity of gas injected or withdrawn from the storage).
This methodology has been adapted and popularized recently for solving
Backward Stochastic Differential Equations (BSDEs) in \cite{han-jen-e-18} 
(referred to as the deep BSDE algorithm), see also \cite{han_convergence_2020}. 
Additionally, notable works include \cite{hur-pha-21-a} and \cite{hur-pha-21-b}, 
focusing on approximating stochastic control problems within finite horizons using Bellman's dynamic programming principle.

In the deterministic context, \cite{ban-tom-21} employs DNNs 
to approximate the Hamilton-Jacobi-Bellman (HJB) equation, as represented in \eqref{eq:hjb-pde}, 
for solving state-constrained reachability control problems in dimensions up to $d=10$.

\MODIF{

Neural networks for differential games have been extensively studied in various works.
Approximations for pursuit-evasion games where already addressed in \cite{pes-gab-mie-bre-95} (1995) and in
\cite{johnson2009numerical} (2009), particularly in cases where the game has a value. 
More recently,~\cite{wei2023differential} considers underwater targets with state constraints (such as collision avoidance) using a reinforcement learning approach.  In~\cite{Song_et_al_2020}, a perturbation approach is proposed to handle disturbances for specific cost functionals, among others. However, we have not found any general approach concerning non-anticipative strategy approximations using neural networks, despite its strong connection to the definition of the value (specifically, the set $\mGT$ involved in the definition of the value 
in~\eqref{eq:v0}).
}

This work specifically investigates neural network approximations for deterministic differential games,
in the form of Eq.~\eqref{eq:v0}.  
The presented algorithms are demonstrated on a running cost optimal control problem, 
but the approach can be generalized to Bolza problems 
as discussed in \cite{alt-bok-zid-2013,ass-bok-des-zid-18}.
Our particular emphasis lies in achieving a rigorous error analysis for such approximations.


We propose two schemes: one is a "global" scheme, which directly attempts to approximate the desired value; 
the other is a "local" scheme, or time-stepping scheme, 
that initiates from the terminal value and proceeds backward until it approximates 
the value at the initial time. 
This approach is similar to the "Lagrangian" scheme in~\cite{bok-pro-war-23}
or in connection with the "Performance Iteration" scheme presented in \cite{hur-pha-21-a}.

We highlight two major differences compared to \cite{bok-pro-war-23}. 
Firstly, we provide new representation formulas for games, presented as minimax expectation values
over "feedback strategies" and feedback controls. 
One is associated with the approach of Elliot and Kalton, while the other appears to be novel.
Each representation formula naturally lends itself to approximations using neural networks and optimization algorithms
for the strategy, employing Stochastic Gradient Descent Ascent (SGDA) algorithms. 

Secondly, our convergence proof strategy is distinct and more straightforward.
It enables us to establish the convergence of the "global" scheme in neural network spaces in a weak sense.
In \cite{bok-pro-war-23}, the convergence was proved for a "local" time-marching scheme
following the dynamic programming principle.
However, within the game context, we encountered challenges in adapting the present proof strategy to establish the convergence of the "local" algorithm.
We leave this aspect for future developments.


The structure of the paper is outlined as follows. 
Section 2 provides the general setting and definitions. 
A semi-discrete problem is introduced, accompanied by various formulas 
utilizing different concepts of non-anticipative strategies. 
An error bound of order $O(\dt)$ is established between the continuous problem 
and the semi-discrete counterpart, particularly applicable to dynamics 
decomposed in the form $f(x,a,b)=f_1(x,a) + f_2(x,b)$, as observed in pursuit-evasion games.

In Section 3, different expectation formulas for the semi-discrete value are presented. 
This includes formulations employing the classical notion of non-anticipative strategies, 
as well as a modified version.

\MODIF{Section 4 presents two schemes based on gradient descent methods.}
Following this, section~5 provides error estimates and 
a convergence result for one of the algorithms.
\MODIF{These results hold for general approximation spaces, such as neural network approximations.}

Finally, Section 6 \MODIF{focusses on some elementary benchmark numerical tests, using neural network approximation spaces.}
These tests include examples in dimensions $d=2$, with analytic solutions provided for comparison purposes.
Additionally, a two-player game example in dimension $d=4$ is presented,
with a comparative analysis involving a finite difference scheme.


%
%

{\bf Notations.} Given any two sets $X$ and $Y$
we denote by $\cF(X,Y)$, or $Y^X$, the set of functions from $X$ to $Y$.
If $X$ and $Y$ are Borel sets (with $\ms$-algebra $\cB_X$ and $\cB_Y$ resp.), then we denote 
by $\Meas(X,Y)$ the set of measurable functions from $(X,\cB_X)$ to $(Y,\cB_Y)$.
Unless otherwise precised, $|.|$ is a norm on $\R^q$ ($q\geq 1$).
The notation $\llbracket p,q\rrbracket=\{p,p+1,\dots,q\}$ is used, for any integers $p\leq q$.
For any function $\ma:\R^p\conv \R^{q}$ for some $p,q\geq 1$, $[\ma]:=\sup_{y\neq x}\frac{|\ma(y)-\ma(x)|}{|y-x|}$ denotes
the corresponding Lipschitz constant. 
We also denote $a\vee b :=\max(a,b)$ for any $a,b\in \R$.


\section{Setting of the problem and first results}

\paragraph{Preliminary definitions.}
Let the following assumptions hold on the sets $A$, $B$ and functions $f,g,\varphi$:

\medskip

\noindent\textbf{\hypo}
{\em
$A$ and $B$ are non-empty compact subsets of $\R^{n_A}$ and $\R^{n_B}$ with $n_A,n_B\geq1$.
}

\medskip

\noindent {\textbf\hypf}
{\em
$f:\R^d\times A \times B \conv \R^d$ is Lipschitz continuous and we denote 
$[f]_{1},[f]_2,[f]_3 \geq 0$ constants such that
\beno
 & & |f(x,a,b)-f(x',a',b')| \ \leq \ [f]_1|x-x'| + [f]_2|a-a'| + [f]_3 |b-b'|, \\
 & & \hspace{5cm}
   \forall (x,x')\in (\R^d)^2,\ \forall (a,a') \in A^2,\ \forall(b,b')\in B^2.
\eeno
}

\vspace{-1ex}

\noindent{\textbf \hypG}
{\em
$g:\R^d\conv \R$ is Lipschitz continuous.
}

\medskip

\noindent\textbf{\hypPhi}{\ \em
$\varphi:\R^d\conv \R$ is Lipschitz continuous.
}

The following standard definitions can be found in the textbook \cite[Chap.\ VIII]{BarCap97}.
We consider $\cA_T:=\{a:(0,T)\conv A, \ measurable\}$ and 
similarly $\cB_T:=\{a:(0,T)\conv B, \ measurable\}$.
The set of non-anticipative (continuous) strategies, denoted $\mGT$,  is defined as the set of functions 
$\ma:\cB_T\conv\cA_T$, such that for all $t\in[0,T]$, ${b}_{\big| [0,t]}\equiv {\bar b}_{\big|[0,t]}$
$\implique$ ${\ma[b]}_{\big|[0,t]}\equiv {\ma[\bar b]}_{\big|[0,t]}$.
For given controls $a\in \cA_T$ and $b\in\cB_T$, 
we denote by $y_{0,x}^{a,b}$ the unique Carath{\'e}odory solution 
of $\dot y(s)=f(y(s),a(s),b(s))$, a.e. $s\in(0,T)$ and $y(0)=x$. 
The continuous value we consider is defined by 
\be\label{eq:v0}
  v_0(x) &=& \inf_{\ma \in \mGT} \sup_{b\in \cB_T} 
    \big( \max_{s\in(0,T)} g(y_{0,x}^{\ma[b],b}(s)) \big) \bigvee \varphi(y_{0,x}^{\ma[b],b}(T)).
\ee

The value $v_0(x)$ is also equal to the solution $v(t,x)$ at time $t=0$ 
of the following Hamilton-Jacobi-Bellman-Isaacs (HJBI) partial differential equation with an obstacle term,
in the viscosity sense (see for instance \cite{bok-for-zid-2010-2})
\begin{subequations}\label{eq:hjb-pde}
\be 
  & & \min\bigg(- v_t  - \max_{b\in B} \min_{a\in A} (\nabla_x v \cdot f(x,a,b)), \ v- g(x)\bigg) = 0, \quad t\in [0,T],\ x\in \R^d \nonumber \\
   & & \\ 
  & & v(T,x)=\max(\varphi(x),g(x)), \quad x\in \R^d.
\ee
\end{subequations}
Note that the "$\inf\sup$" in \eqref{eq:v0} is reverted into a "$\max\min$" in \eqref{eq:hjb-pde}.
This is classical for the HJI equation for games,
one may also look at Lemma~\ref{lem:a-b-inversion}$(ii)$ to understand this in a simplified context.
\MODIF{This is a two-player zero-sum game, in which} the first player with control $a$ aims to minimize the cost, 
while the second player with control $b$ aims to maximize the cost.
\MODIF{Here the cost functional involves a terminal cost $\varphi(.)$, and a maximum running cost with the function $g(.)$. 
Such a value is motivated by target problems, or backward reachability, with state constraints and under disturbances, see
Remark~\ref{rem:1}.}


\begin{rem}\label{rem:1}
By considering the functional \eqref{eq:v0}, assuming that $\cT=\{x,\varphi(x)\leq 0\}$ is the target and that 
$\cK:=\{x,\ g(x)\leq 0\}$ is the set of state constraints, then $v_0(x)\leq 0$ 
is equivalent to have, for any $\eps>0$, the existence of a non-anticipative strategy $\ma_\eps$ such that for any adverse control
$b\in \cB_T$, $y_{0,x}^{\ma_\eps[b],b}(T)\in \cT_\eps$ (we reach a neighborhood of the target set at final time) 
and $y_{0,x}^{\ma_\eps[b],b}(t)\in \cK_\eps$ for all $t\in [0,T]$ (we stay in a neighborhood of the set of state constraints),
where $\cT_\eps:=\{x,\ \varphi(x)\leq \eps\}$ and $\cK_\eps:=\{x,\ g(x)\leq \eps\}$.
Hence the value $v_0(x)$ is a level set function for the following "robust" backward reachable set under state constraints
\beno
   \scriptsize
   \mO:=\bigcap_{\eps>0}\{x\in\R^d,\ \exists \ma\in \mGT,\ \forall b\in \cB_T,\ y_{0,x}^{\ma[b],b}(T)\in \cT_\eps\ \mbox{and}\  
    y_{0,x}^{\ma[b],b}(t)\in \cK_\eps\ \forall t\in[0,T]\}
\eeno
  in the sense that $\{x,\ v_0(x)\leq 0\}\equiv \mO$. See also \cite{buc-car-qui-11} for this type of problems.
  The above approach allows to consider state constraints problems and avoid technical difficulties concerning the 
  boundary of the set of state
  constraints and the definition of the value.
  This can be also applied to more general Bolza problem for games and with state constraints as explained in 
  \cite{alt-bok-zid-2013}.
\end{rem}

Note that the methodology developed in this paper can be adapted to deal with other functional costs, 
such as the sum of a distributional cost with a terminal cost
\be\label{eq:v0-distributed}
    \int_0^T \ell(y_{0,x}^{\ma[b],b}(s),\ma[b](s),b(s)) ds +  \varphi(y_{0,x}^{\ma[b],b}(T)).
\ee

\paragraph{The semi-discrete problem.}
Following \cite[Chap.\ VIII]{BarCap97}, we introduce a semi-discrete problem, 
corresponding to a time discretization of the problem.
For a given $N\geq 1$, the set of discrete controls are $A^N$ and $B^N$.
The set of {\em discrete non-anticipative strategies}, denoted $S_N$, is
the set of measurable
functions $\ma:B^N\conv A^N$ such
that for any $0\leq k\leq N-1$:
$$
  \mbox{ $\big(\forall 0\leq j\leq k$, $b_j=\bar b_j\big)$ $\implique$ $\big(\forall 0\leq j\leq k$, $\ma[b]_k=\ma[\bar b]_k\big)$.}
$$
Hereafter we will allow both notations $\ma[b]$ or $\ma(b)$ when $\ma$ is a strategy.

\begin{rem}
  Notice at this stage that the mesurability condition in the strategies is not really necessary (in particular Theorem \ref{th:1} would hold with 
  strategies simply defined as functions : $B^N\conv A^N$). However later on all functions and strategies will be needed to be measurable because 
  we will apply them to a random variable and consider expectations.
\end{rem}

This also means that $\ma=(\ma_0,\dots,\ma_{N-1})$ where $\ma_k$ is only a function of $(b_0,\dots,b_k)$:
$$
 \forall k,\quad \ma_k=\ma_k(b_0,\dots,b_k).
$$  
For any controls $a=(a_0,a_1,\dots)\in A^N$ and $b=(b_0,b_1,\dots)\in B^N$, 
we define $(X_{k,x}^{a,b})_{k\geq 0}$ recursively 
by 
$$ \mbox{$X^{a,b}_{0,x}=x$ $\quad$ and $\quad$  $X_{k+1,x}^{a,b} = F(X_{k,x}^{a,b},a_k,b_k)$, $\forall k\geq 0$.}
$$
Here $F:\R^d\times A\times B\conv \R^d$ can be a one time-step approximation of the dynamics, and is assumed at least continuous in its variables.
The simplest example is to consider the Euler scheme with time step $\dt=T/N$:
\be\label{eq:euler} 
   F(x,a,b)=x + \dt f(x,a,b).
\ee
The semi-discrete value (corresponding to the approach of Elliott and Kalton~\cite{ell-kal-72})
is then defined by
%
%
\be\label{eq:V0}
  V_0(x) =  \inf_{\ma \in S_N} \sup_{b \in B^N} J_0(x,\ma[b],b)
\ee
for some given functional $J_0:\R^d\times A^N \times B^N$.
In the case of \eqref{eq:v0}, $J_0$ can be defined by
$$
  J_0(x,a,b):=\bigg(\max_{0\leq k\leq N-1} g(X_{k,x}^{a,b})\bigg) \bigvee \varphi(X_{N,x}^{a,b}),
$$
where $\varphi(.)$, $g(.)$ and $F(.)$ are given functions.
(In the case of \eqref{eq:v0-distributed}, we could consider
$\sum_{k=0}^{N-1} \dt \ell(X_{k,x}^{a,b},a_k,b_k) +  \varphi(X_{N,x}^{a,b})$.)
Notice that if $g=-\infty$ (or some large negative value), then
$J_0(x,a,b)= \varphi(X_{N,x}^{a,b})$.

More general Runge-Kutta schemes can be considered for the definition of $F$,
as well as multi-step approximations, as in \cite{bok-pro-war-23},
but in the present work this will not be necessary in order to obtain the convergence of the neural network schemes.

\medskip
The first objective is to recall some equivalent formulations for $V_0(x)$.

Following Fleming's approach \cite{Fleming-64}, corresponding to using piecewise constant controls,
let us also define the value $\barV_0(x)$ by 
\be\label{eq:V0-Fleming}
  \barV_0(x):=\max_{b_0} \min_{a_0}\ \ \cdots \ \ \max_{b_{N-1}}\min_{a_{N-1}} J_0(x,a,b)
\ee
for controls $a=(a_0,\dots,a_{N-1})\in A^N$ and 
controls $b=(b_0,\dots,b_{N-1})\in B^N$.
\MODIF{This approach also corresponds to a game where player $b$ plays first and player $a$ plays second, alternatively.}

We also introduce a third definition, which is
\be\label{eq:def-3}
  \bbarV_0(x):= \inf_{\ma\in \cG^N} \sup_{b\in B^N} \tJ_0(x,\ma[b],b)
\ee
where
$$
  \cG:=\Meas(\R^d\times B,A),
$$
(hence $\ma=(\ma_0,\dots,\ma_{N-1})\in \cG^N$ means that $\ma_k\in \cG$ for all $k$),
and the payoff function $\tJ_0$ is the same as the usual payoff function $J_0$ 
but using modified trajectories $\tX^{a,b}_{k,x}$ such that 
$\tX^{a,b}_{0,x} :=x$ and
$$ 
 \tX^{a,b}_{k+1,x} := F(\tX_{k,x}^{a,b}, \ma_k(\tX_{k,x}^{a,b},b_k), b_k), \quad k\geq 0.
$$
More precisely, whenever $\ma\in \cG^N$  and $b\in B^N$, 
we have 
\be
\label{eq:tildeJ0}
  \tJ_0(x,\ma[b],b): = \max_{0\leq k\leq N-1} g(\tx_k) \vee \varphi(\tx_N)
\ee
with $\tx_k=\tX_{k,x}^{\ma[b],b}$, now defined by 
\be\label{eq:tildeX}
  \mbox{$\tx_0=x$ and  $\tx_{k+1} = F(\tx_k,\ma_k(\tx_k,b_k),b_k)$, $k\geq 0$.}
\ee

Notice that $a_0=\ma_0(x,b_0)$, making explicit a possible dependence in $x$ in $\ma_0$.
This is not really necessary for our results at this stage in the sense that, for a given $x\in \R^d$
we could have also considered controls $\ma\in \Meas(B,A)\times \cG^{N-1}$ and, instead of \eqref{eq:def-3}:
\be\label{eq:def-3-bis}
  \bbarV_0(x):= \inf_{\ma\in \Meas(B,A)\times\cG^{N-1}} \sup_{b\in B^N} \tJ_0(x,\ma[b],b).
\ee

We shall still allow the notation $\ma(x,b)\equiv \ma[b](x)$, as well as 
$\ma_k(x,b_k) \equiv \ma_k[b_k](x)$.

Notice, in the definition of $V_0$ and of $\bbarV_0(x)$, that the "$\inf$" is a "$\min$".
Indeed, by using a measurable selection theorem \cite[Theorem 1]{bro-pur-73},
the measurability of $J_0(x,.,.)$ (resp. $\tJ_0(x,.,.)$) and the compactness of $A$ and $B$, 
there exists a minimum $\ma$ which is measurable
(see also the direct definition of optimal strategies $\bar\ma_k$ and $\ma_k^*$ 
in Theorem~\ref{th:2} below).

Let us first remark that the previous values are identical.
\begin{thm}\label{th:1}
  For all $x\in \R^d$, 
  $$ V_0(x)=\barV_0(x)= \bbarV_0(x).$$ 
\end{thm}
\begin{rem}\label{rem:th1}
  Note that the feedback controls used in $\bbarV_0$ are classical and natural, 
  but the fact that the value $\bbarV_0(x)$ exactly corresponds  to $V_0(x)$ seems less classical
  (we were not able to find a reference for this statement).
\end{rem}

We have also some information on the corresponding optimal strategies.
We first need to introduce, in the same way, for $x\in \R^d$, the intermediate values
$$
  V_k(x):=\max_{b_0} \min_{a_0}\ \ \cdots \ \ \max_{b_{N-k-1}}\min_{a_{N-k-1}}
  \bigg(\max_{j=0,\dots,N-k-1} g(X^{a,b}_{j,x})\bigg) \bigvee \varphi(X^{a,b}_{N-k,x}).
$$
The following dynamic programming relation is well known: for $0\leq k\leq N-1$ and $x\in \R^d$:
\be\label{eq:dpp}
  V_k(x):=\max_{b \in B}\min_{a\in A} \bigg(g(x) \bigvee V_{k+1}(F(x,a,b)) \bigg).
\ee
In view of Theorem~\ref{th:1} (the identity $V_0(x)=\barV_0(x)$),
the function $V_k(x)$ also satisfies
\be\label{eq:Vk-def1}
    V_k(x)=\inf_{\ma\in \Meas(B^{N-k},A)} \sup_{b\in B^{N-k}} J_k(x,\ma[b],b)
\ee
where $J_k(x,\ma[b],b):=\max_{k\leq j\leq N-1}g(x_k)\vee \varphi(x_{N-k})$, with $x_{k+1}=F(x_k, \ma_k[b_0,\dots,b_k], b_k)$
for $k\geq 0$ and $x_0=x$.
In the same way, as for the identity $\bbarV_0(x)=\barV_0(x)$ in Theorem~\ref{th:1}, we have
\be\label{eq:Vk-def3}
    V_k(x)=\inf_{\ma\in \Meas(\R^d\times B,A)^{N-k}} \sup_{b\in B^{N-k}} \tJ_k(x,\ma[b],b)
\ee
where
$\tJ_k(x,\ma[b],b):=\max_{k\leq j\leq N-1}g(\tx_k)\vee \varphi(\tx_{N-k})$, with $\tx_{k+1}=F(\tx_k, \ma_k[\tx_k,b_k], b_k)$
for $k\geq 0$ and $\tx_0=x$.
From these definitions we can deduce the following result.

\begin{thm} \label{th:2}
  $(i)$
  Let $x\in \R^d$. Any element 
   $\bar\ma=(\bar\ma_k)_{0\leq k\leq N-1}$ of $S_N$ such that, for all $k$,
  $$ 
   \bar\ma_k[b_0,\dots,b_k]\in \argmin_{a_k\in A} g(\bar x^b_k) \vee V_{k+1}(F(\bar x^b_k,a_k,b_k)),
      \quad \mbox{for a.e. $(b_i)_{0\leq i\leq k}\in B^{k+1}$}
  \vspace*{-0.2cm}
  $$
  where $\bar x^b_0:=x$ and for $0\leq i\leq k-1$:
  $\bar x^b_{i+1} = F(\bar x^b_i, \bar \ma_i[b_0,\dots,b_i], b_i)$,
  is an optimal non-anticipative strategy for $V_0(x)$
  in the sense that it reaches the infimum in \eqref{eq:Vk-def1}.
  \\
  $(ii)$ Any element $\ma^*=(\ma_k^*)_{0\leq k\leq N-1}$ of $\cG^N$ such that, for all $k$,
  $$ 
     \ma_k^*(x,b)\in \argmin_{a\in A} g(x) \vee V_{k+1}(F(x,a,b)),\quad 
     \mbox{for a.e.}\ (x,b)\in \R^d\times B
     \vspace*{-0.2cm}
  $$
  is an optimal non-anticipative strategy for $\bbarV_0(x)$, for a.e.\ $x\in \R^d$
  in the sense that it reaches the infimum in \eqref{eq:Vk-def3}.
\end{thm}
Notice that from $(i)$-$(ii)$, any optimal non-anticipative strategy $\ma^*\in \cG^N$ leads to 
an optimal non-anticipative strategy $\bar\ma\in \cS_N$ defined by
\be \label{eq:link-SN-GN}
  \bar \ma_k(x,b_0,\dots,b_k):=\ma_k^*(\bar x^b_k,b_k),\quad \forall k
\ee
where $\cS_N$ is now the set of measurable functions $\ma:\R^d\times B^N\conv A$ such that 
$\ma_k$ is only a function of $x$ and $b_0,\dots,b_k$:
$$  \ma_k(x,b_0,\dots,b_N) = \ma_k(x,b_0,\dots,b_k), \quad k=0,\dots,N-1, \ a.e.\ (x,b)\in \R^d\times B^N.
$$


Theorems~\ref{th:1} and~\ref{th:2} will be proved in Appendix~\ref{app:1}.
The equality $V_0=\barV_0$ is well known since the works of Elliott and Kalton~\cite{ell-kal-72},
but the other identity with $\bbarV_0$ seems classical but we were not able to find a reference for it.
It will be useful to define our numerical schemes. 
As a consequence, the non-anticipative strategy $a_k=\ma_k(b_0,\dots,b_k)$ (for a given starting point $x$),
can be also searched in the form
of $a_k=\tilde\ma_k(x_k,b_k)$ where $x_k$ corresponds to the position of the trajectory at time $t_k$.

  \begin{rem}
  The advantage of the formulation of $\bbarV_0$ is that it reduces the complexity 
  of the representation of the non-anticipative strategies:
  in each strategy for $V_0$, $\ma_k$ is a function from $\R^d\times B^k$ to $A$ 
 - if we make explicit the dependency over $x\in \R^d$, while for $\tV_0$
 it becomes only a function  from $\R^d\times B$ to $A$. 
  \end{rem}

We also give here a new error estimate between the semi-discrete value $V_0(x)$ 
and the value of the continuous problem $v_0(x)$.
When the value is Lipschitz continuous, as it is the case here,
the known error estimate is of order $O(\dt^{1/2})$ (see e.g. \cite{BarCap97}),
and the proof can be obtained by using viscosity arguments.
Adding a separability assumption on the dynamics, we can improve this result, for games, as follows.

\begin{thm}\label{th:error-v0}
  Assume the dynamics has a separate dependency in the controls, that is:
  \be\label{eq:f-separable} 
    f(x,a,b)=f_1(x,a) + f_2(x,b)
  \ee 
  for some Lipschitz continuous functions $f_1,f_2$. 
  Assume furthermore that $f_1(x,A)$ and $f_2(x,B)$ are convex for all $x\in \R^d$. 
  Consider the Euler scheme approximation $F(x,a,b)=x+\tau f(x,a,b)$.
  Then there exists a constant $C\geq 0$ such that, for all $x\in \R^d$: 
  $$ 
     |v_0(x)-V_0(x)|\leq C (1+|x|)\ \dt.
  $$
\end{thm}
Notice that for pursuit-evasion games, with dynamics of the form
$$f((y,z),a,b)=(g_1(y,a),g_2(z,b))$$ 
the relation \eqref{eq:f-separable} holds
true with $f_1(x,a)=(g_1(y,a),0)$ and $f_2(x,a)=(0,g_2(z,b))$.

Theorem~\eqref{th:error-v0} is proved in Appendix~\ref{app:2}. The proof is based on approximation of trajectories.
Note that for more complex Runge Kutta schemes $F$ and multi-step approximations, a similar error
estimate of order $O(\dt)$ would hold (see \cite{bok-pro-war-23}).


\section{Expectation formula for min-max problems}\label{sec:min-max}

\MODIF{From now on, the notation $\cA=\Meas(\R^d,A)$ (the set of measurable functions from $\R^d$ to $A$) will be used.
Later on, for a given compact set $B$ we will also use the notation
$\cB=\Meas(\R^d,B)$ as the set of measurable functions from $\R^d$ to $B$.}

Let us first recall the following Lemma 
that links pointwise minimization over open-loop controls $a\in A$  and minimization 
of an averaged value over feedback controls $a\in \cA$. 
This Lemma is the same as \cite[Lemma 3.2]{bok-pro-war-23} stated for continuous functions $Q$, but that is now stated 
for slightly more general measurable functions~$Q$.

From now on we consider a random variable $X$ on some probability space and consider the following assumptions.

\medskip

\noindent\textbf{\hypXa}
{\em
$X$ is a random variable with values in $\R^d$, and $\E[|X|]<\infty$.
}

\noindent\textbf{\hypXb}
$X$ admits a Lebesgue measurable density $\rho$ supported on $\bar\mO$ for some $\mO\subset\R^d$ 
and such that $\rho(x)>0$ a.e. $x\in \mO$, with $|\partial \mO|=0$ (i.e., $\partial\mO$ is negligible).

\medskip

\begin{lem} \label{lem:min}
  Assume \hypXa.
  Let a given measurable function $Q:\R^d\times A \conv \R$, with linear growth
  ($\exists C\geq 0$, $\forall (x,a)$, $Q(x,a)\leq C(1+|x|)$), and such that $a\in A\conv Q(x,a)$ is continuous for a.e. $x$.
  \\
  $(i)$ 
  \be\label{eq:min.eq.min.a} 
      \E\bigg[\inf_{a\in A} Q(X,a)\bigg]  
        = \inf_{a\in \cA} \E\bigg[ Q(X,a(X)) \bigg].
  \ee 
  Furthermore  
  an optimal $a^*\in \cA$ that minimizes 
  the quantity \eqref{eq:min.eq.min.a} exists.
  \\
  $(ii)$
  Assume furthermore \hypXb.
  Then
  \beno
    \ba(.) \in \argmin_{a\in \cA } \E\bigg[Q(X,a(X))\bigg]  
    \quad \Longleftrightarrow \quad
    \bigg( \ba(x) \in \argmin_{a\in A } Q(x,a),\ \mbox{a.e. $x\in \mO$} \bigg).
  \eeno 
\end{lem}
\begin{proof}
   A proof is given for sake of completeness.

  $(i)$ Let $I$ (resp. $J$) denote the left hand side (resp. r.h.s.) of \eqref{eq:min.eq.min.a}.
  For any $a\in \cA$, we have $Q(X,a(X))\geq \min_{a\in A} Q(X,a)$ and therefore $\E[Q(X,a(X)]\geq I$, hence $J\geq I$.
  Conversely, there exists $a^*\in \cA$ such that 
  $a^*(x) \in \argmin_{a\in A} Q(x,a)$ for a.e. $x\in \R^d$ 
  (by a measurable selection Theorem, see for instance Theorem~1 of~\cite{bro-pur-73}).
  Hence $J\leq \E\big[ Q(X,a^*(X)) \big] = I$, which concludes to $(i)$ and the existence of a minimizer $a^*\in \cA$.

  $(ii)$ For any minimizer $\bar a\in \cA$ of $J$, we have:
  $\E[Q(X,\bar a(X)) - \inf_{a\in A} Q(X,a)] = 0$. But the integrand is a.s. positive, hence we have
  $Q(X,\bar a(X)) =\inf_{a\in A} Q(X,a)$  a.s. Therefore $\bar a(X) \in \argmin_{a\in A}Q(X,a)$ a.s.,
  from which we can conclude, for a.e. $x$ in the support of $\rho$, that $\bar a(x) \in \argmin_{a\in A}Q(x,a)$
  (using assumption \hypXb).

\end{proof}

In the same way, in order to tackle the minimization problem such as
$$\inf_{a\in A}\sup_{b\in B} Q(x,a,b)$$
for a measurable function 
$Q:\R^d\times A\times B\conv \R$ (continuous in its variable $(a,b)$, for a.e. $x$),
we would like to proceed by considering $\inf$/$\sup$ of an averaged functional
using feedback controls $a\in \cA$ and $b\in \cB$.
Notice that because $Q$ is continuous in $(a,b)$ and $A$ and $B$ are compact sets, the 
infimum or supremum of $Q(x,a,b)$ with respect to $a$ and $b$ are reached and we can write $\min$/$\max$ without ambiguity.
The previous Lemma can be extended as follows.
\begin{lem} \label{lem:min-max}
  Assume $\hypXa$.
  Let a given measurable function $Q:\R^d\times A\times B\conv \R$, with linear growth
  ($\exists C\geq 0$, $\forall (x,a,b)$, $Q(x,a,b)\leq C(1+|x|)$),
  and such that $(a,b)\in A\times B\conv Q(x,a,b)$ is continuous for a.e. $x$.
  \\
  $(i)$
  \be\label{eq:min.eq.min.ab} 
    \E\bigg[\inf_{a\in A} \sup_{b\in B} Q(X,a,b)\bigg]  
      = \inf_{a\in \cA} \sup_{b\in \cB} \E\bigg[ Q(X,a(X),b(X)) \bigg].
  \ee
  and an optimal $a^*\in \cA$ that minimizes the quantity~\eqref{eq:min.eq.min.ab} exists.
  \\
  $(ii)$ Assume furthermore \hypXb. Let $\bar a\in \cA$. The following statements are equivalent:
  $$
    \ba(.) \in \arginf_{a\in \cA} \sup_{b\in \cB} \E\bigg[ Q(X,a(X),b(X)) \bigg] 
    \equiv \argmin_{a\in \cA} \E\bigg[\max_{b\in B} Q(X,a(X),b)\bigg]
  $$
  and
  $$
    \ba(x) \in \argmin_{a\in A} \max_{b\in B} Q(x,a,b) , \quad \mbox{a.e. }  x\in \mO.
  $$
\end{lem}
\begin{proof}
  By applying Lemma~\ref{lem:min} to the function $R(x,a)=\max_{b\in B} Q(x,a,b)$,  it holds
\begin{align}
  \E\big[\min_{a\in A} \max_{b\in B} Q(X,a,b)\big]  
  = \min_{a\in \cA} \E\big[ \max_{b\in B} Q(X,a(X),b)\big].
  \label{eq:ineqminmax}
\end{align}
Then, for any $a\in \cA$, and for any $b\in\cB$ the following inequality is immediate:
$$
  \E[ \max_{b\in B} Q(X,a(X),b)\big] \geq   \E\big[Q(X,a(X),b(X))\big],
$$
therefore also
\be\label{eq:ineq1}
  \E[ \max_{b\in B} Q(X,a(X),b)\big] \geq  \sup_{b\in\cB} \E\big[Q(X,a(X),b(X))\big].
\ee
Indeed equality holds in \eqref{eq:ineq1}: it suffices to choose $b^*\in\cB$ such that
$$
    b^*(x) \in \argmax_{b\in B} Q(x,a(x),b) \quad \mbox{a.e. } \ x\in \R^d.
$$
Hence using \eqref{eq:ineqminmax} and the equality in \eqref{eq:ineq1} leads to
$$
  \min_{a\in \cA} \E\big[ \max_{b\in B} Q(X,a(X),b)\big] = 
  \inf_{a\in \cA} \sup_{b\in \cB} \E\big[ Q(X,a(X),b(X))\big]
$$
which concludes to the desired result.

  The proof of $(ii)$ is deduce from Lemma~\ref{lem:min} $(ii)$ applied to the function $R$.
\end{proof}

We now give a formula in order to reverse the order $\inf_a \sup_b$ and $\sup_b\inf_a$.
In general the equality 
  $\inf_{a\in A} \sup_{b \in B} Q(a,b) = \sup_{b \in B} \inf_{a \in A} Q(a,b)$ is only true when $Q$ admits a saddle points $(a^*,b^*)$
  or more generally if for instance $Q$ is convex-concave (for all $b$, $a\conv Q(a,b)$ is convex and for all $a$, 
$b\conv Q(a,b)$ is concave) by Von Neumann's Theorem~\cite{neu-28}.

Recalling that $A^B:=\{f:B\conv A\}$, notice that for all $a$, 
\begin{align}
\label{eq:simple-eq}
    \sup_{b\in B} Q(a,b) = \sup_{b[.]\in A^B} Q(a,b[a]).
\end{align}

The following result is a particular case of Theorem 1.4.1 of \cite{Friedman-71}.
\begin{lem} \label{lem:a-b-inversion}
  Let $Q$ be a bounded function, and $A$, $B$ any two sets. 

  $(i)$ 
  We have
  \be\label{eq:lem-i}
    \inf_{a\in A} \sup_{b\in B} Q(a,b) = \sup_{b[.]\in A^B} \inf_{a\in A} Q(a,b[a]) 
  \ee
  and, in view of \eqref{eq:simple-eq}, this is also equal
  to $\inf\limits_{a\in A} \sup\limits_{b[.]\in A^B} Q(a,b[a])$.

  $(ii)$ In the same way, we have
  \be\label{eq:lem-ii}
    \sup_{b\in B} \inf_{a\in A} Q(a,b) = \inf_{a[.]\in B^A} \sup_{b\in B}  Q(a[b],b) 
  \ee
  and this is also equal to
    $\sup\limits_{b\in B} \inf\limits_{a[.]\in B^A} Q(a[b],b)$.

$(iii)$ The previous statements are still valid if $Q$ is a measurable function and 
    $A^B$ (resp $B^A$) is replaced by the space of measurable functions $\Meas(B,A)$ (resp. $\Meas(A,B)$).
    
\end{lem}

\begin{proof}
  We give a proof for self-completeness.
  We assume that $Q$ is continuous and $\inf$ and $\sup$ are reached to simplify the proofs, 
  but otherwise the proof can be obtained by approximation arguments.

  We focus on the proof of \eqref{eq:lem-i} (since \eqref{eq:lem-ii} is similar).
  Let 
  $\ma:= \inf_{a\in A}  \sup_{b\in B} Q(a,b)$
  and $\mb:=\sup_{b[.]} \inf_{a\in A} Q(a,b[a])$. 
  Obviously, $Q(a,b[a])\leq \sup_{b\in B} Q(a,b)$, for any $a$ and $b[.]$. Hence 
  $\inf_{a} Q(a,b[a]) \leq \inf_{a\in A} \sup_{b\in B} Q(a,b)$, and taking the supremum over $b[.]$, we deduce $\mb \leq \ma$.

  Conversely, for any $a\in A$, let $b[a]\in B$ such that $b[a]\in {\rm argmax}_{b\in B} Q(a,b)$, so that 
  $Q(a,b[a]) = \sup_{b\in B} Q(a,b)$.
  Then $\mb\geq \inf_{a\in A} Q(a,b[a]) = \inf_{a\in A} \sup_{b\in B} Q(a,b) = \ma$.
  Therefore we conclude to $\mb=\ma$. 

The proof of $(iii)$  is similar.
\end{proof}

Now we generalize the previous commutation results to the case of inf/sup over feedback controls and expectation formulas.

\begin{lem} \label{lem:min-max-bis}
  Assume $\hypXa$.
  Let a given measurable function $Q:\R^d\times A\times B\conv \R$, with linear growth
  ($\exists C\geq 0$, $\forall (x,a,b)$, $Q(x,a,b)\leq C(1+|x|)$),
  and such that $(a,b)\in A\times Q(x,a,b)$ is continuous a.e. $x$.
  It holds\\
  \be
      \label{eq:min.eq.min.ab-bis} 
    \E\bigg[\inf_{a\in A^B} \sup_{b\in B} Q(X,a[b],b)\bigg]  
      &  = &  \inf_{\ma\in \cG} \sup_{b\in \cB} \E\bigg[ Q(X,\ma(X,b(X)),b(X)) \bigg] \\
      &  = &  
     \label{eq:min.eq.min.ab-tri} 
     \E\bigg[\sup_{b\in B} \inf_{a\in A} Q(X,a,b)\bigg]  
   \\
      &  = & 
      \sup_{b\in \cB} \inf_{a\in \cA} \E\bigg[ Q(X,a(X),b(X)) \bigg].
     \label{eq:min.eq.min.ab-qua} 
  \ee
  where $A^B$ denotes the set of functions from $B$ to $A$, 
  and $\cG :=\Meas(\R^{d}\times B, A)$.
\end{lem}

\begin{proof}
 We just have to prove the first equality~\eqref{eq:min.eq.min.ab-bis}. Indeed, the equality between the l.h.s. of 
  \eqref{eq:min.eq.min.ab-bis} and \eqref{eq:min.eq.min.ab-tri} comes from Lemma~\ref{lem:a-b-inversion},
  and the equality between \eqref{eq:min.eq.min.ab-tri} and \eqref{eq:min.eq.min.ab-qua} comes from Lemma~\ref{lem:min-max}
  (reverting $\sup$ and $\inf$).
  Let $I$ (resp. $J$) denote the left- (resp. right-) hand side 
  of~\eqref{eq:min.eq.min.ab-bis}.

  For a.e. $x\in \R^d$ and for all $b\in B$, $\inf\limits_{a\in A} Q(x,a,b)$ is reached by some $a=\ma^*(x,b)$.
  Also, $\sup\limits_{b\in B} \big(\inf\limits_{a\in A} Q(x,a,b)\big)$ 
  is reached by some $b=b^*(x)$ (by using the compactness assumptions on $A$ and $B$).
  By using a measurable selection Theorem (see for instance Theorem~1 of~\cite{bro-pur-73}),
  we can find furthermore  $\ma^*\in \cG$ and then $b^*\in \cB$.
  By using Lemma~\ref{lem:min} (with "$\sup$" instead of "$\inf$"), we obtain
  \be \label{eq:J-identity}
    J & = &  \inf_{\ma \in \cG} \E \big[  \sup_{b\in B} Q(X,\ma(X,b),b) \big].
  \ee 
  Then, in particular,
  \beno
    J & \leq &   \E \big[  \sup_{b\in B} Q(X,\ma^*(X,b),b) \big] \ = \
     \E \big[  \sup_{b\in B}\inf_{a \in A} Q(X,a,b) \big]
  \eeno 
  (where we have used the definition of $\ma^*$ for the last identity).
  We also have $\sup\limits_{b\in B}\inf\limits_{a \in A} Q(x,a,b) = \inf\limits_{a \in A^B} \sup\limits_{b\in B} Q(x,a[b],b)$,
  for a.e.~$x\in \R^d$
  by Lemma~\ref{lem:a-b-inversion}. 
  Hence we deduce that $J\leq I$.

  Conversely, 
  for any $\ma\in \cG$,  
  $\E \big[  \sup\limits_{b\in B} Q(X,\ma(X,b),b) \big] \geq 
  \E \big[  \sup\limits_{b\in B} \inf\limits_{a\in A} Q(X,a,b) \big]$.
  Therefore
  $$
    J \geq  \E \big[  \sup_{b\in B} \inf_{a\in A} Q(X,a,b) \big].
  $$
  By using again  Lemma~\ref{lem:a-b-inversion}, we deduce that $J\geq I$.
\end{proof}

Now all the tools are in place to formulate a general expectation formula related to the definition \eqref{eq:def-3} for the
value $\tV_0$.
Recall that $\cG:=\Meas(\R^d\times B, A)$.

\begin{notation}
When $\ma\in \cG^N$ and $b\in \cB^N$, we denote 
$$ \tJ_0(x,\ma[b],b):= \max_{0\leq k\leq N-1} g(\tx_k) \vee \varphi(\tx_N),
$$
where $\tx_0=x$ and  $\tx_{k+1}=F(\tx_k,\ma_k(\tx_k,b_k(\tx_k)),b_k(\tx_k))$ for $k\geq 0$. 
Of course in the case when $b$ is constant (i.e., $b\in B^N$), we find the same definition as the previous
  function $\tJ_0$ used in \eqref{eq:def-3}.
\end{notation}

Recalling the definition of $\tV_0$ given by \eqref{eq:def-3}, 
by using similar arguments as in Lemma~\ref{lem:min-max-bis}, we obtain the following 
representation formula for $\tV_0$. 

\begin{thm}\label{th:minmax-formula} 
Let assumptions (H0)-(H4) be satisfied.

$(i)$
\be\label{eq:optim-def-3}
  \E\big[\tV_0(X)\big] =
    \E\big[ \inf_{\ma\in \cG^N} \sup_{b\in B^N} \tJ_0(X,\ma[b],b) \big]
   = \inf_{\ma\in \cG^N} \sup_{b\in \cB^N} \E\big[\tJ_0(X,\ma[b],b) \big].
\ee

$(ii)$
  Assume furthermore \hypXb, and that $\ma^*\in \cG^N$ in the r.h.s. of~\eqref{eq:optim-def-3} is optimal,
  then it is an optimal strategy for $\bbarV_0(x)$ 
  in the sense that
  $$
    \bbarV_0(x)= \sup_{b\in B^N} \tJ_0(x,\ma^*[b],b), \quad 
    \mbox{a.e. $x\in \mO$}
  $$
  \MODIF{(where we recall that $\mO$ is the support of $X$ by (H5)).}
\end{thm}

%


\section{Algorithms}\label{sec:algos}

We propose two algorithms. The first one follows the min-max formulation~\eqref{eq:optim-def-3}
in order to characterize the optimality of $\ma$. The second algorithm follows the dynamic programming principle
(in a similar way as in \cite{bok-pro-war-23} for control problems). 
The algorithms are well defined assuming that 
$X$ is a random variable on $\R^d$ with $\E[|X|]<\infty$.


\medskip
\noindent
\paragraph{{Algorithm 1} (Global scheme)}
Let $\hG$ (resp. $\hB$)
be approximation spaces for $\cG$  (resp. $\cB$).
Let $\eta=(\eta_1,\eta_2)$ be in $(\R_+^*)^2$ (margin errors), and let $X$ be some r.v. on $\R^d$.
\begin{quote}
- compute feedback strategy and control $(\hat\ma,\hb)\in \hG^N\times \hB^N$ according to
  \be
    \label{eq:G-scheme-a}
     (\hat\ma,\hb) \in \eta-\arg \inf_{\ma\in\hG^N} \sup_{b\in \hB^N}
      \E\bigg[  \tJ_0(X, \ma[b], b) \bigg]
  \ee
  (in a sense made precise below)
  \\
- set 
  \begin{align}\label{eq:G-scheme-b}
    \hat{V}_{0} (x) \coloneqq  \tJ_0(x,\hat\ma[\hb],\hb)
  \end{align}
\end{quote}

More precisely, the notation ``$\eta-\arg \inf\sup$'' in \eqref{eq:G-scheme-a} means, by convention, for some $\eta=(\eta_1,\eta_2)$,
that 
  \begin{subequations} 
  \be
    \sup_{b\in \hB^N} \E\big[\tJ_0(X,\hat\ma[b],b)\big] 
    \leq 
       \inf_{\ma\in \hG^N} \sup_{b\in \hB^N} \E\big[\tJ_0(X,\ma[b],b )\big]  + \eta_1
    \label{eq:argminmax-ab-1}
  \ee
    and
  \be
    \E\big[\tJ_0(X,\hat\ma[\hb],\hb)\big] 
    \geq
       \sup_{b\in \hB^N} \E\big[\tJ_0(X,\hat\ma[b],b)\big]  -\eta_2.
    \label{eq:argminmax-ab-2}
  \ee
  \end{subequations}
This allows for solving the min-max problem on $\hG^N\times \hB^N$ within some margin error, 
as this is the case in practice.
From a computational point of view, the min-max problem will be approximated by an adapted version of the 
 stochastic gradient descent-ascent algorithm (SGDA), see Section~\ref{sec:num}.


\medskip
\noindent
\paragraph{{Algorithm 2} (Local scheme)}
Let $\hat\cG$ (resp. $\hB$) be a given finite-dimensional space for the approximation of $\cG$  (resp. $\cB$).
Let $\eta_n=(\eta_{n,1},\eta_{n,2})$ be a sequence of positive numbers.
Set $\hat{V}_N \coloneqq g \vee \varphi$. 
For $n = N-1,\dots,0$:
\begin{quote}
- compute feedback strategy and controls $(\hat\ma_n,\hb_n)$ according to
  \be
    \label{eq:L-scheme-a}
     (\hat\ma_n,\hb_n) \in \eta_n-\arg \inf_{\ma\in\hG} \sup_{b\in \hB}
          \E\bigg[  g(X_n) \bigvee \hV_{n+1} (F(X_n,\ma(X_n,b(X_n))),b(X_n))  \bigg]
          \nonumber
          \\
  \ee
  (in a sense made precise below) 
  \\
- set 
  \begin{align}\label{eq:L-scheme-b}
    \hat{V}_{n} (x) \coloneqq g(x) \bigvee \hat{V}_{n+1}\left(F(x,\hat\ma_n(x,\hb_n(x)),\hb_n(x) )\right)
  \end{align}
\end{quote}
More precisely, let us denote $F^{\ma[b],b}(x):=F(x,a(x,b(x)),b(x))$, then ``$\eta_{n}-\arg\inf\sup$''
in \eqref{eq:L-scheme-a} means, by convention, for some $\eta_n=(\eta_{n,1},\eta_{n,2})$,
that 
  \begin{subequations} \label{eq:argminmax-ab-loc-1}
  \be
    & & \hspace*{-2cm} \sup_{b\in \hB} \E\big[g(X_n) \vee \hat{V}_{n+1}\left( F^{\hat\ma_n[b],b}(X_n) \right)\big]  \nonumber \\
    & & \leq\ \inf_{\ma\in \hG} \sup_{b\in \hB} \E\big[g(X_n) \vee \hat{V}_{n+1}\left( F^{\ma[b],b}(X_n) \right)\big]  + \eta_{n,1}
  \ee
  \be \label{eq:argminmax-ab-loc-2}
    & & \hspace*{-2cm}  \E\big[g(X_n) \vee \hat{V}_{n+1}\left( F^{\hat\ma_n[\hb_n],\hb_n}(X_n) \right)\big]   \nonumber\\
    & & \geq\ \sup_{b\in \hB} \E\big[g(X_n) \vee \hat{V}_{n+1}\left(  F^{\hat\ma_n[b],b}(X_n) \right)\big]  - \eta_{n,2}.
  \ee
  \end{subequations}
This allows for solving the min-max problem on $\hG\times\hB$ within some margin error.

In this algorithm, only the feedback strategies/controls $(\hat\ma_k,\hb_k)$ are stored ($\hat{V}_n$ is not stored).
Each evaluation of the value $\hat{V}_{n+1}(x)$ 
uses the previous strategies $(\hat\ma_{n+1},\dots,\hat\ma_{N-1})$ and controls
$(\hb_{n+1},\dots,\hb_{N-1})$ to compute the approximated characteristics, 
in a full Lagrangian philosophy. 

\section{Convergence analysis}

In this section an error estimate for algorithm~1 is given.
For algorithm~2, because we do not have $\hV_0(x)\geq V_0(x)$ in general,
we were not able to obtain an error estimate as in the setting of  \cite{bok-pro-war-23}.
Here, we will only give error estimates for the difference $\E[\hV_0(X)]-\E[V_0(X)]$, which we call a "weak" error estimate.



First we state an approximation result of the exact value by Lipschitz continuous strategies and
feedback controls.
From now on, for given constants $L,M\geq 0$, let us denote
$$ 
  \cG_L:=\{\ma\in \cG,\ [\ma]\leq L\}, \quad 
  \cB_M:=\{b\in \cB,\ [b]\leq M\}.
$$
In the following Lemmata we assume \hypo-\hypXa.

\begin{lem}\label{lem:v0-bounds}
  $(i)$
  Assume that the boundary of $B$ is of null Lebesgue's measure.
  Let $\eps_1>0$. Then there exist $L\geq 0$ and 
  $\ma^*\in (\cG_L)^N$  such that
  \be
    & & \E[V_0(X)] \geq  \sup_{b\in \cB^N} \E\big[\tJ_0(X,\ma^*[b],b)\big]  -\eps_1.
    \label{eq:lip-eps-minimax-1}
  \ee
  $(ii)$ Let $\hat\ma$ be a given strategy of $(\cG_L)^N$.
  Let $\eps_2>0$. Then there exist $M\geq 0$ and 
  $b^*\in (\cB_M)^N$  such that
  \be
    & & \E[V_0(X)] \leq \E\big[\tJ_0(X,\hat\ma[b^*],b^*)\big]  +\eps_2.
    \label{eq:lip-eps-minimax-2}
  \ee
\end{lem}
Notice that \eqref{eq:lip-eps-minimax-1} can be written, in the same way, in the form of 
\beno
    & & \E[V_0(X)] \geq  \inf_{\ma\in (\cG_L)^N}
        \sup_{b\in \cB^N} \E\big[\tJ_0(X,\ma[b],b)\big]  -\eps_1,
\eeno
which means an approximation of $\E[V_0(X)]$ by using Lipschitz continuous strategies. 
\begin{proof}
  $(i)$
It is in principle possible to follow the regularization approach of HJB-Isaacs equation with Lipschitz continuous controls as 
  in \cite{Barron_Evans_Jensen_84}. 
We can also obtain the approximation by a mollifying argument as follows.
  Let $(\ma,b)\in \cG\times \cB$.
  We first consider $\ma^\eps:=\rho_\eps * \ma_k$, a mollifying sequence for $\ma_k$ ($\ma_k \in \Meas(\R^d\times B,A)$
  can be extended on $\R^d\times \R^{n_B}$ by $\ma_k(x,b)=0$ whenever $b\notin B$.)
  Therefore $\lim_{\eps\conv 0}\ma_k^{\eps}(x,b)=\ma_k(x,b)$ a.e. on $\R^d\times int(B)$, for all $k$. By using the assumption that $\partial B$ is 
  of null Lebesgue's measure ($\lambda(\partial B)=0$),
  we get the a.e. convergence on $\R^d\times B$.
  In view of the expression of $J_0(x,a,b)$, which is continuous in the variable $a_{N-1}$, $\tJ_0(x,\ma[b],b)$ will be continuous
  in its dependence on the variable $\ma_{N-1}$.
  Hence by Lebesgue's dominated convergence Theorem,
  we first have $\lim_{\eps_{N-1}\conv 0} \E[\tJ_0(X,\ma^\eps[b],b)] =  \E[\tJ_0(X,\ma[b],b)]$, where 
  $\ma^\eps=(\ma_0,\dots,\ma_{N-2},\ma^\eps_{N-1})$, i.e., we only regularize the $\ma_{N-1}$ part.
  Then we can regularize in the same way $\ma_{N-2}$ by some $\ma^{\eps_{N-2}}_{N-2}$ for $\eps_{N-2}$ small enough. 
  We proceed by recursion until $\ma_0^{\eps_0}$. The regularized functions $\ma_{k}^\eps$ are then Lipschitz continuous on $\R^d\times B$.
  This gives the desired result.

  $(ii)$ we can proceed in the same way.
\end{proof}

  For given $\eps_1,\eps_2>0$, $\ma^*$ and $\mb^*$ are chosen as in the previous Lemma.
  Let $L:=[\ma^*]$ and $M:=[b^*]$, so that $\ma^*\in (\cG_L)^N$ and $b^*\in(\cB_M)^N$.


\begin{lem}\label{lem:gronwall-1}
  $(i)$ 
  Let $\ma$ and $\bar\ma$ be two elements of $(\cG_L)^N$,
  and let $b\in B^N$.
  Let $x,y$ be in $\R^d$ and assume that for all $0\leq k\leq N$, $X^{\ma[b],b}_{k,x} \in \mO_N$, 
  where $\mO_N$ is some given subset of $\R^d$.
  Then it holds
  $$
  \max_{0\leq k\leq N} |X^{\bar\ma[b],b}_{k,y}-X^{\ma[b],b}_{k,x}|
    \leq e^{C_{1,L} T} 
     \bigg(|y-x| + C_2 T \max_{0\leq k\leq N-1} \| \bar\ma_k - \ma_k\|_{L^\infty(\mO_N\times B)}  \bigg)
  $$ 
  where $C_{1,L}:=[f]_1+[f]_2 L$ and $C_2:=[f]_2$.
  \\
  $(ii)$ 
  Let $\ma$ be in $(\cG_L)^N$, and let $b,\bb$ be in $(\cB_M)^N$ for some $M\geq 0$.
  Let $x,y$ be in $\R^d$ and assume that for all $0\leq k\leq N$, $X^{\ma[b],b}_{k,x} \in \mO_N$.
  Then it holds
  $$
  \max_{0\leq k\leq N} |X^{\ma[\bb],\bb}_{k,y}-X^{\ma[b],b}_{k,x}|
    \leq e^{C_{3,L,M} T} 
     \bigg(|y-x| + C_{4,L} T \max_{0\leq k\leq N-1} \| \bb_k - b_k\|_{L^\infty(\mO_N)}  \bigg)
  $$ 
  with $C_{3,L,M}:=[f]_1 + ([f]_2 L+[f]_3)M$ 
  and $C_{4,L}:=L [f]_2 + [f]_3$.
\end{lem}
\begin{proof}
  $(i)$
  This is a discrete Gronwall estimate. 
  Let $x_k=X_{k,x}^{\ma[b],b}$ and $y_k=X_{k,y}^{\bar \ma[b],b}$ for $k\geq 0$.
  We have $x_{k+1}=x_k + \dt f(x_k,\ma_k(x_k,b_k),b_k)$ and 
  $y_{k+1}=y_k + \dt f(y_k,\bar\ma_k(y_k,b_k),b_k)$. Therefore
  \beno 
    & & \hspace{-1cm} \big| y_{k+1}- x_{k+1} \big| \\
    & & \leq
      |y_k-x_k| (1 + \dt ([f]_1 + [f]_2 [\ma_k])) +  \dt [f]_2 | \bar\ma_k(x_k,b_k)- \ma_k(x_k,b_k)|.
  \eeno
  with $[f]_1 + [f]_2 [\ma_k]\leq [f]_1 + [f]_2 L = C_1$. 
  Then by induction:
  $$
  |X^{\bar\ma[b],b}_{n,y}-X^{\ma[b],b}_{n,x}|
    \leq (1+C_1 \dt)^n 
     \bigg(|y-x| + C_2 \dt \sum_{k=0,\dots,n-1} \bigg|\bar\ma_k(x_k,b_k) - \ma_k(x_k,b_k)\bigg| \bigg)
  $$ 
  where $C_1:=[f]_1+[f]_2 L$ and $C_2:=[f]_2$. 
  Notice that $(1+C_1\dt)^n\leq e^{C_1T}$ for $t_n=n\dt\leq T$. The desired result follows.
  \\
  $(ii)$ 
  Let $x_k=X_{k,x}^{\ma[b],b}$ and $y_k=X_{k,y}^{\ma[\bb],\bb}$ for $k\geq 0$, we have now
  \beno
   & &\mbox{$x_{k+1}=x_k + \dt f(x_k,\ma_k(x_k,b_k(x_k)),b_k(x_k))$}
  \\
  & &  \mbox{$y_{k+1}=y_k + \dt f(y_k,\ma_k(y_k,\bb_k(y_k)),\bb_k(y_k))$.}
  \eeno
  Therefore
  \beno 
    & & \hspace{-1cm} \big| y_{k+1}- x_{k+1} \big| \\
    & & \leq
      |y_k-x_k| (1 + \dt ([f]_1 + [f]_2 [\ma_k]) [\bb_k]) +  \dt ([f]_2[\ma_k] +[f]_3)|\bb_k(x_k)- b_k(x_k)|.
  \eeno
  with $[f]_1 + [f]_2 [\ma_k] [\bb_k]\leq [f]_1 + [f]_2 L M =: C_{3,L,M}$
  and  $[f]_2[\ma_k] +[f]_3\leq     [f]_2 L +[f]_3 =: C_{4,L}$.
  We conclude then as in $(i)$.
\end{proof}

\begin{thm}[\bf error estimate]
  \label{th:err}
  Assume \hypo-\hypPhi, and that $X$ is a random variable with compact support denoted $\mO_0$ and such that 
  $\E[|X|]<\infty$.
Consider the strategy and controls obtained by algorithm~1, and assume that $\ha\in(\hG_L)^N$ and $\hb\in (\hB_M)^N$,
where $L,M$ are constants large enough in order that the estimates of Lemma~\ref{lem:v0-bounds} hold.
  Let $\ma^*$ and $\mb^*$ be such that  \eqref{eq:lip-eps-minimax-1} and \eqref{eq:lip-eps-minimax-2} hold.
Then there exist positive constants $C_L$ and $C_{L,M}$ 
  \MODIF{(that depends only of $L$  - respectively $L$ and $M$ - and of the data)}
  such that 
  \be\label{eq:thm-err-1}
    \E[\hV_0(X)] - \E[V_0(X)]  \leq C_L 
    \max_{0\leq k\leq N-1} d_{L^\infty(\mO_N\times B )}(\ma_k^*,\hG_L) + \eta_1 +\eps_1 
  \ee
  and
  \be\label{eq:thm-err-2}
    \E[\hV_0(X)] - \E[V_0(X)]  \geq - C_{L,M}
    \max_{0\leq k\leq N-1} d_{L^\infty(\mO_N)}(b_k^*,\hB_M) - \eta_2 -\eps_2 
  \ee
  where 
  \be \label{eq:OmegaN}
    \mO_N:=\{X_{k,x}^{a,b},\ x\in\mO_0,\ 0\leq k\leq N,\ (a,b)\in A^N\times B^N\}.
  \ee
\end{thm}
\begin{rem}
  Notice that by standard Gronwall estimates, if $\mO_0\subset B(0,r_0)$, and if we denote $C,L$ two constants such that 
  $\|f(x,a,b)\|\leq C + L\|x\|$ (i.e. $L=[f]_1$ and $C=\max_{A\times B} \|f(0,a,b)\|$),
  then $\mO_N\subset B(0,r_N)$ where $r_N:=e^{LT}(r_0+ CT)$.
  
\end{rem}
\begin{proof}[Proof of Theorem~\ref{th:err}]
We first consider the upper bound.
By the scheme definition, using \eqref{eq:argminmax-ab-1} (i.e., the $\eta_1$-suboptimality of $\hat\ma$),
and by the definition of $V_0=\tV_0$ and the $\eps_1$-suboptimality of $\ma^*$ with respect to $\tV_0$:
\be
  \E[\hV_0(X)] - \E[V_0(X)] 
    & \leq & 
       \inf_{\ma\in \hG^N} \sup_{b\in \hB^N} \E\bigg[  \tJ_0(X, \ma[b], b) \bigg] + \eta_1
        - \sup_{b\in \cB^N} \E\bigg[  \tJ_0(X, \ma^*[b], b) \bigg] + \eps_1 \nonumber \\
    & \leq & 
       \inf_{\ma\in \hG^N} \sup_{b\in \cB^N} \E\bigg[  \tJ_0(X, \ma[b], b) \bigg] 
        - \sup_{b\in \cB^N} \E\bigg[ \tJ_0(X, \ma^*[b], b) \bigg] + \eta_1 +\eps_1 \nonumber
\ee
where we have used the fact that $\hB\subset \cB$ in the last inequality.
Therefore we have
\be
  \E[\hV_0(X)] - \E[V_0(X)] 
    & \leq & 
      \inf_{\ma\in \hG^N} \sup_{b\in \cB^N} \E\big[  \tJ_0(X, \ma[b], b) - \tJ_0(X, \ma^*[b], b) \big]  
        + \eta_1 +\eps_1 \nonumber \\
    &  \leq   &
      \inf_{\ma\in \hG^N_L} \E\bigg[\sup_{b\in B^N} (\tJ_0(X, \ma[b], b) - \tJ_0(X, \ma^*[b], b)) \bigg]  
        + \eta_1 +\eps_1 \nonumber
\ee
where we have furthermore restricted the space $\hG^N$ to $\hG^N_L$.
Recall that $\tJ_0(x,\ma[b],b)=\max_{0\leq k\leq N-1} g(x_k) \vee \varphi(x_N)$ where  $x_k=X_{k,x}^{\ma[b],b}$,
which also corresponds to $x_0=x$ and $x_{k+1}=F(x_k,\ma_k(x_k,b_k),b_k)$.
Denoting $y_k=X_{k,x}^{\ma^*[b],b}$, we have 
\beno
  |\tJ_0(x,\ma^*[b],b) - \tJ_0(x,\ma[b],b)| 
  & \leq & \max_{0\leq k\leq N-1} |g(y_k)-g(x_k)| \vee |\varphi(y_N)-\varphi(x_N)| \\
  & \leq & \max([g],[\varphi]) \max_{0\leq k\leq N} |y_k-x_k|.
\eeno 
By using the estimate of Lemma~\ref{lem:gronwall-1}$(i)$, we obtain
$$ \max_{0\leq k\leq N} |y_k-x_k| \leq 
      e^{C_1 T}  C_2 T \max_{0\leq k\leq N-1} \| \ma_k - \ma_k^* \|_{L^\infty(\mO_N\times B )}
 $$
  and therefore with $C_L:=\max([g],[\varphi]) C_2 T e^{C_1 T}$ (which only depends of $L$ and of the data)
\be
   & & \hspace{-2cm} \E[\hV_0(X)] - \E[V_0(X)] \nonumber \\
    &  & \leq
       C_L
       \E\big[ 1_{\mO_N} \big] 
       \inf_{\ma\in \hG^N_L} \max_{0\leq k\leq N-1} \| \ma_k - \bar\ma_k \|_{L^\infty(\mO_N\times B )}
       + \eta_1 +\eps_1 .
\ee
Hence by using the fact that  $\E\big[ 1_{\mO_N} \big] \leq 1$,
we conclude to 
\beno
  \E[\hV_0(X)] - \E[V_0(X)]  
  &  & \leq
  C_L
  \bigg(\max_{0\leq k\leq N-1}
  \inf_{\ma\in \hG_L} \| \ma_k - \ma_k^* \|_{L^\infty(\mO_N\times B )}
  \bigg) + \eta_1 +\eps_1.
\eeno
Since $\inf_{\ma\in \hG_L} \| \ma - \ma_k^* \|_{L^\infty(\mO_N\times B )}= d(\ma_k^*,\hG_L)$,
this concludes the upper bound estimate.

For the lower bound,
by the scheme definition, using \eqref{eq:argminmax-ab-2} (i.e., the $\eta_2$-suboptimality of $\hat\ma,\hb$),
and by using Lemma~\ref{lem:v0-bounds} and the $\eps_2$-suboptimality of some $b^*\in(\cB_M)^N$ (for some $M\geq 0$) with respect to $\tV_0$ as
in \eqref{eq:lip-eps-minimax-2},
we obtain
\beno
 \E[\hV_0(X)] - \E[V_0(X)] 
   & \geq & 
        \sup_{b\in \hB^N} \E\bigg[  \tJ_0(X, \hat\ma[b], b) \bigg] - \eta_2
        - \bigg( \E\bigg[  \tJ_0(X, \hat\ma[b^*], b^*) \bigg] - \eps_2\bigg) .
       \nonumber
\eeno
Hence, using also the fact that $(\hB_M)^N \subset \hB^N$,
\beno
  \E[\hV_0(X)] - \E[V_0(X)] 
     & \geq &
        - \inf_{b\in (\hB_M)^N} 
        \E\bigg[ \big|\tJ_0(X, \hat\ma[b^*], b^*) - \tJ_0(X, \hat\ma[b], b) \big|\bigg] - \eta_2 - \eps_2.
\eeno
By using the estimate of Lemma~\ref{lem:gronwall-1}$(ii)$, we obtain
$$
  |\tJ_0(x,\hat\ma[b^*],b^*) - \tJ_0(x,\hat\ma[b],b)|
   \leq \max([g],[\varphi]) 
   C_{4,L} T e^{C_{3,L,M} T} 
   \max_{0\leq k\leq N-1} \| b_k^* - b_k\|_{L^\infty(\mO_N)}.
  $$ 
We conclude to the desired estimate, with $C_{L,M}:= \max([g],[\varphi]) C_{4,L} T e^{C_{3,L,M} T}$.
\end{proof}

From the previous estimates, we deduce the following convergence result.

\begin{thm}[\bf convergence]
  Assume \hypo-\hypPhi, and that $X$ is a random variable with compact support and $\E[|X|]<\infty$.
  Let $N\geq 1$ be fixed. Let $\mT$ be the set of parameters of the approximation spaces $\hG_L,\hB_M$.
  Let us assume that for all compact $K$, and fixed $L,M$,
  $$ \forall \ma\in \cG_L,\quad \lim_{\mT\conv\infty} d_{L^\infty(K\times B)}(\ma,\hG_L) = 0,
  $$
  and also, in the same way,
  $$ \forall b\in \cB_M,\quad \lim_{\mT\conv\infty} d_{L^\infty(K)}(b,\hB_M) = 0.
  $$
  Then for any $\eps>0$, there exist constants $L,M$ large enough, constants $(\eta_i,\eps_i)_{i=1,2}$ small enough,
  such that, for $\mT$ large enough,
  algorithm 1 gives
  \be
       \big| \E[\hV_0(X)-V_0(X)] \big| \leq \eps.
  \ee
\end{thm}

In conclusion, in the above weak sense, we can say that algorithm~1 converges to the value $\E[V_0(X)]$.


%% file: sec-num.tex

\newcommand{\figw}{}
\newcommand{\rdir}{}
\newcommand{\name}{}
\newcommand{\bx}{\overline{x}}{}
\newcommand{\xbmin}{{\underline x}}
\newcommand{\xbmax}{{\overline  x}}
\newcommand{\ybmin}{{\underline y}}
\newcommand{\ybmax}{{\overline  y}}
\newcommand{\epsbmin}{{\underline \epsilon}}
\newcommand{\epsbmax}{{\overline  \epsilon}}

\newcommand{\MPOTE}{M_{epoch}^{pote}}
\newcommand{\FIGURE}[2]{#2} 

\section{Numerical examples} \label{sec:num}

\medskip\noindent
{\bf Feedforward neural networks.}\label{sec:num-nn}
In our approximations we use Feedforward neural networks for the approximation of the feedback control strategies and for the adversarial controls. 
We denote by
\begin{align*}
\cL_{d_1,d_2}^{\rho} &=\;  \Big\{ \phi : \R^{d_1} \rightarrow \R^{d_2}: \exists \;  (\cW,\beta) \in   \R^{d_2\times d_1} \times \R^{d_2}, \; 
\phi(x) \; = \: \rho( \cW x + \beta) \; \Big\}
\end{align*}
the set of layer  functions with input dimension $d_1$, output dimension $d_2$, and activation function $\rho$ $:$ $\R^{d_2}$ $\rightarrow$ $\R^{d_2}$.
The operator $x$ $\in$ $\R^{d_1}$ $\mapsto$ $\cW x + \beta$ $\in$ $\R^{d_2}$  is an affine mapping
with  $\cW$  a matrix called weight, and  $\beta$ a vector called bias.  
The activation is applied component-wise i.e.,  
$\rho_{}(x_1,\ldots,x_{d_2})$ $=$ $\big(\hat \rho(x_1),\ldots,\hat \rho(x_{d_2})\big)$ with $\hat \rho : \R \mapsto \R$ non decreasing.    
When $\rho_{}$ is the identity function, we simply write $\cL_{d_1,d_2}$ and when $\rho(x) = \max(x,0)$ we write 
$\cL_{d_1,d_2}^{ReLu}$.

Each control with values  in $\R^{d_1}$ is approximated in the space of 
neural networks with $L$ hidden layers of $m$ neurons using the ReLu activation function for internal activation:
\begin{align*}
  \cN_{d_0,d_1,L,m} 
     &= \;  \Big\{ \varphi : \R^{d_0}  \rightarrow {\R^{d_1}}: \exists   \phi_0 \in  \cL^{ReLu}_{d_0,m},
        \; \exists \phi_i \in \cL^{ReLu}_{m,m}, i=1,\ldots,{L-1}, \; \\
     & \hspace{4cm} \exists  \phi_L \in \cL_{m,d_1},  \varphi  \; = \; 
     \xi\;\circ\; \phi_L \circ \phi_{L -1} \circ \cdots \circ \phi_0 \Big\}
\end{align*}
and $\xi:\R^{d_1}\conv \R^{d_1}$ is a final activation function.
When for instance $a=a(x,b)$ and $b=b(x)$ then 
$a$ is approximated in $\cN_{d+n_B,n_A,L,m}$ and $b$ in  $\cN_{d,n_B,L,m}$  assuming $A\subset \R^{n_A}$ and $B\subset \R^{n_B}$.
Depending on the test case,
since $A$ and $B$ are compact sets, 
a final activation function $\xi_A$ mapping  $\R^{n_A}$ to $A$ or $\xi_B$ mapping $\R^{n_B}$ to $B$ is used.
In particular for $A=B=[-1,1]$ (as in the first three examples), the function $\xi_A(x)=\xi_B(x)=\tanh(x)$ for $x\in \R$ is used.
In the last example, $A=B=B_2(0,1)$ is the unit ball of $\R^2$ for the Euclidean norm,
then we will used $\xi_A(x)=\xi_B(x)=\frac{x}{\|x\|_2} \tanh(\|x\|_2)$  for $x\in\R^2$ (following \cite{bok-pro-war-23}).
The previous set $\cN_{d_0,d_1,L,m}$  is parametrized by $\theta$ $=$ $(\cW_0,\beta_0,\ldots,\cW_{L},\beta_{L})$ defining the layer functions; a function $\phi$ in this set is denoted $\phi^\theta(.)$ to emphasize on the $\theta$ dependency.

\MODIF{
A fundamental result of Hornick et al.~\cite{HSW89} justifies the use of neural networks as function approximators
(this is also known as a universal approximation theorem). 
In particular, 
if $\rho$ is a non constant $C^k$ function, then any  
any function and its derivatives up to order $k$ can be approximated by
$\bigcup_{ m \in \N} \cN^\rho_{d_0,d_1,\ell,m}$ on any compact set of $\R^{d_0}$ with arbitrary precision.
}

\medskip\noindent
{\bf Min-Max optimization.}\label{sec:num-minmax}
Our typical problem is to deal with 
$$
  \min\limits_{\theta_A}\max\limits_{\theta_B}\E
 \bigg[ J(Z,A^{\theta_A}(Z),B^{\theta_B}(Z))\bigg]
$$
where $Z$ is a random variable in $\R^d$ 
over a set of parameters $\theta_A$ and $\theta_B$ where  $A^{\theta_A} \in \cN_{d,n_A,L,m}$, $B^{\theta_B} \in \cN_{d,n_B,L,m}$ 
and for a given cost functonal $J$. 
Stochastic gradient methods are a classicaly used to deal with general (possibly non convex/non concave) problems of the form
\begin{align*}
  \min_{x} \max_{y}\E\big[ Q(x,y,Z)\big]
\end{align*}
where $Z$ is a random variable. At each step $i$ of a stochastic gradient algorithm, we consider $N_{batch}\in\N^*$ and  
i.i.d. $z^i=(z^i_q)_{1\leq q\leq N_{batch}}$ with same law as $Z$ ($z^i_q\sim Z$), and 
\begin{align*}
  f(x,y,(z^i_q)) := \frac{1}{N_{batch}} \sum_{q=1}^{N_{batch}} Q(x,y,z^i_q).
\end{align*}
Most algorithms such as {SGDA} \cite{lei2021stability}, 
{AGDA} \cite{lei2021stability}, $\gamma$-GDA  of~\cite{jin2020local} 
extended to the stochastic case are designed for convex-concave problems and may behave badly on our problems.
Since we deal with general functional we consider here {\em Potential reduction algorithms} \cite{razaviyayn2020nonconvex}, 
\cite{jin2020local} extended to the stochastic case. 
The following iteration, with $(\eta_i)_{i \ge 0}$ a sequence of positive step size where $\eta_0=\eta$, 
gives the general outline of an iterative resolution algorithm:
    \begin{align*}
    y^{i+1} = & \argmax_{ y} f(x^i,y, z^i) \\ 
    x^{i+1} = & x^i -  \eta_i \nabla_x f(x^i,y^{i+1},z^i).
    \end{align*}
In order to get a feasible algorithm, we introduce integers $M_{epoch},\MPOTE\geq 1$ 
and positive sequences $(\eta_i)_{i\geq 0}$, $(\rho_{i,k})_{i,k\geq 0}$ with $\eta_0:=\eta$ and 
$\rho_{i,0}:=\rho$.
The resulting algorithm is as follows.



\medskip

\begin{small}
\noindent
  \underline{{\bf Potential reduction algorithm} for $\min_{x}\max_y \E[Q(x,y,Z)]$:}\\[-0.2cm]
\begin{algorithm}[H]
  Start with randomly chosen set of parameters $x^0,y^0$.\\
  \For{$i=0, \dots M_{epoch}-1$}{
    $y \gets y^i$; $\rho_{i,0} \gets \rho$\\
    \For{$k=0, \dots \MPOTE-1$}{
      $y \gets y + \rho_{i,k} \nabla_y f(x^i,y,(w^{i,k}_q))$ with $w^{i,k}_q \sim Z$
    }
    $y^{i+1} \gets y$\\
    $x^{i+1} \gets x^i -  \eta_i \nabla_x f(x^i,y,(z^i_q))$ with $z^{i}_q \sim Z$
  }
\end{algorithm}
\end{small}

\MODIF{
\begin{rem}
As we use neural networks, even in the convex-concave case, the resulting optimization problem is not convex-concave anymore and there is no guarantee that the stochastic gradient algorithm converges to the optimum. 
However the representation formula and proposed algorithms could potentialy be used with
other approximation spaces and optimisation methods, when available.
\end{rem}
}

In the case of a $\min\limits_{a[\cdot]}\max\limits_{b} Q(a[b],b)$ problem,  
$\MPOTE$ is therefore the number of internal iterations optimizing  the  parameters of $b$,
and $M_{epoch}$ is the number of external iterations.
In practice, each optimization step is achieved using the ADAM optimizer \cite{kingma2014adam}
using an adaptive learning rates derived from estimates of first and second moments of the gradients. 

In all examples the computational domain is a parallelipedic box $\mO\subset \R^d$ (depending on the example),
$Z$ is the uniform random variable on $\mO$ and therefore 
the batch points are drawn uniformly in $\mO$.

\medskip
\paragraph{Multi-step approximations.}
Finally, for the approximation of the dynamics we consider a multi-step approximation $F$ instead of the
Euler scheme~\eqref{eq:euler}.
Let $p\geq 1$ be a given integer (the number of intermediate sub-steps),
for given $x\in\R^d$ and controls $(a,b)\in A\times B$, we first define a Runge Kutta step by
$$
  F_h(x,a,b):=x + \frac{h}{2} (f(x,a,b) + f(x+h f(x,a,b), a,b)), \quad \mbox{with $h:=\frac{\dt}{p}$}
$$
(here corresponds to the "Heun" scheme).
Then we define 
$$
  y=F(x,a,b)
$$ 
by $y=y_p$ where
$$
   \mbox{$y_0:=x$ \ and \ $y_{k+1}=F_h(y_k,a,b)$, $k=0,\cdots,p-1$}.
$$
We then consider the following approximation $V_{0,p}(x)$ of the continuous value $v_0(x)$:
\be\label{eq:V0-p}
  V_{0,p}(x) =  \inf_{\ma \in S_N} \sup_{b \in B^N} \tJ_{0,p}(x,\ma[b],b)
\ee
where
\be \label{eq:J0-p}
  \tJ_{0,p}(x,\ma[b],b):=\bigg(\max_{0\leq k\leq N-1} 
  G(\tx_k,\ma_k(\tx_k,b_k(\tx_k)), b_k(\tx_k))\bigg) \bigvee \varphi(\tx_N),
\ee
with $\tx_0:=x$ and $\tx_{k+1}:=F(\tx_k,\ma_k(\tx_k,b_k(\tx_k)),b_k(\tx_k))$ for $k=0,\dots,N-1$,
and 
$$G(x,a,b):=\max_{0\leq j < p} g(Y_{j,x}^{a,b})$$ 
with $Y_{0,x}^{a,b}:=x$ and $Y_{j+1,x}^{a,b}:=F_h(Y_{j,x}^{a,b},a,b)$ for $j=0,\dots, p-1$ (for given $(a,b)\in A\times B$),
which amounts to taking the maximum of $g(.)$ along the intermediate substeps of the trajectory.
Hereafter we consider $p=5$ in all the numerical examples, and  the value $V_{0,p}(x)$ (resp. $\tJ_{0,p}$)
will be still denoted $V_0(x)$ (resp. $\tJ_{0}$).

One can show that the
general statements  (representation formula, error estimates, convergence results) remain valid for this multi-step approximation,
see e.g. \cite{bok-pro-war-23} in the one-player context.

All numerical tests are performed using Python 3.10 and Tensorflow, on a Dell xps13-9320 under linux ubuntu 
13th Gen. Intel® Core™ i7-1360P with 12 cores (up to 5 GHz), 32 GiB RAM.

The following numerical examples correspond to computing backward reachable sets for two player differential games.
Hence we can focus on zero-level sets of the value in order to visualise the performance of the algorithms.

\noindent
\paragraph{Example 1.}
We consider the control problem with control sets $A=B=[-1,1]$ and the following dynamics, for $x\in \R^2$ and $(a,b)\in A\times B$:
$$
   f(x,a,b)= a Rx + c b \frac{x}{\|x\|}
$$
with constant $c=0.3$, $R=\begin{pmatrix} 0  &  -1\\ 1 & 0\end{pmatrix}$ (rotation matrix), $\|.\|_2$ is the Euclidean norm.
We fix $T=0.6\pi$.
The terminal cost $\varphi(.)$ is made precise in appendix \ref{app:vex}.
It is designed in order that the solution $v_0(x)$ can be solved analytically.
The target set, defined as $\{x,\varphi(x)\leq 0\}$, is also represented in Fig.~\ref{fig:ex1-1} (left).

We first consider for $v_0(x)$ the general definition \eqref{eq:v0} with a negative obstacle term such as $g(x):=-1$.
Since $\min\varphi>-1$ in this example, this amounts to 
$$ 
  v_0(x)=\inf_{\ma \in \cG} \sup_{b\in \cB_T} \varphi(y_{0,x}^{\ma[b],b}(T)).
$$



The functional to be optimized is chosen as in 
\eqref{eq:optim-def-3} of Thereom~\ref{th:minmax-formula}, that is
\be
   \label{eq:inf-sup-ex1} 
   \inf_{\ma\in \cG^N} \sup_{b\in \cB^N} \E\big[\tJ_0(X,\ma[b],b) \big]
\ee
with functional cost $\tJ_0$ defined as in  \eqref{eq:J0-p} 
and we apply algorithm~1 to deal with the $\inf\sup$.

The neural network space for $A_\theta$ and $B_\theta$ are $\cN_{2+1,1,3,20}$ and $\cN_{2,1,3,20}$ ($3$ layers with $20$ neurons each),
and the function $\xi= \tanh$ is used for both outputs.
The following numerical parameters are used 
$M_{epoch}=500$, $\MPOTE=5$, $N_{batch}=1000$  (which means $2.5\times 10^6$ evaluation of the
functional cost), with initial learning rates $\rho=\eta=2\times 10^{-3}$.
Typical CPU time is about 56 sec. for the obstacle case.

Results are given in Figure~\ref{fig:ex1-1}.
Here we plug in the computed optimal feedback controls $\ma$ (and computed optimal adverse controls $b$)
 in order to draw pictures from the estimate $V_0(x)\simeq \tJ_0(x,\ma[b],b)$.
    \MODIF{We observe in both cases a very good numerical convergence of the global scheme towards the reference solution,
    in the sense that the stochastic gradient algorithm converges rapidly to the optimal value
    and that the semi-discrete value using $N=4$ time steps and $p=5$ substeps 
    is already a good approximation of the continuous value.
    }


\begin{figure}[!hbtp]
  \centering
  \renewcommand{\figw}{0.30\linewidth}
  \renewcommand{\rdir}{ex1-1/}

  \FIGURE{ex1-1 : no obstacle}{
  \renewcommand{\name}{show__lev_G2BA_D2_Nit4_TStep0_Arelu_Neu20_Lay3_num_epochExt100_Order20_Batch1000_POTE5_ADAM_io0000.png}
  \includegraphics[height=\figw,align=c]{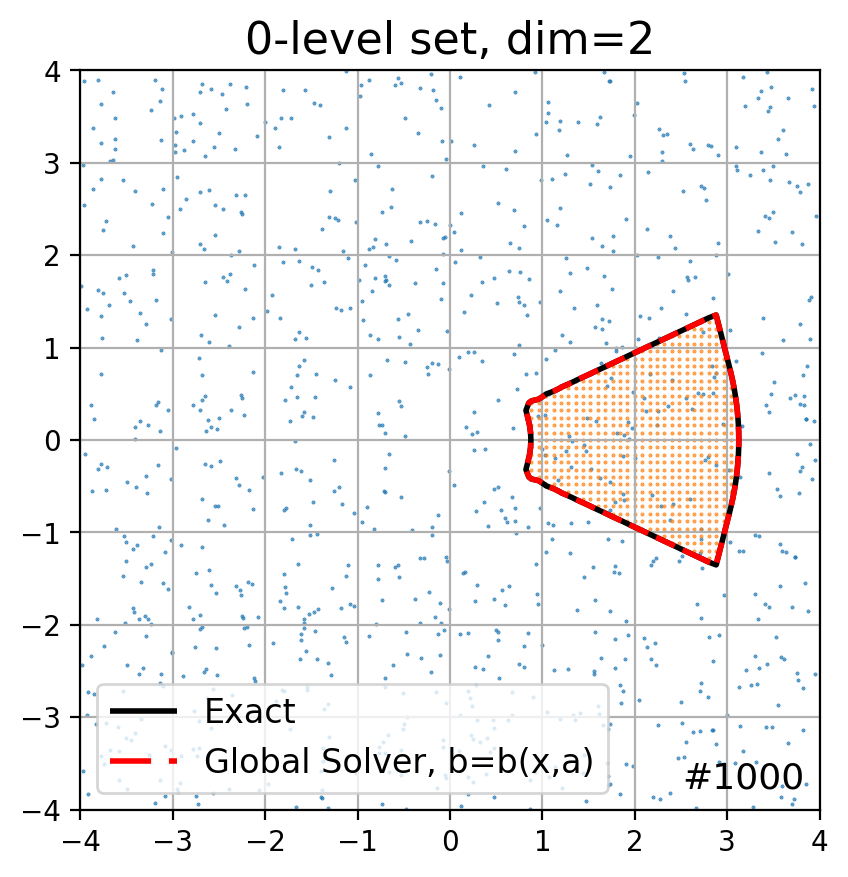}
  %
  \renewcommand{\name}{show__lev_G2BA_D2_Nit4_TStep0_Arelu_Neu20_Lay3_num_epochExt100_Order20_Batch1000_POTE5_ADAM_io0100.png}
  \includegraphics[height=\figw,align=c]{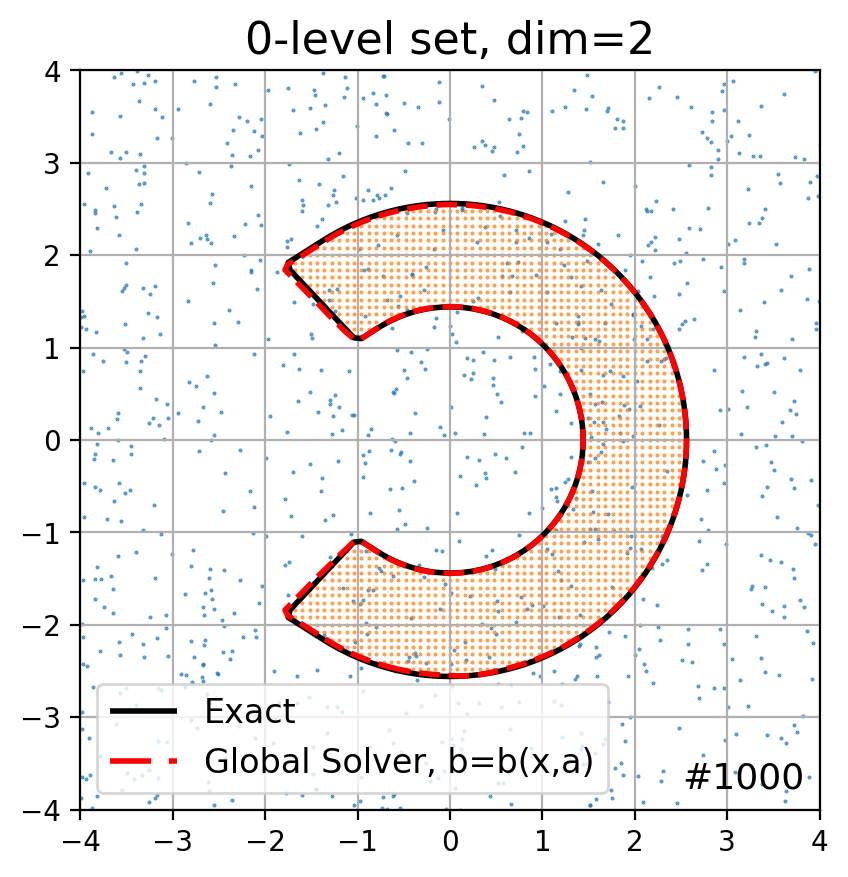}
  %
  \renewcommand{\name}{fig__surf_G2BA_D2_Nit4_TStep0_Arelu_Neu20_Lay3_num_epochExt100_Order20_Batch1000_POTE5_ADAM_io0100.png}
  \includegraphics[height=\figw,align=c]{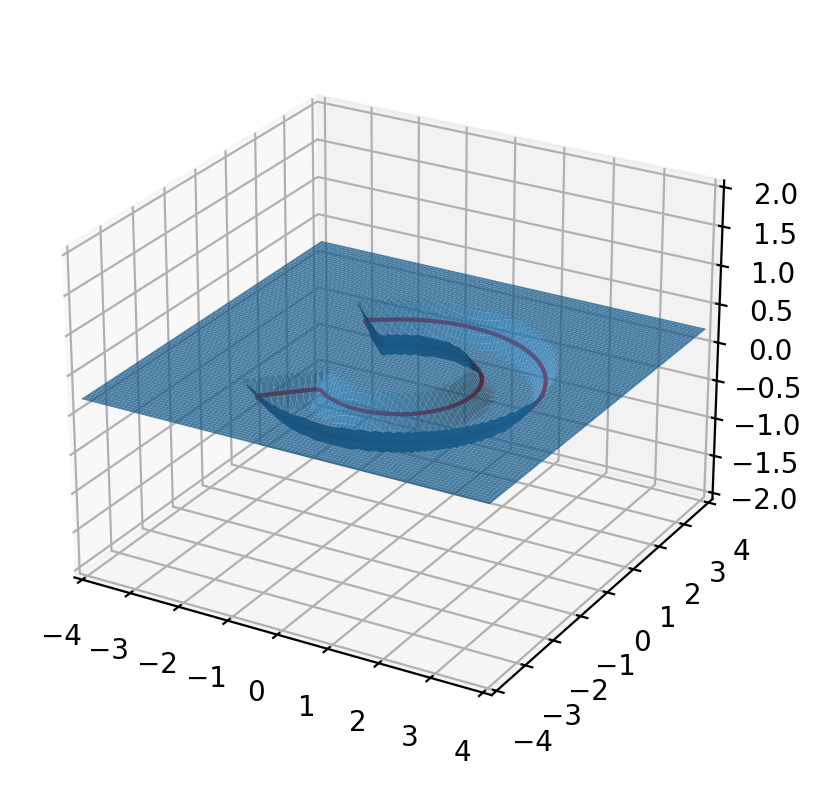}
  }

  \caption{
    \small
    (Example 1, no obstacle)
    Results obtained with the "global scheme":
    $0$-level set of the terminal data (left), $0$-level set of the numerical solution $V_0(x)$ (center) and corresponding surface plot (right).
    Dimension $d=2$, using $N=4$ time steps.
  }
  \label{fig:ex1-1}
\end{figure}


Secondly, using the same data we now consider an other similar problem with an obstacle function given by
$$ g(x) := \min(\bar\eps,\max(-\bar\eps,r_B-\|x-q\|_2), \quad \mbox{with $\bar\eps:=0.2$}
$$
with center $q=(0.5,1.5)$ and radius $r_B:=0.5$, so that $\{x,\ g(x)\leq 0\}$ corresponds to the disk $B(q,r_B)$.
We do not have anymore an analytical solution for this obstacle case. 
Instead, we use a high-order finite difference scheme in order to compute
a reference solution (WENO3-TVDRK3 scheme of Jiang and Peng \cite{Jiang-Peng-2000}, using a Cartesian mesh of size $401^2$).
Results with the obstacle term are given in Figure~\ref{fig:ex1-2}, and we still observe a very good numerical convergence
using the same parameters as before.


\begin{figure}[!hbtp]
  \centering
  \renewcommand{\figw}{0.30\linewidth}

\if{
  \renewcommand{\rdir}{ex1-1/}
  \renewcommand{\name}{show__lev_G2BA_D2_Nit4_TStep0_Arelu_Neu20_Lay3_num_epochExt100_Order20_Batch1000_POTE5_ADAM_io0000.png}
  \includegraphics[height=\figw,align=c]{\rdir\name}
  %
  \renewcommand{\name}{show__lev_G2BA_D2_Nit4_TStep0_Arelu_Neu20_Lay3_num_epochExt100_Order20_Batch1000_POTE5_ADAM_io0100.png}
  \includegraphics[height=\figw,align=c]{\rdir\name}
  %
  \renewcommand{\name}{fig__surf_G2BA_D2_Nit4_TStep0_Arelu_Neu20_Lay3_num_epochExt100_Order20_Batch1000_POTE5_ADAM_io0100.png}
  \includegraphics[height=\figw,align=c]{\rdir\name}
}\fi

  \FIGURE{ex1-2 (obstacle)}{
  \renewcommand{\rdir}{ex1-2/}
  \renewcommand{\name}{show__lev_G2BA_D2_Nit4_TStep0_Arelu_Neu20_Lay3_num_epochExt400_Order20_Batch1000_POTE10_ADAM_io0000.png}
  \includegraphics[height=\figw,align=c]{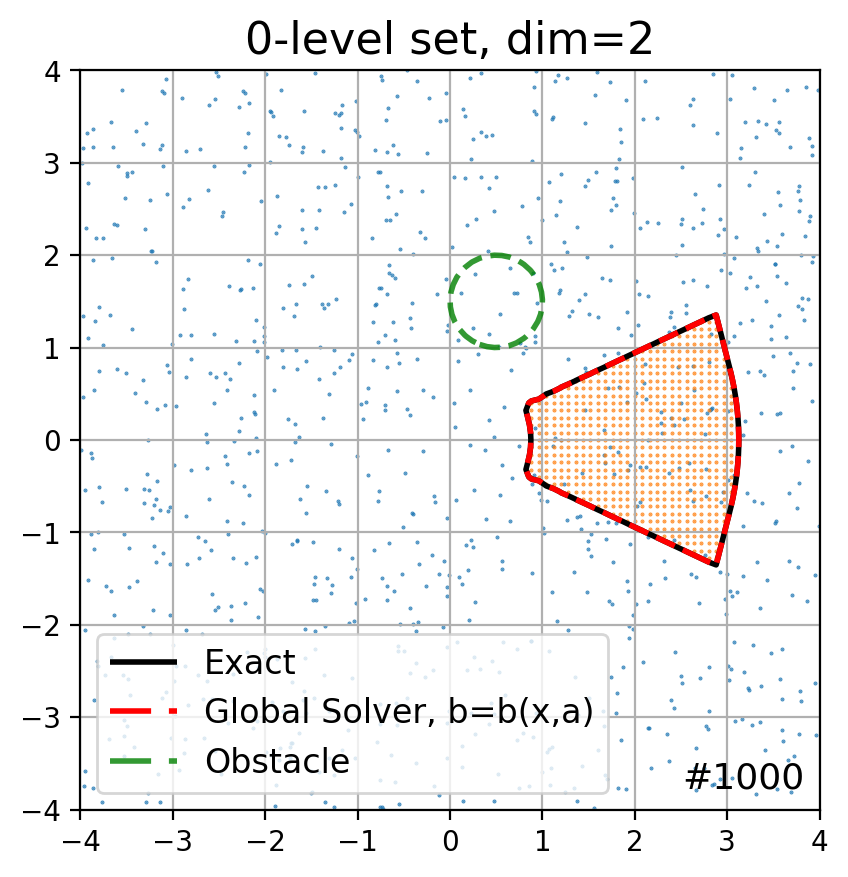}
  %
  \renewcommand{\name}{show__lev_G2BA_D2_Nit4_TStep0_Arelu_Neu20_Lay3_num_epochExt400_Order20_Batch1000_POTE10_ADAM_io0100.png}
  \includegraphics[height=\figw,align=c]{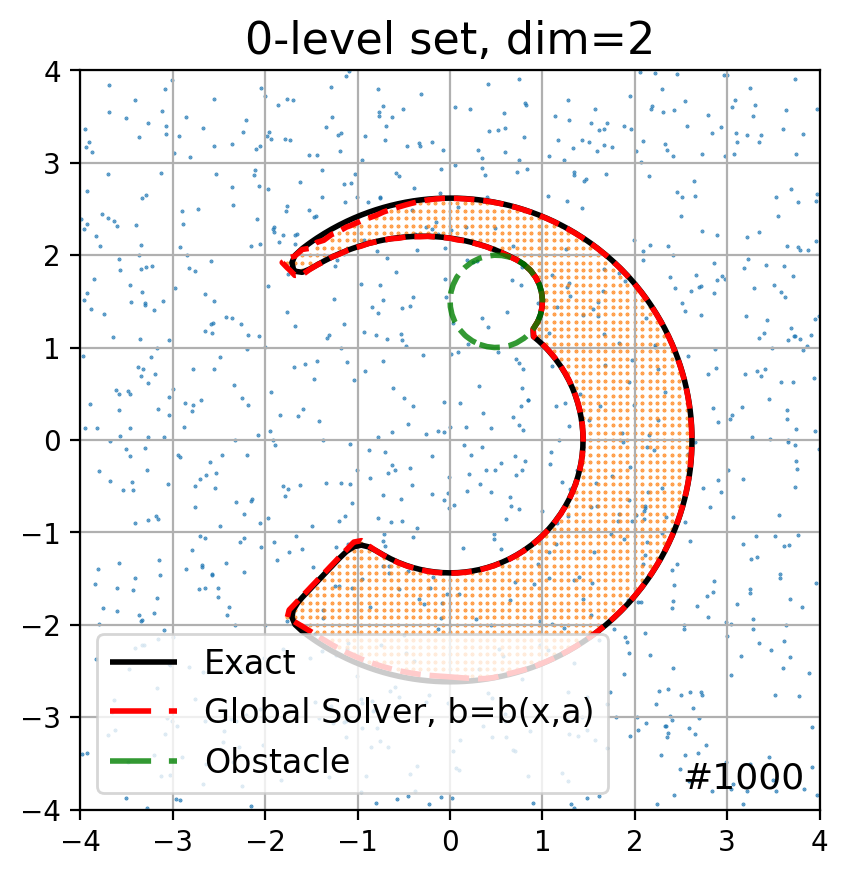}
  %
  \renewcommand{\name}{fig__surf_G2BA_D2_Nit4_TStep0_Arelu_Neu20_Lay3_num_epochExt400_Order20_Batch1000_POTE10_ADAM_io0100.png}
  \includegraphics[height=\figw,align=c]{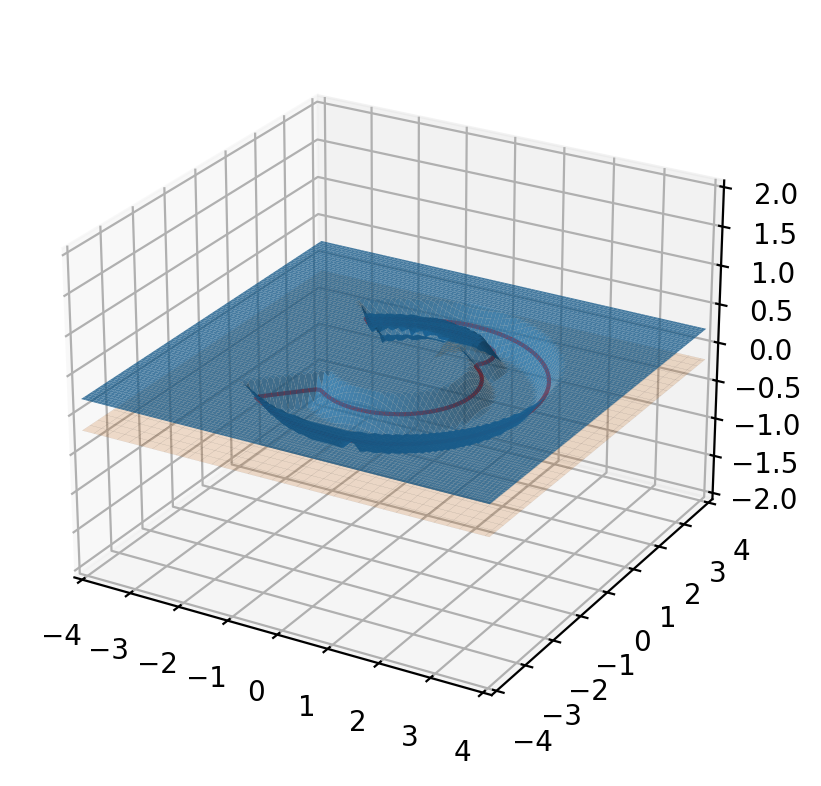}
  }

  \caption{
    \small
    (Example 1, with obstacle)
    Results obtained with the "global scheme".
    Left: $0$-level set of the terminal data (in red) and of the obstacle function (dotted green line);
    center: $0$-level set of the numerical solution $V_0(x)$: right: corresponding surface plot.
    Dimension $d=2$, using $N=4$ time steps.
  }
  \label{fig:ex1-2}
\end{figure}

Finally, we also present results obtained with the time-marching algorithm~2 ("local scheme")
in Figure~\ref{fig:ex1-3}, where the value $V_k$ is estimated for different time index $k\in\{4,3,2,1,0\}$,
with $k=4$ corresponding to $t_4=T$ (the target data), and for $k=0$, $t_0=0$ corresponding to the approximation of $V_0(.)$
by the algorithm.
The neural network space for both $A_\theta$ and $B_\theta$ is now $\cN_{2,1,3,20}$, the other parameters being otherwise unchanged.
We observe quite similar results, although the number of iterations needed might be greater than in algorithm~1 in order
to obtain good results. Indeed, each time step requires a full optimization procedure, whereas
algorithm~1 needs only one optimization problem to be solved.

For the other forthcoming Examples  2 to 4,
our findings is that similar behavior holds, and algorithm~1 leads to similar results with a reduced CPU time cost
than with algorithm~2. Therefore we will not report the results with algorithm~2 in detail.


\begin{figure}[!hbtp]
  \centering
  \renewcommand{\figw}{0.30\linewidth}
  \renewcommand{\rdir}{ex1-3-algo2/}
  \FIGURE{ex1-3-algo2}{
  \begin{tabular}{ccc}
  \noindent
  \renewcommand{\name}{show__lev_L2___D2_Nit4_TStep0_Arelu_Neu20_Lay3_num_epochExt100_Order20_Batch1000_POTE10_ADAM_it00_io0000.png}
  \includegraphics[height=\figw,align=c]{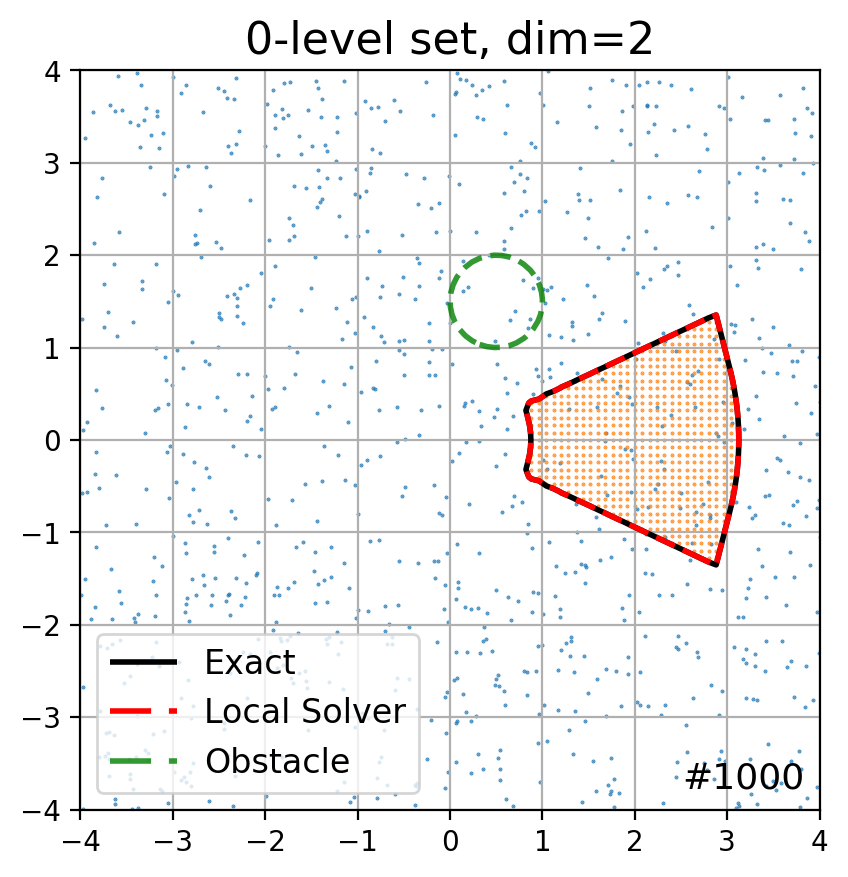}
  &
  \renewcommand{\name}{show__lev_L2___D2_Nit4_TStep0_Arelu_Neu20_Lay3_num_epochExt100_Order20_Batch1000_POTE10_ADAM_it00_io0100.png}
  \includegraphics[height=\figw,align=c]{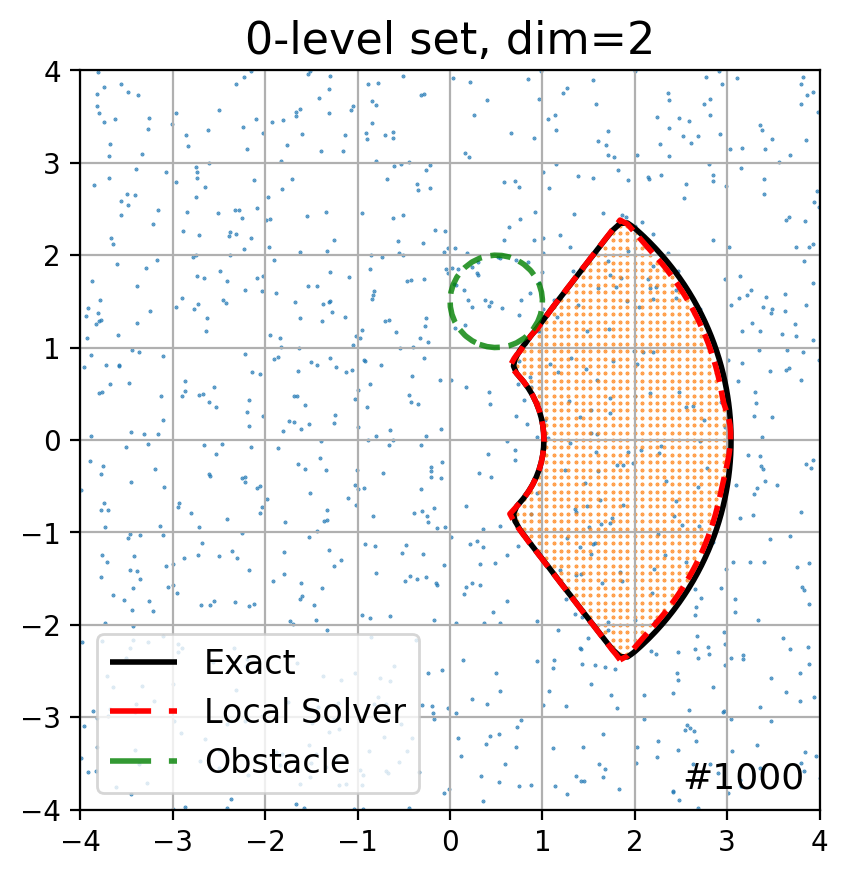}
  &
  \renewcommand{\name}{show__lev_L2___D2_Nit4_TStep0_Arelu_Neu20_Lay3_num_epochExt100_Order20_Batch1000_POTE10_ADAM_it01_io0100.png}
  \includegraphics[height=\figw,align=c]{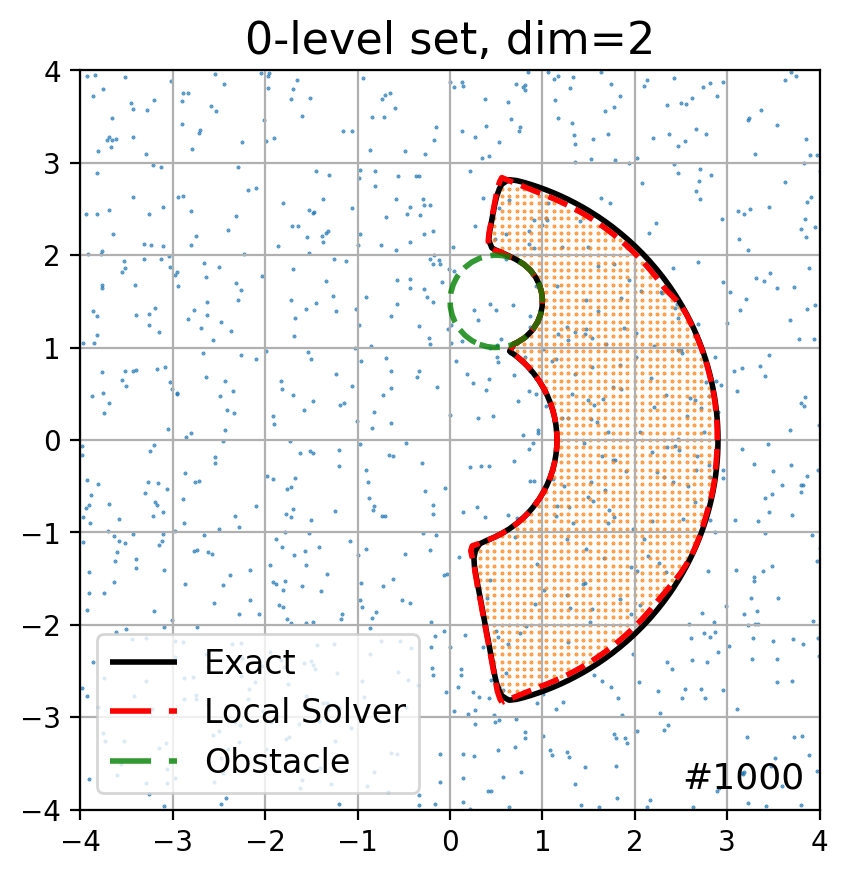}\\
  $t=T$ & $t=0.75\,T$ &  $t=0.5\,T$ \\
  \\[-0.2cm]
  \renewcommand{\name}{show__lev_L2___D2_Nit4_TStep0_Arelu_Neu20_Lay3_num_epochExt100_Order20_Batch1000_POTE10_ADAM_it02_io0100.png}
  \includegraphics[height=\figw,align=c]{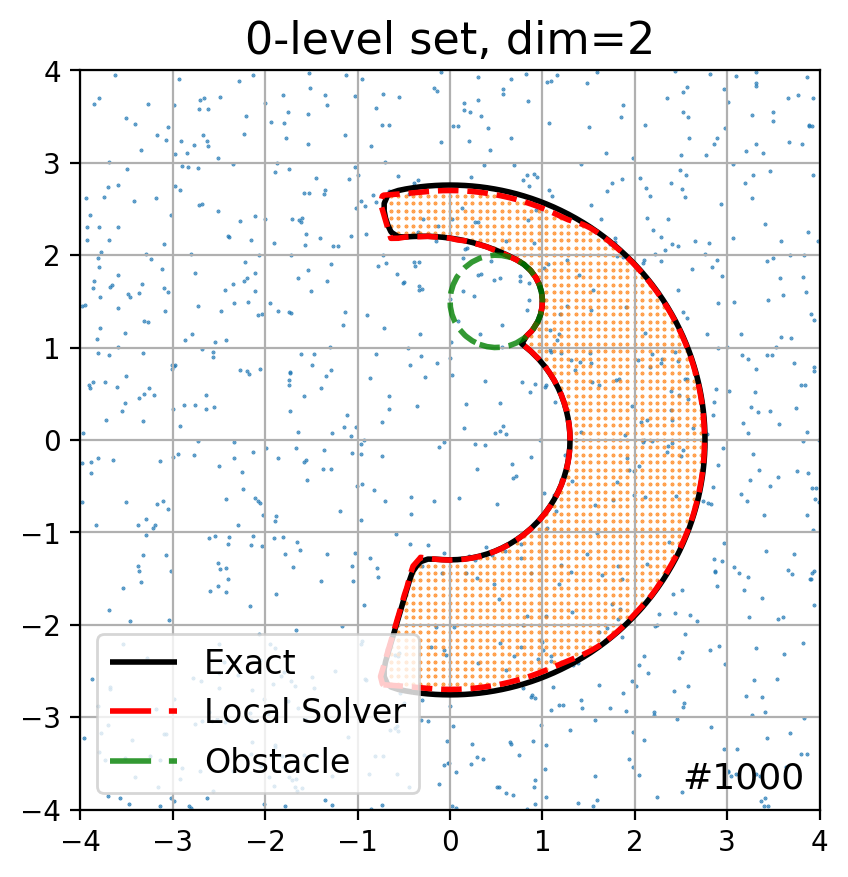}
  & 
  \renewcommand{\name}{show__lev_L2___D2_Nit4_TStep0_Arelu_Neu20_Lay3_num_epochExt100_Order20_Batch1000_POTE10_ADAM_it03_io0100.png}
  \includegraphics[height=\figw,align=c]{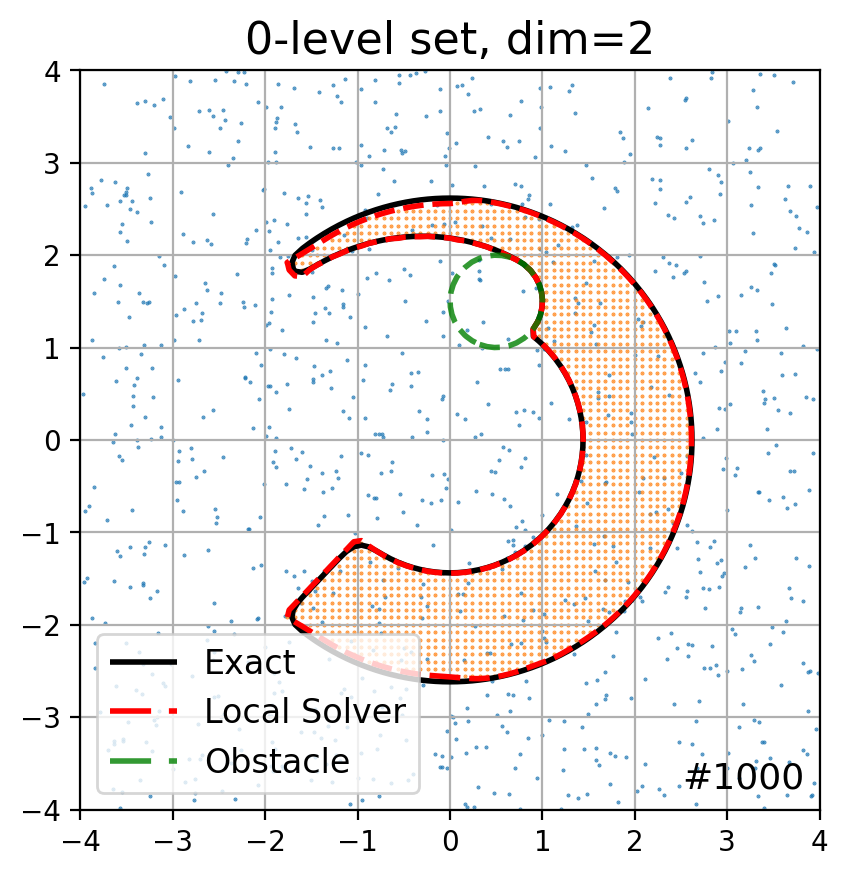}
  &
  \renewcommand{\name}{fig__surf_L2___D2_Nit4_TStep0_Arelu_Neu20_Lay3_num_epochExt100_Order20_Batch1000_POTE10_ADAM_it03_io0100.png}
  \includegraphics[height=\figw,align=c]{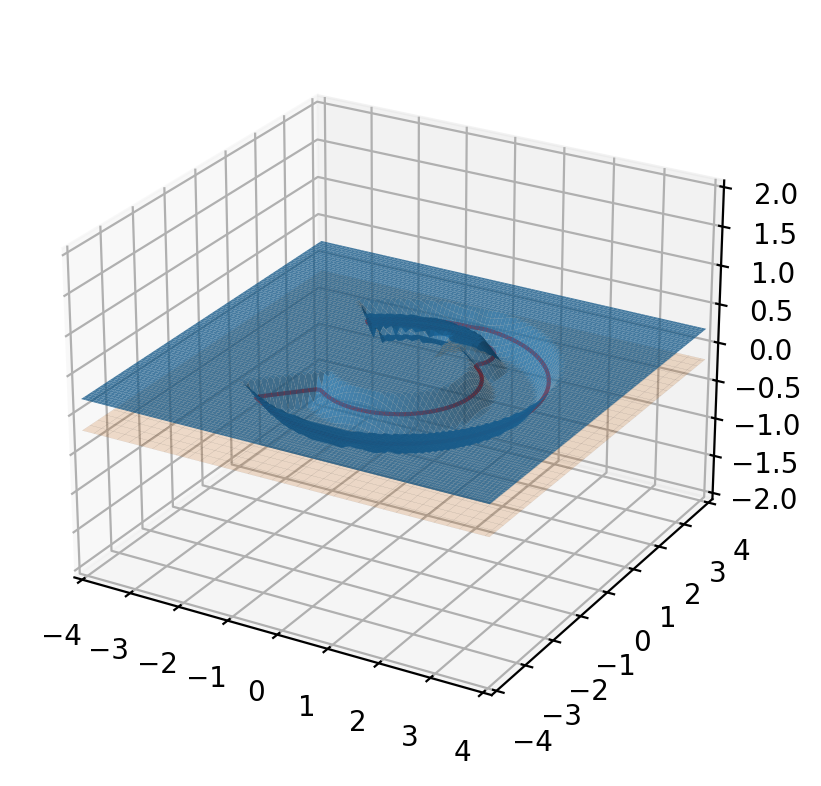}
  \\[-0.1cm]
  $t=0.25\,T $ & $t=0$ & $t=0$
  \end{tabular}
  }

  \caption{
    \small
    (Example 1, with obtacle)
    Results obtained with the algorithm~2 ("local scheme").
    $0$-level sets of $v(t,\cdot)$ at different times $t=\frac{k}{N} T$, $k=0,\dots,N$.
    Dimension $d=2$, using $N=4$ time steps.
  }
  \label{fig:ex1-3}
\end{figure}

\noindent
\paragraph{Example 2.}

We consider the control problem with control sets $A=B=[-1,1]$ and the following dynamics, for $x\in \R^2$ and $(a,b)\in A\times B$:
$$
   f(x,a,b)= \VECT{2 \max(-1,\min(1,a-2b)) \\ a+b}
$$
(we use again $\xi = \tanh$), and
the terminal time and value are given by $T=0.4$ and 
$$
  \varphi(x)= \min(0.5,\max(-0.5, \|x\|_\infty-1)).
$$
In this case we found an analytical formula for the value (see appendix~\ref{app:vex}).
In this exemple it happens that both $\min\max$ and $\max\min$ formulations are equivalent
(i.e., the game {\em has a value}) although this is not immediate when looking at the
dynamics.

We have tested different number of time steps $N=2,4,8,16$, with same optimization problem \eqref{eq:inf-sup-ex1} as in Example~1.
The numerical parameters are otherwise as follows: 
$A_\theta  \in \cN_{2+1,1,3,20}$   and $B_\theta  \in \cN_{2,1,3,20}$ (3 layers of $20$ neurons each), with $\xi_A=\xi_B=\tanh$ for the output,
and
$M_{epoch}=500$, $\MPOTE=5$, $N_{batch}=1000$  (hence an overall of $2.5\times 10^6$ evaluations of the
functional cost), with initial learning rates $\rho=\eta=2\times 10^{-3}$, computational domain $\mO:=[-3,3]^2$.

Results are given in Figure~\ref{fig:ex2} for $N=2,4,8$.
We also give an error table~\ref{tab:ex2} for $N=2,4,8,16$.
For the error we have chosen a {\em local relative error} around the $0$-level set of the value: for a given treshold parameter $\eta>0$,
for a given cartesian mesh grid $(x_i)\in \mO$ of $101^2$ points, we set
$$  e_{L^1,loc} := 
 \frac{\sum_{i,\, x_i\in \mO_{\eta}} | V_0(x_i) - v_0(x_i)|}
      {\sum_{i,\, x_i\in \mO_{\eta}} 1}
$$
where $\mO_\eta:=\{x \in \mO^2,\ |v_0(x)|\leq \eta\}$. Here the values of $v_0$ lay in $[-0.5,0.5]$ and we have set $\eta=0.2$.
We observe roughly a convergence of order between $0.5$ and $1$ with respect to $\tau=\frac{T}{N}$.
(Results with a global relative $L^1$ error are similar on this example.)


\begin{figure}[H]
  \centering
  \renewcommand{\figw}{0.30\linewidth}

  \FIGURE{ex2}{
  \renewcommand{\rdir}{ex2/}
  \renewcommand{\name}{fig__surf_G2AB_D2_Nit2_TStep0_Arelu_Neu20_Lay3_num_epochExt100_Order20_Batch1000_POTE5_ADAM_io0100.png}
  \includegraphics[height=\figw,align=c]{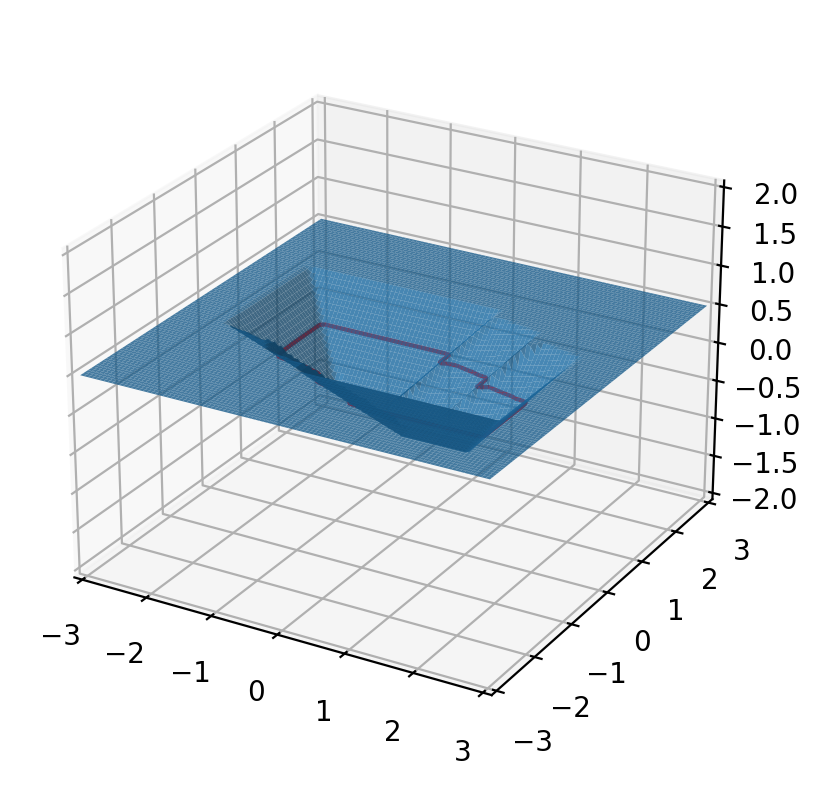}
  %
  \renewcommand{\rdir}{ex2/}
  \renewcommand{\name}{fig__surf_G2AB_D2_Nit4_TStep0_Arelu_Neu20_Lay3_num_epochExt100_Order20_Batch1000_POTE5_ADAM_io0100.png}
  \includegraphics[height=\figw,align=c]{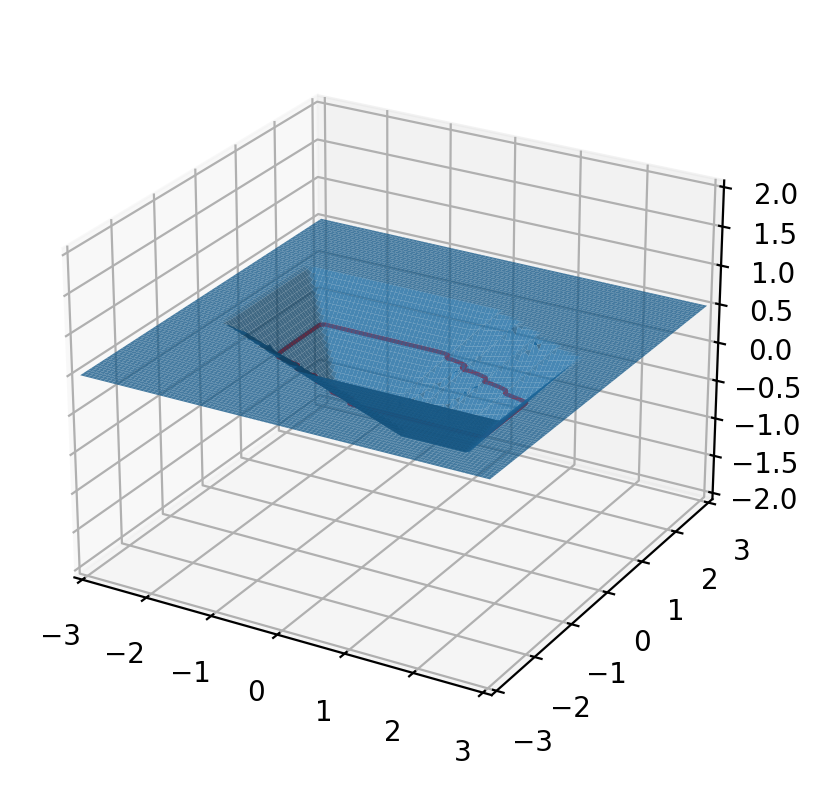}
  %
  \renewcommand{\rdir}{ex2/}
  \renewcommand{\name}{fig__surf_G2AB_D2_Nit8_TStep0_Arelu_Neu20_Lay3_num_epochExt100_Order20_Batch1000_POTE5_ADAM_io0100.png}
  \includegraphics[height=\figw,align=c]{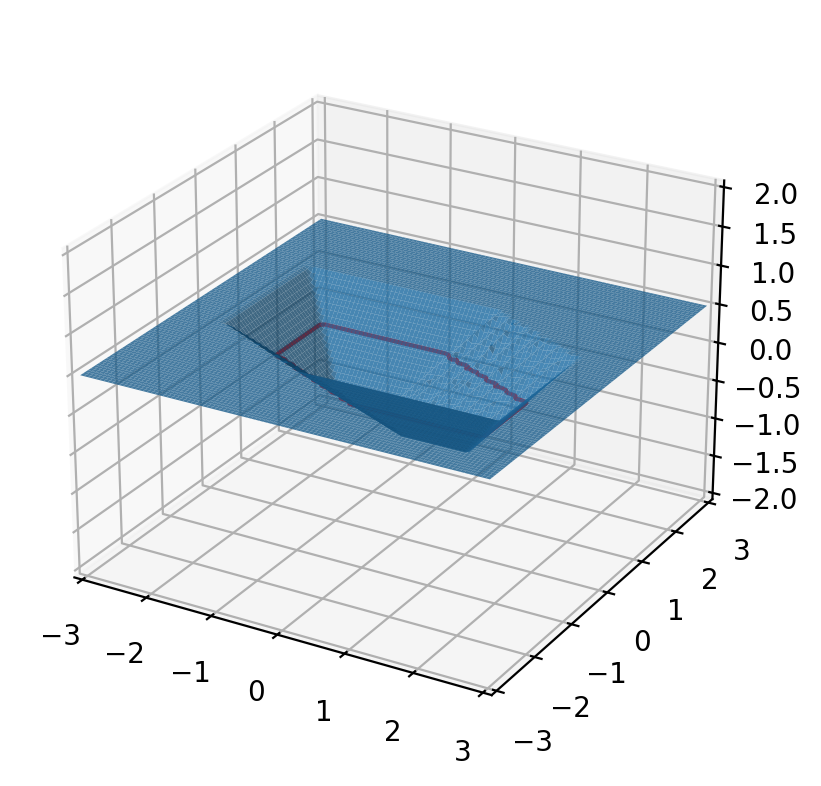}

  \centering
  \renewcommand{\name}{show__lev_G2AB_D2_Nit2_TStep0_Arelu_Neu20_Lay3_num_epochExt100_Order20_Batch1000_POTE5_ADAM_io0100.png}
  \includegraphics[height=\figw,align=c]{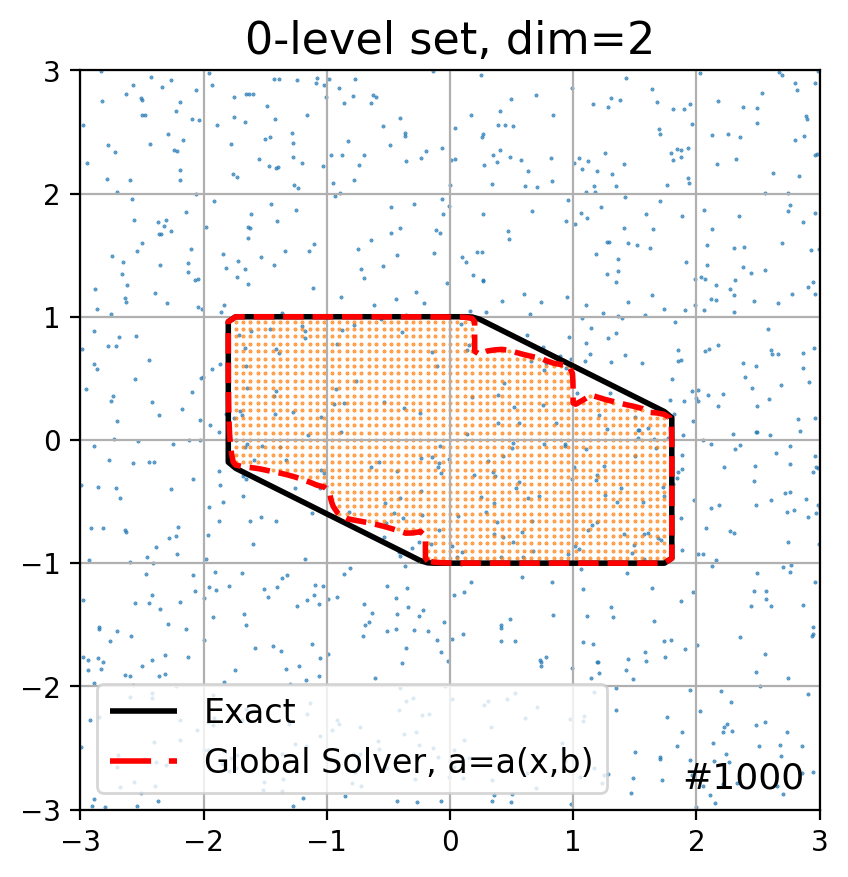}
  %
  \renewcommand{\name}{show__lev_G2AB_D2_Nit4_TStep0_Arelu_Neu20_Lay3_num_epochExt100_Order20_Batch1000_POTE5_ADAM_io0100.png}
  \includegraphics[height=\figw,align=c]{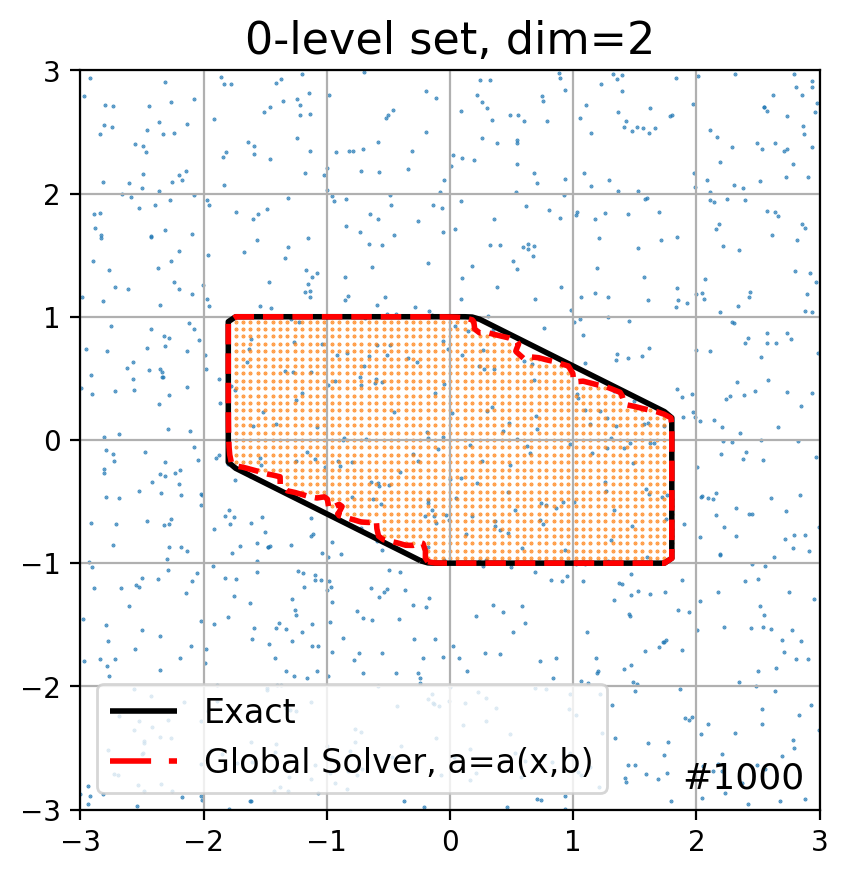}
  %
  \renewcommand{\name}{show__lev_G2AB_D2_Nit8_TStep0_Arelu_Neu20_Lay3_num_epochExt100_Order20_Batch1000_POTE5_ADAM_io0100.png}
  \includegraphics[height=\figw,align=c]{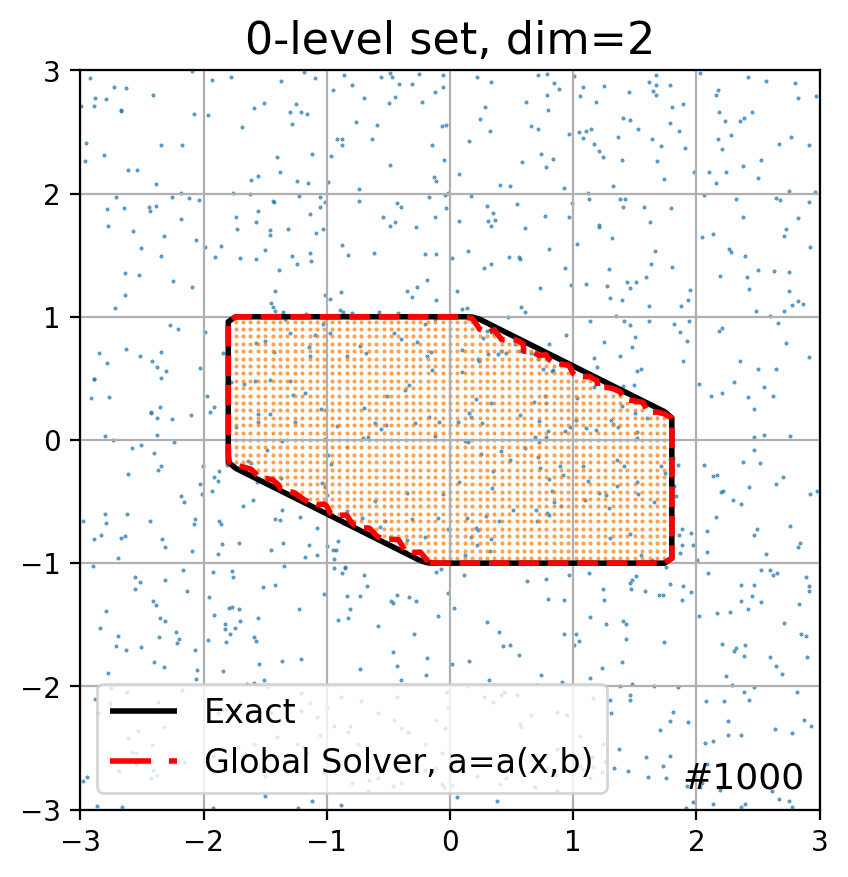}
  }
  %
  \caption{
    \small
    (Example 2)
    Results obtained with the global scheme
    Dimension $d=2$, using $N=2,4$ and $8$ time steps (left, middle and right).
  }
  \label{fig:ex2}
\end{figure}



 \begin{table}[H]
 \centering
 \begin{tabular}{|c|c|c|c|} 
  \hline
  N &  CPU time (s) & $e_{L^1,loc}$ & order \\ \hline
%
   2  &    83.  &  2.95e-02  &   -        \\ 
   4  &   155.  &  1.79e-02  &      0.72  \\ 
   8  &   311.  &  1.27e-02  &      0.50  \\ 
  16  &   809.  &  7.49e-03  &      0.76  \\ 
\hline
 \end{tabular}
   \caption{\label{tab:ex2} (Example 2) Error table with respect to time discretisation parameter $N$}
\end{table}

\noindent
\paragraph{Example 3.}

In this example the dynamics is given by 
$$
   f(x,a,b)= \VECT{2 (1-|a-b|) \\ a+b}
$$
with same terminal value $\varphi$ as in the previous example.

In a first case we consider 
the usual value
$$
  V_0^{-}(x):=\inf_{\ma\in \mG^N} \sup_{b\in B^N} J_0(x,\ma[b],b).
$$
In a second case we consider also the $\disp \sup_{\ma}\inf_{b}$ problem, corresponding to 
$$
  V_0^{+}(x):=\sup_{\ma\in \mG^N} \inf_{b\in B^N} J_0(x,\ma[b],b).
$$

\begin{rem}
  We have
  \be 
    V_0^-(x)\leq V_0^+(x)
  \ee
  and hence 
  $\{x,\ V_0^+(x)\leq 0\} \subset \{x,\ V_0^-(x)\leq 0\}$: the negative region of $V_0^+$ must be included in the one of $V_0^-$.
  Indeed, we already know that 
$$
  V_0^-(x)=\disp\sup_{b_0} \inf_{a_0} \sup_{b_1} \inf_{a_1}  \cdots  J_0(x,a,b).
$$
In the same way, we have
$$
  V_0^+(x)=\disp\inf_{b_0} \sup_{a_0} \inf_{b_1} \sup_{a_1}  \cdots  J_0(x,a,b).
$$
  By using the symmetry of $f$ ($f(x,a,b)=f(x,b,a)$) and the fact that  $A=B$,
  we have also $J_0(x,a,b)=J_0(x,b,a)$ and 
$$
  V_0^+(x)=\disp\inf_{a_0} \sup_{b_0} \inf_{a_1} \sup_{b_1}  \cdots  J_0(x,a,b).
$$
By using the inequality $\disp\inf_a\sup_b\geq \sup_b\inf_a$ we get the desired inequality.
\end{rem}

For the numerical tests, we use the parameter $T=0.4$ and $N=4$ time steps.
Neural network spaces are as in the previous example: 
$A_\theta  \in \cN_{2+1,1,3,40}$   and $B_\theta  \in \cN_{2,1,3,40}$ (3 layers of $40$ neurons each), with $\xi=\tanh$ for the output
for both controls values.
The numerical parameters for SG are otherwise as follows: $N_{batch}=8000$ batch points, 
$M_{epoch}=3000$ and $\MPOTE=10$ internal iterations,
and initial learning rates $\rho=\eta=10^{-3}$,
 computational domain $\mO:=[-3,3]^2$.

Results are given in Figure~\ref{fig:ex3}, together with the exact solutions which are given in appendix~\ref{app:vex}.
CPU time is about 960 sec for each example.

In this case we remark that the $\min$ and $\max$ do not commute, in the sense that in general 
$$ \min_a \max_b f(x,a,b)\cdot p \neq \max_b \min_a f(x,a,b)\cdot p. $$
We found this example numerically more difficult, with important oscillations in the stochastic gradient (SG) algorithm,
compared to the previous examples,
and the need of a finer sampling and greater number of SG iterations in order to get relevant results.

\begin{figure}[!hbtp]

  \renewcommand{\figw}{0.30\linewidth}
  \centering
  \FIGURE{ex3}{
  \small
  \renewcommand{\rdir}{ex3-1/}
  \begin{tabular}{ccc}
  \renewcommand{\name}{fig__surf_G2AB_D2_Nit4_TStep0_Arelu_Neu20_Lay3_num_epochExt400_Order20_Batch8000_POTE5_ADAM_io0000.png}
  \includegraphics[height=\figw,align=c]{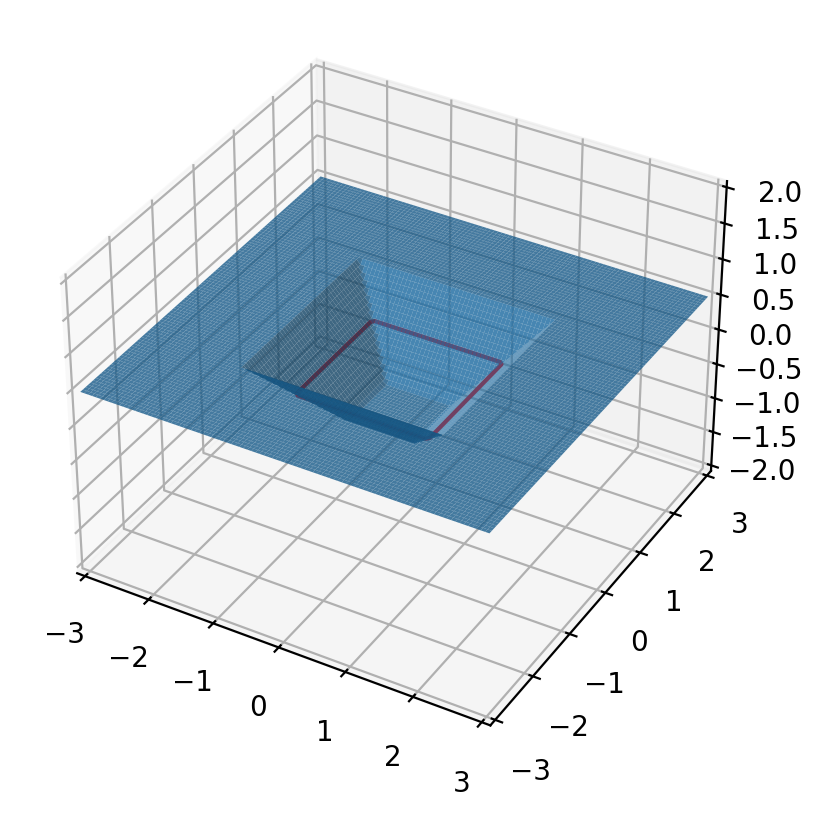}
  &
  \renewcommand{\rdir}{ex3-1/}
  \renewcommand{\name}{fig__surf_G2AB_D2_Nit4_TStep0_Arelu_Neu20_Lay3_num_epochExt500_Order20_Batch8000_POTE10_ADAM_io0361.png}
  \includegraphics[height=\figw,align=c]{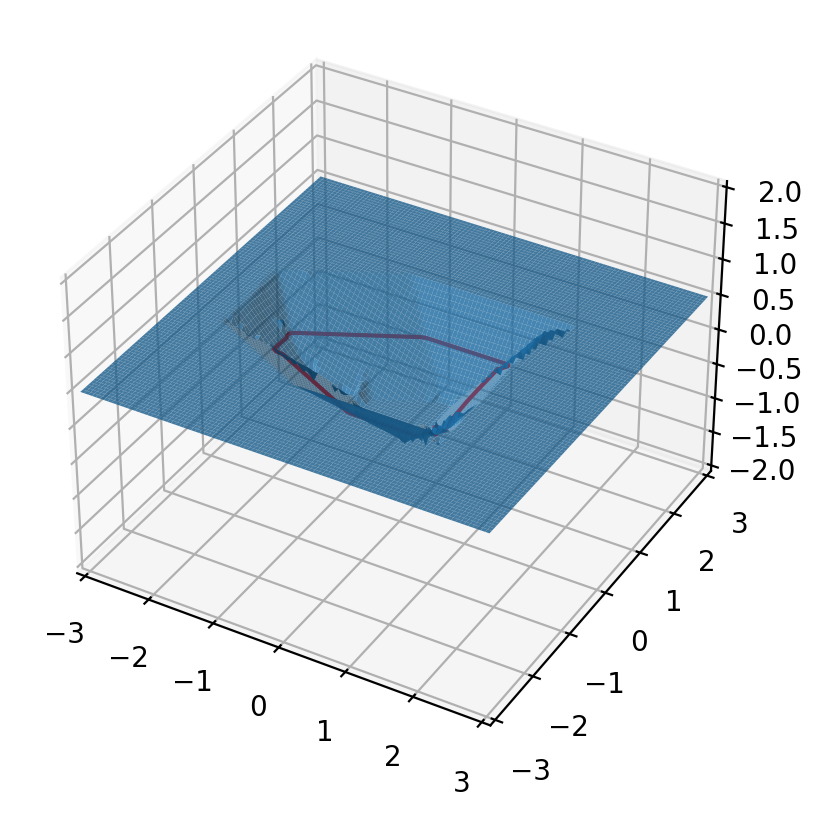}
  &
  \renewcommand{\rdir}{ex3-2/}
  \renewcommand{\name}{fig__surf_G2AB_D2_Nit4_TStep0_Arelu_Neu40_Lay3_num_epochExt500_Order20_Batch8000_POTE10_ADAM_io0300.png}
  \includegraphics[height=\figw,align=c]{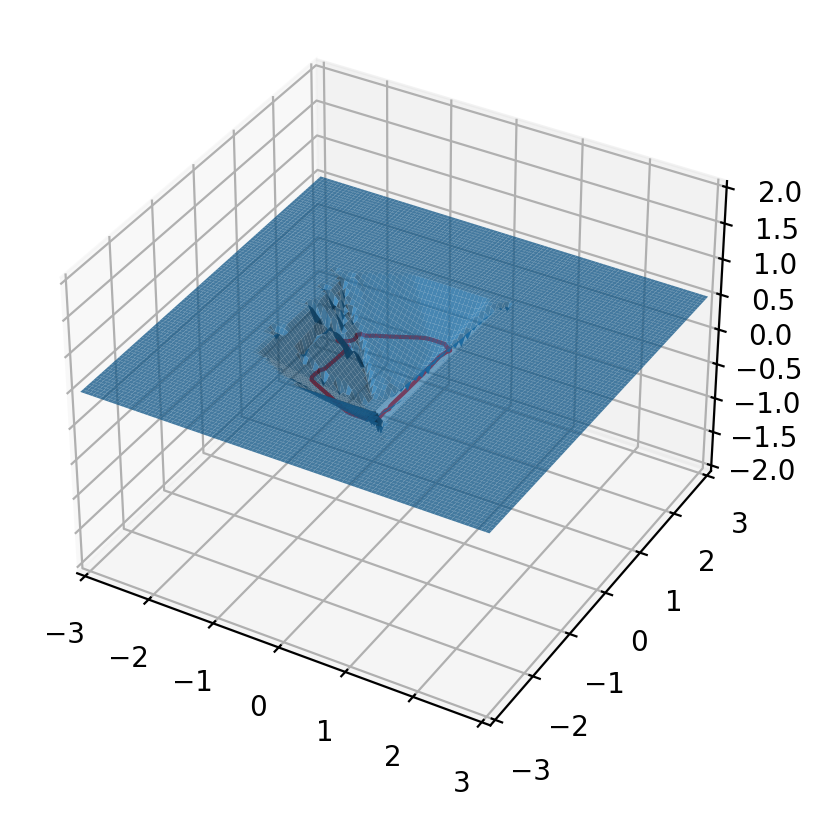}
  \\
  \renewcommand{\rdir}{ex3-1/}
  \renewcommand{\name}{show__lev_G2AB_D2_Nit4_TStep0_Arelu_Neu20_Lay3_num_epochExt500_Order20_Batch8000_POTE10_ADAM_io0000.png}
  \includegraphics[height=\figw,align=c]{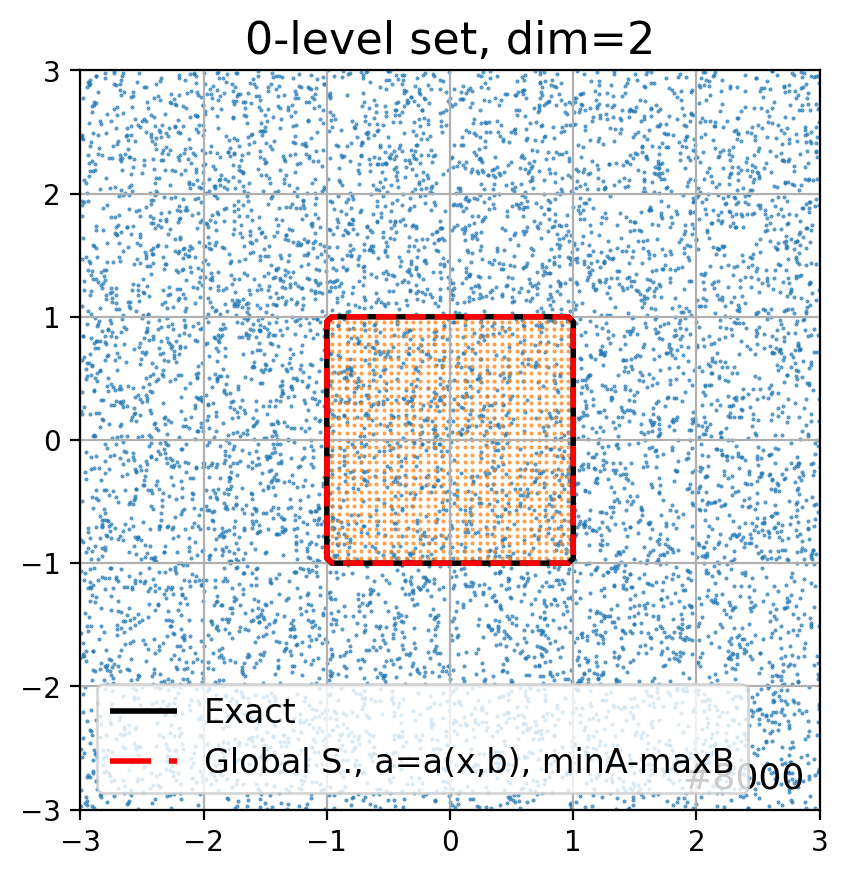}
  &
  \renewcommand{\rdir}{ex3-1/}
  \renewcommand{\name}{show__lev_G2AB_D2_Nit4_TStep0_Arelu_Neu20_Lay3_num_epochExt500_Order20_Batch8000_POTE10_ADAM_io0361.png}
  \includegraphics[height=\figw,align=c]{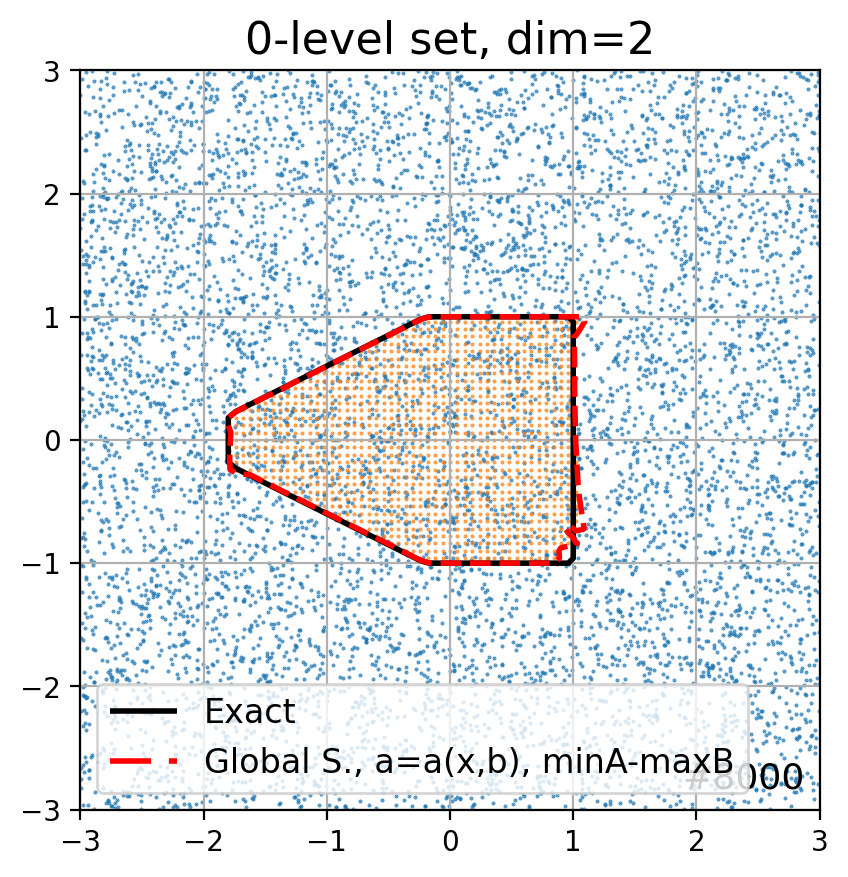}
  &
  \renewcommand{\rdir}{ex3-2/}
  \renewcommand{\name}{show__lev_G2AB_D2_Nit4_TStep0_Arelu_Neu40_Lay3_num_epochExt500_Order20_Batch8000_POTE10_ADAM_io0300.png}
  \includegraphics[height=\figw,align=c]{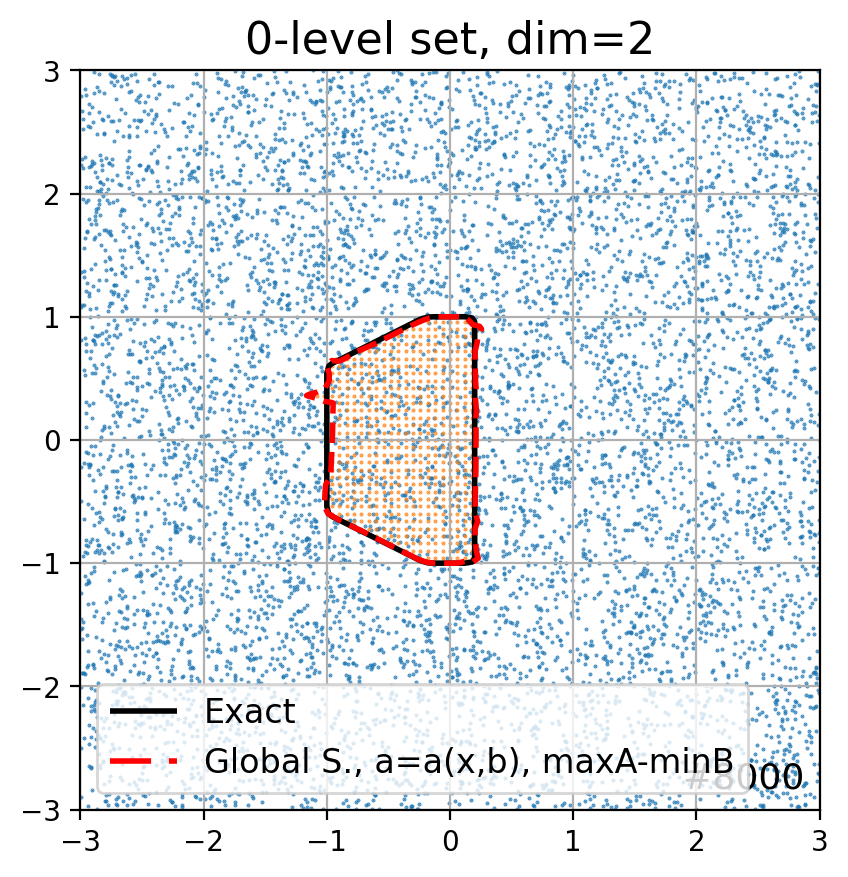}
  \\
  $\varphi(.)$ & $V_0^{-}(.)$ & $V_0^{+}(.)$
  \end{tabular} 
  } 
  \caption{
    \small
    (Example 3, $d=2$)
    Results obtained with the global scheme
    using $N=4$ time iterations, left: terminal data, middle: $V_0^{-}$, right: $V_0^{+}$ (see text).
  }
  \label{fig:ex3}
\end{figure}

\noindent
\paragraph{Example 4.}
We consider now a two-player game, where the first player is $X_1=(x_1,x_2)$ with dynamics given by
\beno
 \left\{\barr{l}
  \dot x_1(t) = V_1 a_1(t)\\ 
  \dot x_2(t) = V_1 a_2(t)   
 \earr\right.
\eeno
with a control $a(t)=(a_1(t),a_2(t))\in A:=B_2(0,1)$ 
which is the "direction" of $\dot X_1$, where $B_2(0,1)$ is the unit ball of $\R^2$ for the Euclidean norm,
and the second player is modelized similarly by $X_2=(x_3,x_4)$ with dynamics
\beno
 \left\{\barr{l}
  \dot x_3(t) = V_2 b_1(t)\\ 
  \dot x_4(t) = V_2 b_2(t)   
 \earr\right.
\eeno
with an adverse control $b(t)=(b_1(t),b_2(t))\in B:=B_2(0,1)$.
The horizon parameter and velocities are as follows
$$
  \mbox{$T=4$, $V_1=1$, $V_2=0.7$.}
$$
In this example, $x\in \R^4$ and the global dynamics is therefore given by 
$$
  f(x,a,b):=(V_1 a_1,\, V_1 a_2,\, V_2 b_1,\, V_2 b_2).
$$
This example is more complex than the previous examples in the sense that we will not have an analytic solution and
it is higher dimensional.
However a reference solution can still be obtained by using a classical full grid approach with a finite difference scheme
(we use here a WENO3-RK3 finite difference scheme as described in Jiang and Peng \cite{Jiang-Peng-2000}, 
with a uniform grid of $51^4$ points in space).

The first player, starting at a given point $(x_1,x_2)\in \R^2$
aims to reach a target $\cC=B_2(x_A,r_A)$ (the Euclidian ball centered at $x_A$ and of radius $r_A$),
at some time $\tau$ before $T$,
keeping away from the the second player starting at $(x_3,x_4)\in \R^2$ (i.e. $\|X_1(t)-X_2(t)\|_2\geq R_0$,
for a given threshold $R_0 > 0$) 
for all $t\in[0,\tau]$ whatever the adverse control can be,
staying also in a given domain 
$X_1\in \cK_1=\R^4\backslash int(\cO)$ of $\R^2$, 
where $\cO$ is a square obstacle which defined by
$$
  \cO:=B_\infty(x_B,r_B)=\{x\in\R^2,\ \|x-x_B\|_\infty\leq r_B\}
$$
(so we require that $X_1(t) \in \cK_1$ for all $t\in [0,\tau]$).

Following \cite{bok-for-zid-2010-2},
the problem is a reachability problem with state constraints, which can be reformulated with level sets as follows.
We consider the value $v_0(x)$ as in \eqref{eq:v0}, which we recall 
$$
  v_0(x)= \inf_{\ma[.]\in \mG_{(0,T)}} \sup_{b\in \cB_{T}} \ \ 
    \max\bigg( \varphi\big(y_{0,x}^{\ma[b],b}(T)\big), \ \max_{\mt\in (0,T)} g\big(y_{0,x}^{\ma[b],b}(\mt)\big) \bigg)
$$
where the obstacle function $g$ and the terminal cost $\varphi$ are now defined.
For $x=(x_1,x_2,x_3,x_4)$, let the terminal cost be defined by
$$
  \varphi(x) :=  \|(x_1,x_2)-x_A\|_2 - r_A
$$
(so that $\varphi(x)\leq 0$ $\equivalent$ $(x_1,x_2)\in B_2(x_A,r_A)$).
For the obstacle part, let
$$
  g(x):=\max(g_0((x_1,x_2)),g_a(x))
$$ 
where
$g_0(.)$ is a level set function to encode the obstacle $B_\infty(x_B,r_B)$:
$$
  g_0(x):= r_B-\|(x_1,x_2)-x_B\|_\infty
$$
so that $g_0(x)\leq 0$ $\equivalent$ $x\in \R^2\backslash B_\infty(x_B,r_B)$.
Finally $g_a(.)$ is defined by
$$
  g_a(x):=R_0- \|(x_1,x_2)-(x_3,x_4)\|_2, \quad \mbox{with $R_0:=1$,}
$$
so that $g_a(x)\leq 0$ $\equivalent$ $\|(x_1,x_2)-(x_3,x_4)\|_2\geq R_0$
(i.e., $g_a(x)$ is a level set function to encode the avoidance of $X_1$ with $X_2$).

For the numerical computations we use also the parameters
$$
  \mbox{$x_A=(3,0)$, $r_A=1$ $\quad$ and $\quad$ $x_B=(0.5,1.5)$, $r_B=0.75$}.
$$
The whole computational domain is $\mO:=[-5,5]^4$.
Plots of the terminal data $\varphi(.)$ and of the obstacle function $g(.)$ are shown in 
Figure~\ref{fig:ex4-0}, with a cut in the plane $\R^2\times\{(\bx_3,\bx_4)\}$ with $(\bx_3,\bx_4)=(0,-2)$, 
which will be the adverse starting position for showing all results.

The functional to be optimized is chosen as in \eqref{eq:inf-sup-ex1}, that is:
\beno
    \inf_{\ma\in \cG^N} \sup_{b\in \cB^N}  \E\big[\tJ_0(X,\ma[b],b) \big]
\eeno
(with functional cost $\tJ_0$ as in  \eqref{eq:J0-p}).

We consider the global solver algorithm~1.
Computations are done with neural networks using $3$ layers of $40$ neurons each,
($A_\theta\in\cN_{4+1,2,3,40}$ and $B_\theta\in \cN_{4,2,3,40}$).
which corresponds to $(N_a,N_b)=(3684,3604)$ network parameters for functions $(a_k,b_k)$ at each time step,
the number of batch points is $N_{batch}=50000$ and the number of SG iterations $M_{epoch}=5000$ (after which small oscillations remain)
$\MPOTE=5$ internal iterations, with $\rho=\eta=10^{-3}$.
CPU time is around $34000$ sec. for $N=8$.


Results are given in Figure~\ref{fig:ex4-1} (for resp. $N=2$, $4$ and $8$ time steps, left side).
The doted (orange) region corresponds to the points $(x_1,x_2)$ or $\R^2$ such that $v(0,x)\leq 0$, where $x=(x_1,x_2,\bx_3,\bx_4)$
and $(\bx_3,\bx_4)$ has a given value (e.g. $(-2,0)$ in the graphics). This region represents the feasible region, from which it 
is safely possible to reach the target before capture. 
The left figure represents the target (corresponding to the negative region of the value at terminal time).
About $10\%$ of the batch points are also projected in the plane of visualization.

We have also tested a reversed formulation, i.e., 
based on  
\be \label{eq:V0-maxB-minA}
    \sup_{b\in \cB^N}  \inf_{\ma\in \cG^N} \E\big[\tJ_0(X,\ma[b],b) \big]
\ee
and algorithm~1.
Note that this value is the same as 
\be
  \sup_{b\in \cB^N}  \inf_{a\in \cA^N} \E\big[\tJ_0(X,a(X),b(X)) \big]
\ee
(in a same way as \eqref{eq:simple-eq}),
and it can be shown that it has the same limit value $v_0(x)$ as $N\conv \infty$.
We observe that this gives numerically more precise results for a lower computational cost.
Results for this formulation are also presented in Figure~\ref{fig:ex4-1} (center and right columns), 
for  $N\in\{2,4,8\}$.
Here the neural networks are composed of $3$ layers of $20$ neurons each,
i.e., $A_\theta\in\cN_{4+1,2,3,20}$ and $B_\theta\in \cN_{4,2,3,20}$
(which corresponds respectively to $(N_a,N_b)=(1044,1004)$ network parameters for functions $(a_k,b_k)$ at each time step),
$N_{batch}=20000$, $M_{epoch}=1000$ (number of SG iterations), with $\MPOTE=10$ internal iterations,
using initial learning rates $\rho=\eta=2\times 10^{-3}$.
We observe a numerical convergence
after about $M_{epoch}=500$ iterations, after which small oscillations remain.
CPU time is around $760$ sec. for $N=4$ and $1600$ sec. for $N=8$.

\begin{figure}[!hbtp]
  \centering
  \FIGURE{ex4-0}{
  \renewcommand{\figw}{0.30\linewidth}
  \renewcommand{\rdir}{ex4-0/} 
  %
  \renewcommand{\name}{show__lev_G2AB_D4_Nit2_TStep2_Arelu_Neu20_Lay3_num_epochExt200_Order20_Batch20000_POTE20_ADAM_io0000.png}
  \includegraphics[height=\figw,align=c]{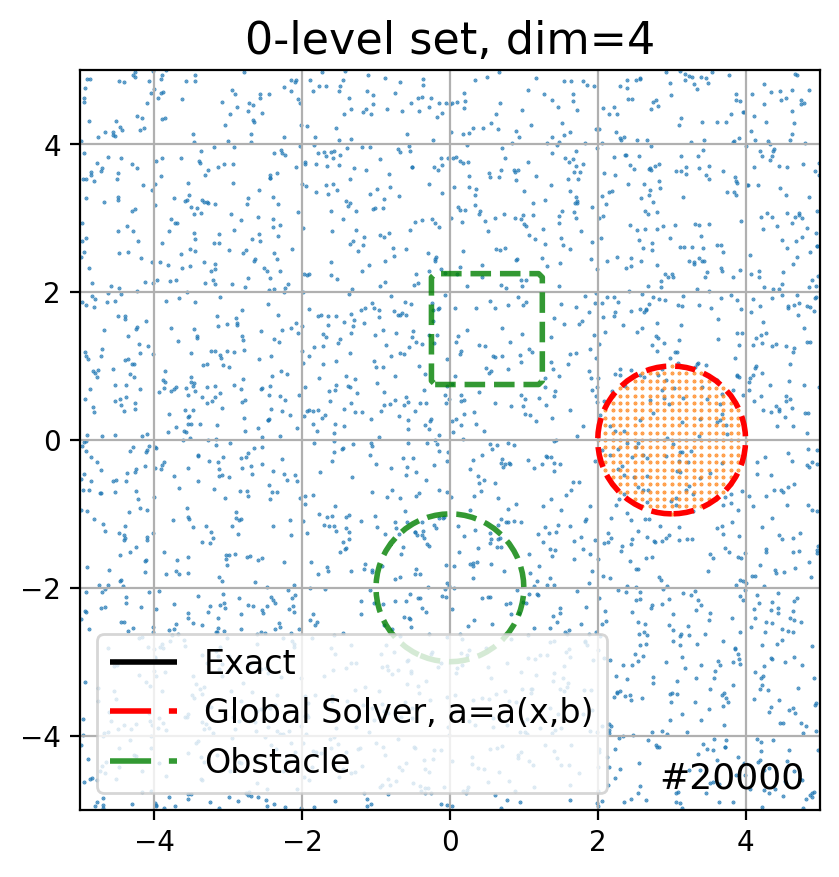}
  \renewcommand{\name}{fig__surf_G2AB_D4_Nit2_TStep2_Arelu_Neu20_Lay3_num_epochExt200_Order20_Batch20000_POTE20_ADAM_io0000.png}
  \includegraphics[height=\figw,align=c]{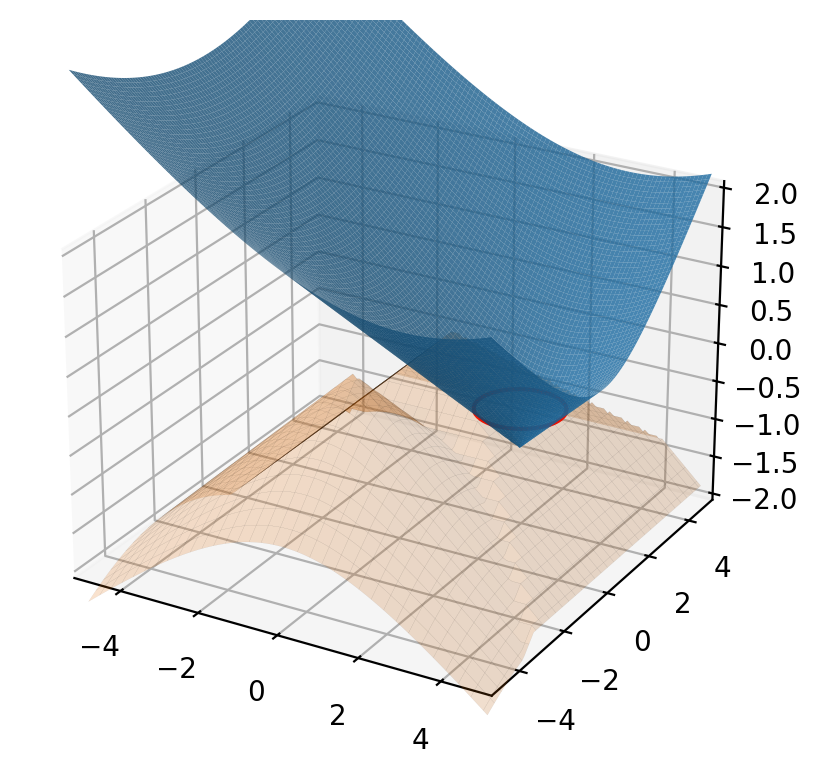}
  }
  \caption{
    \small
    (Example 4, $d=4$)
    cut in plane $(\bx_3,\bx_4)=(0,-2)$ for the adverse starting position.
    Plots of the terminal data and obstacle functions: $0$-level sets (left) and values (right).
  }
  \label{fig:ex4-0}
\end{figure}

\begin{figure}[!hbtp]
  \centering
  \FIGURE{ex4}{
  \renewcommand{\figw}{0.30\linewidth}
  \footnotesize{$N=2$}\hspace*{-2ex}
  \renewcommand{\rdir}{ex4-new-0/} 
  \renewcommand{\name}{show__lev_G2AB_D4_Nit2_TStep2_Arelu_Neu40_Lay3_num_epochExt2000_Order20_Batch50000_POTE5_ADAM_io0050.png}
  \includegraphics[height=\figw,align=c]{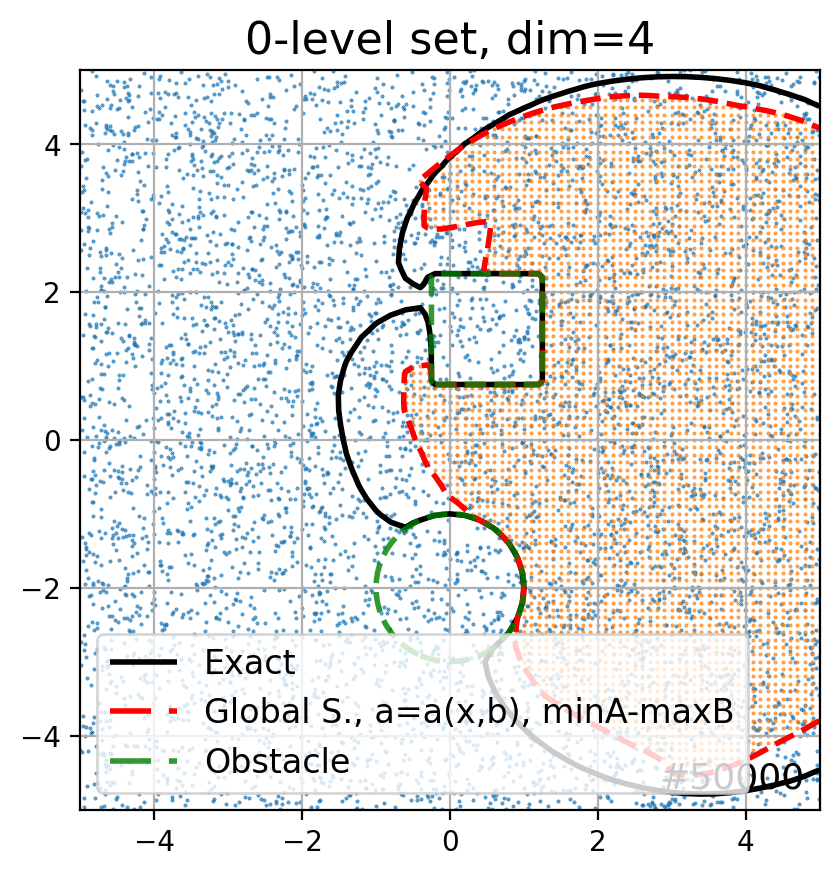}
  \hspace*{-2ex}
  \renewcommand{\rdir}{ex4-0/} 
  \renewcommand{\name}{show__lev_G2AB_D4_Nit2_TStep2_Arelu_Neu20_Lay3_num_epochExt200_Order20_Batch20000_POTE20_ADAM_io0070.png}
  \includegraphics[height=\figw,align=c]{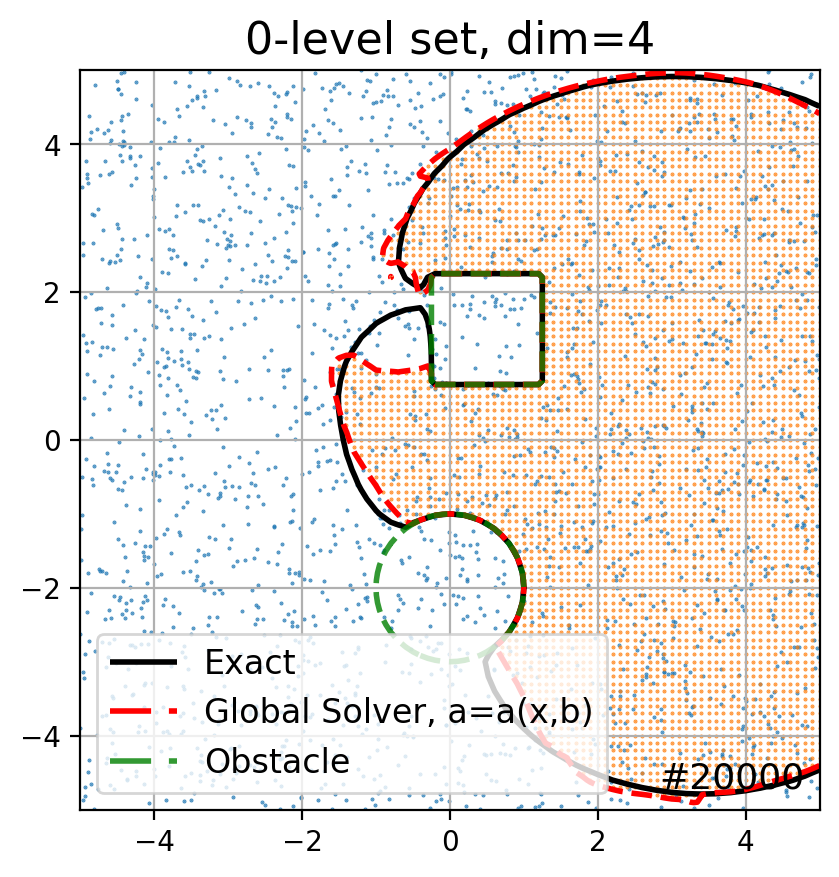}
  \hspace*{-2ex}
  \renewcommand{\name}{fig__surf_G2AB_D4_Nit2_TStep2_Arelu_Neu20_Lay3_num_epochExt200_Order20_Batch20000_POTE20_ADAM_io0070.png}
  \includegraphics[height=\figw,align=c]{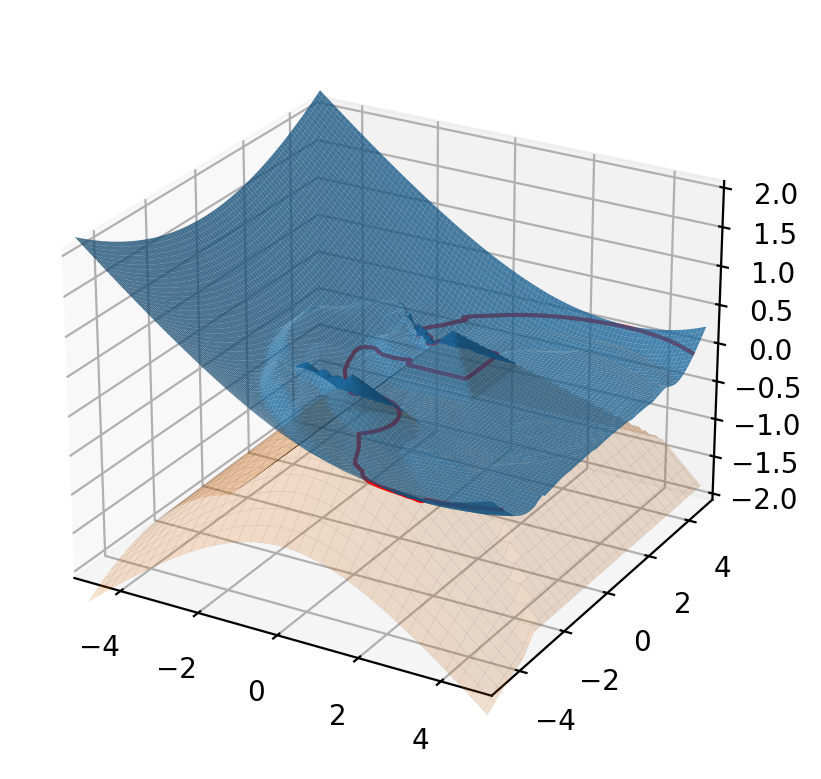}

  \renewcommand{\figw}{0.30\linewidth}
  \renewcommand{\rdir}{ex4-1/} 
  \footnotesize{$N=4$}\hspace{-2ex}
  %
  %
  \renewcommand{\rdir}{ex4-new-1/} 
  \renewcommand{\name}{show__lev_G2AB_D4_Nit4_TStep1_Arelu_Neu40_Lay3_num_epochExt2000_Order20_Batch50000_POTE5_ADAM_io0250.png}
  \includegraphics[height=\figw,align=c]{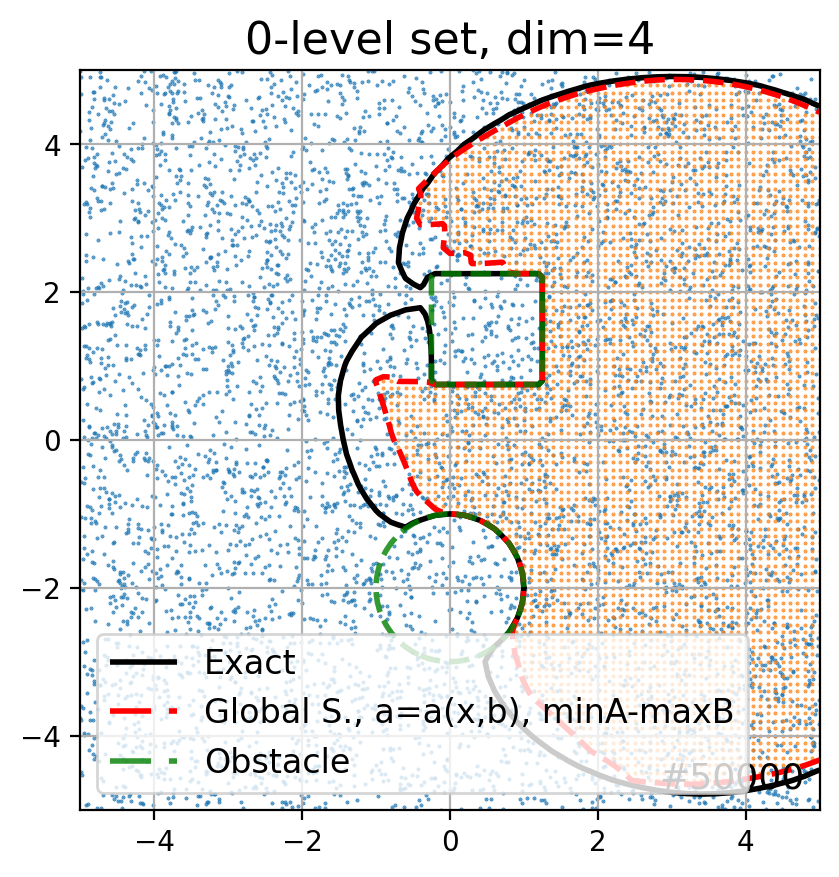}
  \hspace*{-2ex}
  \renewcommand{\rdir}{ex4-1/} 
  \renewcommand{\name}{show__lev_G2AB_D4_Nit4_TStep1_Arelu_Neu20_Lay3_num_epochExt200_Order20_Batch20000_POTE20_ADAM_io0200.png}
  \includegraphics[height=\figw,align=c]{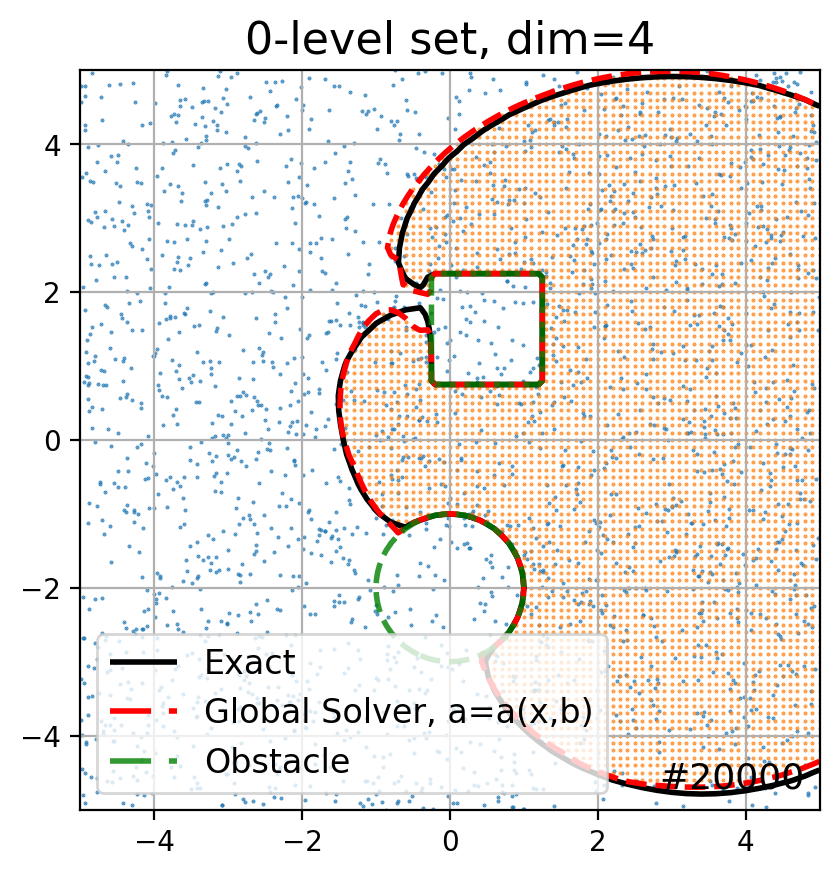}
  \hspace*{-2ex}
  \renewcommand{\rdir}{ex4-1/} 
  \renewcommand{\name}{fig__surf_G2AB_D4_Nit4_TStep1_Arelu_Neu20_Lay3_num_epochExt200_Order20_Batch20000_POTE20_ADAM_io0200.png}
  \includegraphics[height=\figw,align=c]{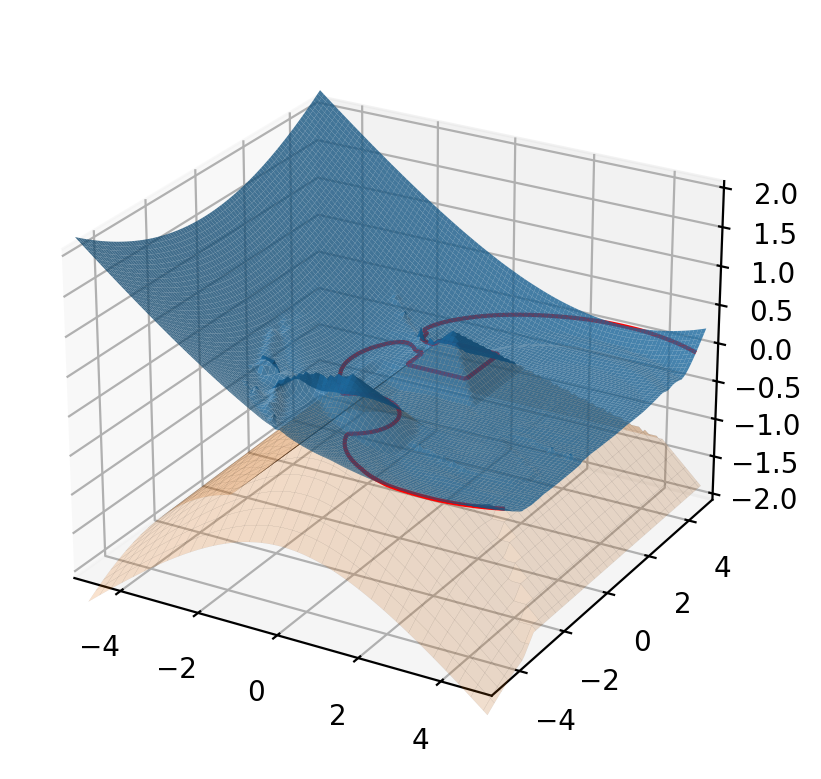}
  %
  %
  
  \renewcommand{\figw}{0.30\linewidth}
  \renewcommand{\rdir}{ex4-2/} 
  \footnotesize{$N=8$}\hspace*{-2ex}
  %
  %
  \renewcommand{\rdir}{ex4-new-2/}  
  \renewcommand{\name}{show__lev_G2AB_D4_Nit8_TStep0_Arelu_Neu40_Lay3_num_epochExt2000_Order20_Batch50000_POTE5_ADAM_io1040.png}
  \includegraphics[height=\figw,align=c]{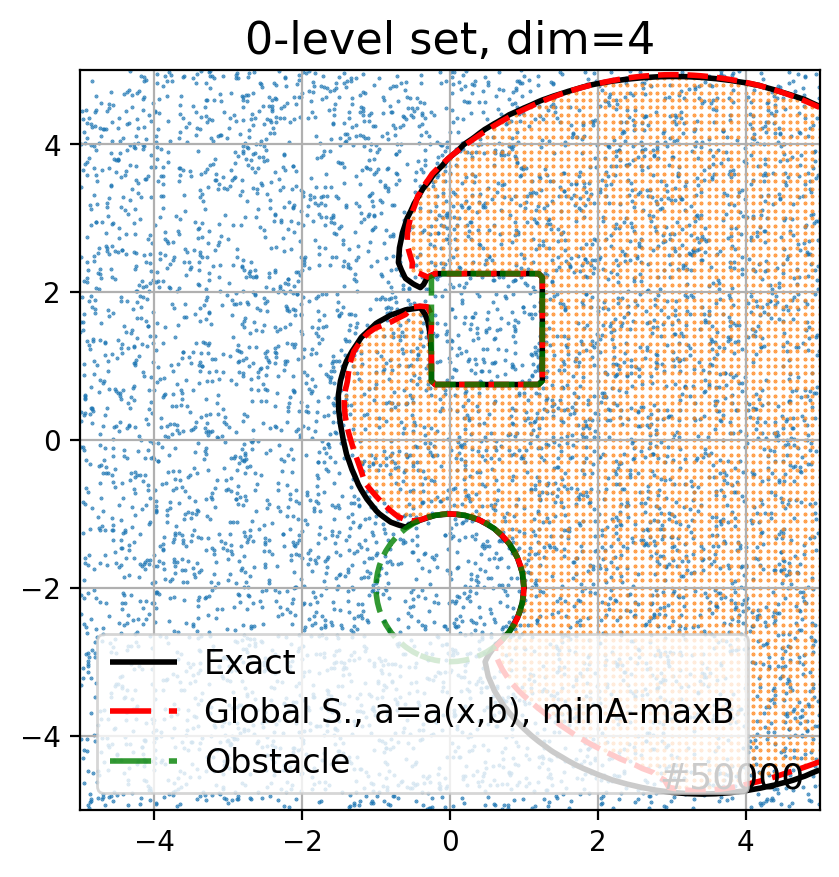}
  \hspace*{-2ex}
  \renewcommand{\rdir}{ex4-2/} 
  \renewcommand{\name}{show__lev_G2AB_D4_Nit8_TStep0_Arelu_Neu20_Lay3_num_epochExt200_Order20_Batch20000_POTE20_ADAM_io0200.png}
  \includegraphics[height=\figw,align=c]{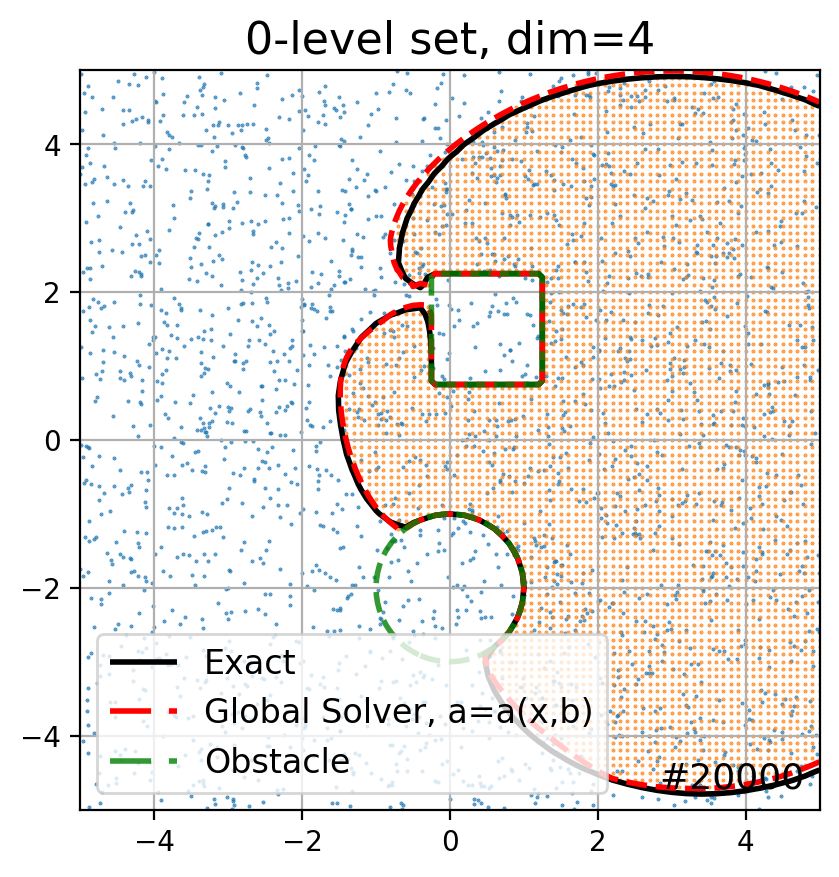}
  \hspace*{-2ex}
  \renewcommand{\rdir}{ex4-2/} 
  \renewcommand{\name}{fig__surf_G2AB_D4_Nit8_TStep0_Arelu_Neu20_Lay3_num_epochExt200_Order20_Batch20000_POTE20_ADAM_io0200.png}
  \includegraphics[height=\figw,align=c]{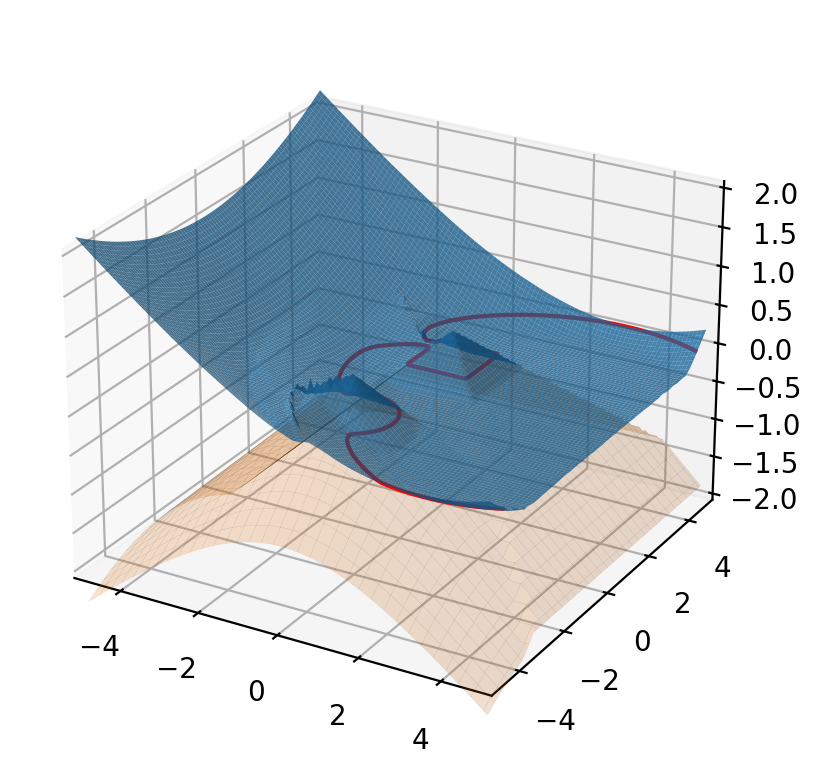}
  }
  \caption{
    \small
    (Example 4, $d=4$)
    cut in plane $(\bx_3,\bx_4)=(0,-2)$ for the adverse starting position.
    Left : "Global" scheme with $N=2,4,8$ time iterations for approach $\inf_{\ma}\sup_{b}$  
    ($0$-level sets). Center and right: "Global" scheme for approach $\sup_b\inf_{\ma}$ 
    and $N=2,4,8$ ($0$-level sets and value).
  }
  \label{fig:ex4-1}
\end{figure}

\medskip\noindent
{\bf Conclusion.} 
We have demonstrated the relevance of our approximations through several numerical tests, 
complemented by a mathematical framework for convergence analysis. 
It is noteworthy that the quality of convergence may still depend on the method employed to address the min-max,
a distinct matter not directly addressed in the present work. 
We aim to explore more complex problems using the methodology outlined in this study.

%% file: app1.tex
\section{Proof of Theorem~\ref{th:1}} \label{app:1}


Let us begin with the first equality.
To simplify, we consider the case $N=2$, the general case $N\geq 1$ being similar.

In order to prove that $V_0(x)=\barV_0(x)$, let us first show $V_0(x)\geq \barV_0(x)$.
By using the general fact that $\inf\limits_p \sup\limits_q Q(p,q) \geq \sup\limits_q \inf\limits_p Q(p,q)$, we 
have
\beno
  V_0(x)
     & =     & \inf_{\ma_0[.]} \inf_{\ma_1[.,.]} \sup_{b_0} \sup_{b_1} \varphi(X_{2,x}^{\ma[b],b}) \\
     & \geq  & \inf_{\ma_0[.]} \sup_{b_0} \sup_{b_1} \inf_{\ma_1[.,.]} \varphi(X_{2,x}^{\ma[b],b})
\eeno
where $\ma_0[.]$ denotes any function of  $A^B\equiv \cF(B,A)$, and $\ma_1[.,.]$ any function of 
$A^{B\times B}\equiv \cF(B\times B, A)$.
But $X_{2,x}^{\ma[b],b} = F(x_1,\ma_1[b_0,b_1],b_1)$ where $x_1=X_{1,x}^{\ma[b],b}$ depends only of $\ma_0[b_0]$ and $b_0$.
Therefore it is easy to see that $\inf_{\ma_1[.,.]} \varphi(X_{2,x}^{\ma[b],b}) 
  = \inf_{\ma_1[.,.]} \varphi(F(x_1,\ma_1[b_0,b_1],b_1) 
  = \inf_{a_1\in A} \varphi(F(x_1,a_1,b_1))$, which leads to
\beno
  V_0(x)
     & \geq  & \inf_{\ma_0[.]} \sup_{b_0} \sup_{b_1} \inf_{a_1} \varphi(X_{2,x}^{(\ma_0[b_0],a_1),(b_0,b_1)}). 
\eeno
Now recall that 
$\inf_{\ma_0 \in B^A} \sup_{b_0\in B} Q(\ma_0[b_0],b_0) = \sup_{b_0\in B} \inf_{a_0\in A} Q(a_0,b_0)$
  (see Lemma~\ref{lem:a-b-inversion}), hence
we can exchange $\inf_{\ma_0}$ and $\sup_{b_0}$ to obtain
\beno
  V_0(x)
     & \geq  & \sup_{b_0} \inf_{a_0} \sup_{b_1} \inf_{a_1} \varphi(X_{2,x}^{(\ma_0[b_0],a_1),(b_0,b_1)}) \equiv \barV_0(x).
\eeno

To prove the reverse inequality, let us define $\bar\ma_0\in A^B$ (where $A^B=\cF(B,A)$) and $\bar\ma_1\in A^{B^2}\equiv \cF(B^2,A)$ such that:
\be \label{eq:ma0-star-b}
  \bar\ma_0[b_0] \in \arginf_{a_0\in A} \bigg( \sup_{b_1} \inf_{a_1} \varphi(X_{2,x}^{(a_0,a_1),(b_0,b_1)}) \bigg)
\ee
and 
\be \label{eq:ma1-star-b}
  \bar\ma_1[b_0,b_1] \in \arginf_{a_1\in A} \varphi(X_{2,x}^{(\bar\ma_0[b_0],a_1),(b_0,b_1)}).
\ee
Then, by using the definition of $V_0$, $\bar\ma_1$ and $\bar\ma_0$, we obtain
\beno
  V_0(x)
    &  \leq &  \sup_{b_0} \sup_{b_1} \varphi(X_{2,x}^{\bar\ma[b],b}) \\
    &   =   &  \sup_{b_0} \bigg( \sup_{b_1} \inf_{a_1}  \varphi(X_{2,x}^{(\bar\ma_0[b_0],a_1),(b_0,b_1)}) \bigg) \\
    &   =   &  \sup_{b_0} \inf_{a_0} \sup_{b_1} \inf_{a_1}  \varphi(X_{2,x}^{(a_0,a_1),(b_0,b_1)}) \equiv \barV_0(x).
\eeno
Hence the desired result.

  Notice that for the general case $N\geq 1$,  for a given $x\in \R^d$, 
  an optimal non-anticipative strategy $\bar\ma=(\bar\ma_k)\in S_N$ can be obtained by choosing first 
  $$ \bar\ma_0[b_0]\in \argmin_{a_0\in A} g(x) \vee V_1(F(x,a_0,b_0)), $$
  then by choosing 
  $$ \bar\ma_1[b_0,b_1] \in \argmin_{a_1\in A} g(\bar x_1) \vee V_2(F(\bar x_1,a_1,b_1)) $$
  where we use the notation $\bar x_1 = F(x,\bar\ma_0[b_0],b_0)$, and so on, 
  choosing at any step $k\leq N-1$ :
  $$ \bar\ma_k[b_0,\dots,b_k] \in \argmin_{a_k\in A} g(\bar x_k) \vee V_{k+1}(F(\bar x_k,a_k,b_k)) $$
  where we use the recursive notation $\bar x_{k} := F(\bar x_{k-1},\bar\ma_{k-1}[b_0,\dots,b_{k-1}],b_{k-1})$, $k\geq 0$ and $\bar x_0:=x$.
  
\medskip

Now let us turn to the second equality between $\bbarV_0(x)$ and $\barV_0(x)$. We first prove $\bbarV_0(x) \geq \barV_0(x)$ in the case $N=2$.
By using the general fact that $\inf\limits_p \sup\limits_q Q(p,q) \geq \sup\limits_q \inf\limits_p Q(p,q)$, we 
obtain
\beno
  \bbarV_0(x)
     & =     & \inf_{\ma_0[.,.]\in \mG} \inf_{\ma_1[.,.]\in\mG} \sup_{b_0} \sup_{b_1} \varphi(X_{2,x}^{\ma[b],b}) \\
     & \geq  & \inf_{\ma_0[.,.]} \sup_{b_0} \sup_{b_1} \inf_{\ma_1[.,.]\in\mG} \varphi(X_{2,x}^{\ma[b],b})\\
     & \geq  & \inf_{\ma_0[.,.]} \sup_{b_0} \sup_{b_1} \inf_{a_1} \varphi(X_{2,x}^{(\ma_0[x,b_0],a_1),(b_0,b_1)}).
\eeno
For the last term we have used the fact that, for any $b=(b_0,b_1)$,
$\inf\limits_{\ma_1[.,.]\in\mG} \varphi(X_{2,x}^{\ma[b],b}) = \inf\limits_{a_1} \varphi(X_{2,x}^{(\ma_0[x,b_0],a_1),(b_0,b_1)})$.
In the same way, as in Lemma~\ref{lem:a-b-inversion},
we have now
\beno
  \bbarV_0(x)
     & \geq  & \sup_{b_0} \inf_{a_0} \sup_{b_1} \inf_{a_1} \varphi(X_{2,x}^{(a_0,a_1),(b_0,b_1)})
      \ =\ \barV_0(x).
\eeno
The general case when $N\geq 1$ can be proved in the same way.


In order to prove the reverse inequality, 
let $\ma_k^*$ be a function of $\mG$ such that
\be \label{eq:mak-star-bb}
 \ma_k^*(x,b) \in \arginf_{a\in A} g(x) \vee V_{k+1}(F(x,a,b)) 
\ee
Recall the DPP (for $V_k$ or $\barV_k$) :
$V_{k}(x)=\sup_{b} \inf_{a} g(x) \vee V_{k+1}(F(x,a,b))$.
Then $V_k(x) = \sup_{b_k} g(x) \vee V_{k+1}(F(x,\ma_k^*(x,b_k),b_k))$.

Hence we deduce, denoting $x_1^*=F(x,\ma_0^*(x,b_0),b_0))$, $x_2^*=F(x_1^*,\ma_1^*(x_1^*,b_1),b_1)$, and so on,
in the expressions below:
\beno
  \bar V_0(x) 
    &  =  &  \sup_{b_0} g(x) \vee  \bar V_1(F(x,\ma_0^*(x,b_0),b_0))\ \equiv\ \sup_{b_0} g(x) \vee \bar V_1(x_1^*)\\
    &  =  &  \sup_{b_0} \sup_{b_1} g(x) \vee g(x_1^*) \vee \barV_2(x_2^*) \\
    & \dots & =\  \sup_{b_0}\sup_{b_1}\cdots\sup_{b_{N-1}}  g(x)\vee g(x_1^*) \vee \dots 
      \vee g(x_{N-1}^*) \vee \varphi(x_N^*).
\eeno
In particular, $\ma^*:=(\ma_0^*,\ma_1^*,\dots)$ is an optimal element of $\mG^N$ for $\bbarV_0$, 
in the sense that it reaches the infimum in \eqref{eq:def-3}.
Hence we conclude to the desired result.

%% file: app2.tex
\section{Proof of Theorem~\ref{th:error-v0}} \label{app:2}

In all this section the assumptions of Theorem~\ref{th:error-v0} holds on the dynamics, in particular
$f(x,a,b)=f_1(x,a)+f_2(x,b)$.
  The first result consists in comparing continuous trajectories (of the form $y_{0,x}^{\ma[b],b]}(t)$) 
  with discrete Euler scheme trajectories (of the form $X_{k,x}^{\ma^d[b^d],b^d}$), with an error of order $O(\dt)$.

  From these estimates, the proof of Theorem~\ref{th:error-v0} will follow.


\begin{prop}\label{prop:B1}
  Let $T>0$, $N\in \N^*$ and $\tau=\frac{T}{N}$. There exists $C_T>0$, depending only on $T$, such that:\\
  $(i)$  for any $\ma\in \mGT$ non-anticipative strategy, there exists $\ma^d \in S_N$
  (a "discrete" non-anticipative strategy) such that 
  \be\label{eq:B1-1} 
     \forall b^d\in B^N,\ \exists b\in \cB_T, \quad
     \sup_{k\in\llbracket 0,N\rrbracket}
       \| y_{0,x}^{\ma[b],b}(t_k) -  X_{k,x}^{\ma^d[b^d],b^d}\|\leq C_T (1+\|x\|)\dt;
  \ee 
  $(ii)$ for any $\ma^d\in S_N$, there exists a non-anticipative 
  strategy $\ma\in\mGT$ such that
  \be\label{eq:B1-2} 
   \forall b\in \cB_T,\ \exists b^d\in B^N, \quad
     \sup_{k\in\llbracket 0,N\rrbracket}
       \| y_{0,x}^{\ma[b],b}(t_k) -  X_{k,x}^{\ma^d[b^d],b^d}\|\leq C_T (1+\|x\|)\dt.
  \ee
\end{prop}
\begin{proof}
  $(i)$ Let $\ma$ be in $\mGT$. We construct $\ma^d=(\ma_0,\dots,\ma_{N-1}):B^N\conv A^N$ as follows:
  for any $b^d=(b_0,\dots,b_{N-1})$ in $B^N$, 
  we consider the piecewise constant control $b$ such that $b(t):=b_k$ on $[t_k,t_{k+1}[$,
  and we set, for all $k$, $\ma^d_k(b^d):=a_k\in A$ such that
  $$ 
    \frac{1}{\dt}\int_{t_k}^{t_{k+1}} f_1(y^{\ma[b],b|}(t_k),\ma[b](s)) ds = f_1(y^{\ma[b],b}(t_k), a_k).
  $$
  This is possible since we assume that $f_1(x,A)$ is convex for all $x$ (see e.g.~\cite{Aubin_Cellina}).

  Let us first check that $\ma^d\in S_N$.

  Now let us denote $y_k:=y^{\ma[b],b}_{0,x}(t_k)$ and $x_k :=X_{k,x}^{\ma^d[b^d],b^d}$.
  In order to prove the desired bound, we need to establish, for some constant $C\geq 0$,
  $$
     \sup_{0\leq k\leq N} \|y_k-x_k\| \leq C(1+\|x\|)\dt.
  $$
  Notice that 
  \beno
    \frac{1}{\dt}\int_{t_k}^{t_{k+1}} f(y_k,\ma[b](s),b(s)) ds  
     &  = &  f_1(y_k, a_k) + \frac{1}{\dt}\int_{t_k}^{t_{k+1}} f_2(y_k,b(s)) ds \\
     &  = &  f_1(y_k, a_k) + f_2(y_k,b_k) \ = \ f(y_k,a_k,b_k)
  \eeno 
  Hence, with $y(t):=y^{\ma[b],b}_{0,x}(t)$, it holds $y_k=y(t_k)$ and
  \be
     y_{k+1}
        & = & y_k + \int_{t_k}^{t_{k+1}} f(y(s),\ma[b](s),b(s))ds \nonumber\\
        & = & y_k + \int_{t_k}^{t_{k+1}} f(y_k,\ma[b](s),b(s))ds  + \eps_k \label{eq:yk0}
  \ee
  where $|\eps_k|\leq \dt L \max_{[t_k,t_{k+1}]} |y(s)-y(t_k)|$, with $L=[f]_1$.
  We have also, by construction of $a_k$ and of $b(.)$:
  \be
     y_{k+1} & = & y_k + \tau f(y_k,a_k,b_k)  + \eps_k \label{eq:yk}.
  \ee
  Furthermore by standard Gronwall estimates, with $|f(x,a,b)|\leq L|x| + C$ where $C=\max_{a,b} |f(0,a,b)|$,
  we have $|y(t)|\leq (TC + |x|)e^{LT}$, and then
  \be \label{eq:bound-y-ytk}
     \sup_{s\in[t_k,t_{k+1}]} |y(s)-y(t_k)|\leq \dt (L(TC + |x|)e^{LT} + C)\equiv C_2 \dt
  \ee
  (where $C_2$ depends linearly on $|x|$).
  Hence 
    $\eps_k\leq C_2L \tau^2$ for all $0\leq k\leq N-1$.
  On the other hand we have
  \be
     x_{k+1} = x_k + \dt f(x_k,a_k,b_k). \label{eq:xk}
  \ee
  Using \eqref{eq:yk} and \eqref{eq:xk} we deduce $|y_{k+1}-x_{k+1}|\leq |y_k-x_k| (1+L\dt) + |\eps_k|$.
  By using a discrete Gronwall estimate, and since $e_0=0$ and $N\tau =T$, we deduce
  $\max_{0\leq k\leq N} |y_k-x_k| \leq (1+ L\dt)^N \sum_{0\leq k\leq N-1} |\eps_k| \leq e^{LT} C_2 L T \tau$.

  This concludes to the desired bound.

  $(ii)$ Conversely,  let $x\in \R^d$ and let $\ma^d\in S_N$ be a discrete non-anticipative strategy. 
  We aim to define some $\ma\in \mGT$ associated to $\ma^d$. 
  Let $b\in \cB_T$ be a given measurable control. To this control $b$, we are going to associate a control $b^d=(b_0^d,\dots,b_{N-1}^d)\in  B^N$,
  a piecewise constant control $\ma[b]$ as well
  as a set of points $y_k = y_{0,x}^{\ma[b],b}(t_k)$ as follows.
  For $k=0$, let $y_0:=x$. 
  From the values $b|_{[0,t_1[}$ and by using the convexity of $f_2(y_0,B)$, we can find $b^d_0\in B$ such that 
  $$ \frac{1}{\dt}\int_{0}^{t_1} f_2(y_0,b(s))ds = f_2(x,b^d_0). $$
  Then we set 
  $$
    \ma[b](s):=\ma^d_0[b_0],\quad \forall s\in[0,t_1[.
  $$
  Now we construct $\ma[b]$ in a recursive way. Assume that $y_k=y_{0,x}^{\ma[b],b}(t_k)$  is known, using the knowledge of 
  $b(s)$ and of $\ma[b](s)$ for $s\in[0,t_k[$. 
  By using the convexity of $f_2(y_k,B)$, there exists $b^d_k\in B$ such that 
  $$ 
    \frac{1}{\dt}\int_{t_k}^{t_{k+1}} f_2(y_k,b(s))ds = f_2(x,b^d_k). 
  $$
  Then we set
  $$
    \ma[b](s):=\ma^d_k[b^d_0,\dots,b^d_k],\quad \forall s\in[t_k,t_{k+1}[.
  $$
  This allows to define $y_{k+1}=y_{0,x}^{\ma[b],b}(t_{k+1})$.
  We can check also that the contructed $\ma$ is non-anticipative. 
  Finally, denoting $y(t):=y_{0,x}^{\ma[b],b}(t)$, we have 
  \beno
    y_{k+1} & = & y_k+ \int_{t_k}^{t_{k+1}} f(y(s),\ma[b](s),b(s))ds \\
         & = &  y_k+\int_{t_k}^{t_{k+1}} f(y_k,\ma[b](s),b(s))ds + \eps_k \\
         & = &  y_k+  \tau f(y_k,\ma^d_k[b],b^d_k) + \eps_k
  \eeno
  by construction and where
  where $ \|\eps_k\| \leq C_2 L \tau$ as in the proof of $(i)$.

  On the other hand, by the definitions of $x_k:=X_{k,x}^{\ma^d[b^d],b^d}$, we have
  $$
    x_{k+1}=x_k + \dt f(x_k,\ma^d_k[b],b^d_k)
  $$
  Hence we obtain a similar error estimate as in $(i)$ and this concludes the proof of $(ii)$.
\end{proof}

\medskip

\begin{proof}[Proof of Theorem~\ref{th:error-v0}]
  Notice first that (by using estimate \eqref{eq:bound-y-ytk}) we have
  \beno
   & & \hspace{-3cm}
     |\max_{s\in(0,T)} g(y(s))\vee \varphi(y(T)) - \max_{0\leq k\leq N-1} g(y(t_k))\vee\varphi(y(T))| \\
   & \leq &  \max_{0\leq k\leq N-1} \bigg|\max_{s\in[t_k,t_{k+1}]} g(y(s)) - g(y(t_k))\bigg|  \ \leq \ [g] C_2\tau.
  \eeno


  Let $\Phi(x_0,\dots,x_N):=\max_{0\leq k\leq N-1} g(x_k) \vee \varphi(x_N)$. 
  From the definition of $v_0(x)$ and the previous inequality we have 
  $$ 
    \bigg| v_0(x) - \inf_{\ma\in \mGT} \sup_{b\in \cB_T} \Phi((y_{0,x}^{\ma[b],b}(t_k))_{k\ge 0})\bigg|\leq [g] C_2 \tau.
  $$
  Let $\eps>0$. There exists $\ma\in \mGT$ such that 
  \begin{align}
    \sup_{b\in \cB_T} \Phi((y_{0,x}^{\ma[b],b}(t_k))) \leq v_0(x) +  [g] C_2 \tau + \eps.
    \label{eq:boundDis}
  \end{align}
  By Proposition~\ref{prop:B1}$(i)$ there exists $\ma^d\in S_N$ such  that \eqref{eq:B1-1} holds.
  This implies in particular that
  \beno 
     \forall b^d\in B^N,\ \exists b\in \cB_T,\ 
     \bigg|\Phi((y_{0,x}^{\ma[b],b}(t_k))) - \Phi((X_{k,x}^{\ma^d[b^d],b^d}))\bigg| \leq 
     ([g]\vee[\varphi])\,C_T(1+\|x\|) \tau.
  \eeno
  Therefore also
  \beno 
     \sup_{b^d\in B^N}\Phi((X_{k,x}^{\ma^d[b^d],b^d})) 
     \leq \sup_{b\in \cB_T} \Phi((y_{0,x}^{\ma[b],b}(t_k)))  + ([g]\vee[\varphi])\,C_T(1+\|x\|) \tau,
  \eeno
  from which we deduce, using \eqref{eq:boundDis}
  $$ V_0(x)\leq v_0(x) + C_2 L \tau + ([g]\vee[\varphi])\,C_T(1+\|x\|) \tau +\eps.
  $$
  Letting $\eps\downarrow 0$, this gives the upper bound for $V_0(x)-v_0(x)$ in the desired form. 
  The lower bound is obtained in a similar way from Proposition~\ref{prop:B1}$(ii)$.
\end{proof}

%% file: app-vex.tex
\section{Analytical formula for some examples} \label{app:vex}

\newcommand{\xmin}{x_{\min}}
\newcommand{\xmax}{x_{\max}}
\newcommand{\ymin}{y_{\min}}
\newcommand{\ymax}{y_{\max}}

\medskip

\noindent
{\bf Data and analytical formula for Example~1 of Section~\ref{sec:num}.}
In this example we consider two constants $x_{A,1}:=2$ and $r_A=1.2$ (the initial data is somehow centered at $x_A:=(2,0)$ 
and has a radius $r_A$ in polar coordinates).
Let $atan2(x_2,x_1)$ denote the angle $\mt\in ]-\!\pi,\pi]$ of the point $x=(x_1,x_2)$.

We first define the following function $\bar u_1(t,x)$, for any $t>0$ and $x\in \R^2$:
  $$
    \bar u_1(t,x):= \min(\overline{\eps},\max(\underline{\eps}, \bar u_2(t,r_1(x), \mt_p(t,x))) )
  $$ 
  where 
  $r_1(x):=\cos(\mt_p) x_1 - \sin(\mt_p) x_2$,
  $\mt_p(t,x):=\min(\max(atan2(x_2,x_1),-t),t)$,
  and 
  $\bar u_2(t,r,\mt):=\max(|r-x_{A,1}|+ b t, 2\pi|\mt-\mt_p)|)-r_A$.

Let $t_0:=0.25$. We define the terminal cost $\varphi$ as follows:
$$
   \varphi(x):= \bar u(t_0,x).
$$
Then one can see that the following analytical formula holds:
$$
    v(t,x):= \bar u(T-t+t_0,x).
$$

\medskip

\noindent
%

\medskip

\noindent
{\bf Analytical formula for Example~2 of Section~\ref{sec:num}.}
It is given by
\beno
  & & \xmin=-1,\ \xmax=1,\ \ymin=-1,\ \ymax=1\\
  & & r_1=\max(\xmin-2t-x,x-(\xmax+2t))\\
  & & r_2=\max(\ymin-y, y-\ymax)\\
  & & x_1=1-2t,\   y_1=1\\
  & & x_2=1+2t,\   y_2=1-2t\\
  & & x_3=1.5+2t,\ y_3=1.5-2t,\  z_3=0.5\\
  & & p = z_3 / ((x_2-x_1)(y_3-y_1)-(y_2-y_1)(x_3-x_1))\\
  & & r_3=p\ ((x_2-x_1)( y-y_1)-(y_2-y_1)( x-x_1))\\
  & & r_4=p\ ((x_2-x_1)(-y-y_1)-(y_2-y_1)(-x-x_1)) \\
  & & r_5=\max(r_1,r_2,r_3,r_4)
\eeno
and
$$
  v(t,x):=\min(0.5,\max(-0.5,r_5)).
$$

\medskip

\noindent
{\bf Analytical formula for Example~3 of Section~\ref{sec:num}.}
We found the following formula, similar to the previous case, adapted with the following values.
First define, for the "$\disp \min_{\ma}\max_b$" problem : 
$$ \bar r_1^{-}=\max(\xmin-2t-x, x-\xmax)
$$
or, for the "$\disp \max_{\ma}\min_b$" problem : 
$$ \bar r_1^{+}=\max(\xmin-x, x-(\xmax-2t)).
$$
%
%
Then define
\beno
  & & \bar r_3=p\ ((x_2-x_1)( y-y_1)-(y_2-y_1)(-x-x_1))\\
  & & \bar r_5^{\pm}=\max(\bar r_1^{\pm},r_2,\bar r_3,r_4)
\eeno
(with same $p$ value as before),
and
$$
  v^{\pm}(t,x):=\min(0.5,\max(-0.5,\bar r_5^{\pm})).
$$
Then $v_0^\pm(x)\equiv v^\pm(0,x)$. Hence $V_0^{\pm}$ is an approximation of $v^{\pm}(0,x)$.


%% file: sec-declarations.tex
\bigskip

\medskip\noindent
{\bf\large Declarations:}
 
\medskip\noindent
{\em Funding.}
This research benefited from the support of the 
FMJH Program PGMO and from the support to this program from EDF.
 
\medskip\noindent
{\em Ethical Approval:}
not applicable
 
\medskip\noindent
{\em Competing interests:} there is no conflict of interest neitheir financial nor personal.
 
\medskip\noindent
{\em Authors' contributions:} each author contributed equally to all parts of the paper.
 
\medskip\noindent
{\em Data availability statement:}
no data associated in the manuscript.
 

%% file: sec-step.tex
\paragraph{Multi-step approximations}
\olivier{Je pense qu'on va pouvoir finalement sans passer, cad qu'il ne semble plus necessaire d'utiliser 
de changt de variable dans les preuves de convergence.}
\olivier{On pourra quand meme le garder pour les schémas si besoin...}

Let $p\in \N$, $p\geq 1$, and $h=\frac{\dt}{p}$ be a given substep.
We introduce also the notation $t_n=n\dt$.  


\medskip

Let us consider $F_h:\R^d\times A \times B \conv\R^d$ (for a given $h>0$),
corresponding to a one time step approximation of $y^a_{x}(h)$ (starting from $y^a_x(0)=x$)
with time step size $h>0$.
For a simplifed setting, 
we may consider the Euler scheme $F_h(x,a,b) = x + h f(x,a,b)$, or the 
Heun scheme  $F_h(x,a,b) = x + \frac{h}{2} (f(x,a,b) + f(x+ h f(x,a,b),a,b))$, and so on.
General Runge Kutta schemes
will be considered later on.
Further assumptions on $F_h$ will be eventually needed.

Then $F(x,a,b)$ is defined as $p$ iterations of $F_h$ starting from $x$ and with fixed controls
$a$ and $b$, that is $F(x,a,b)=y_p$ where $y_{k+1}= F_h(y_k,a,b)$ (for $k=0,\dots,p-1$) and $y_0=x$.
In the case $(a,b)\in \cA\times \cB$ are measurable functions, we extend the previous definition as follows:
   $F(x,a,b):=y_p$ where  $y_{k+1}= F_h(y_k,a(x),b(x))$ ($0\leq k\leq p-1$) and $y_0=x$
   (i.e., same definition with controls $a(x)$ and $b(x)$ fixed at the foot of the characteristic).

For given sequences
$(a,b)\in \cA^N\times \cB^N$
where $a=(a_0,\dots,a_{N-1})$ and $b=(b_0,\dots,b_{N-1})$, we denote
\beno
   Y_{k+1,x}^{a,b} = F_h(Y_{k,x}^{a,b},a_k(x),b_k(x)), \quad k=0,\dots,p-1, \quad Y_{0,x}^{a,b}=x.
\eeno
We also define recursively 
$X_{j,x}^{a,b}$ for $j\geq 0$  by 
\beno
   & & \mbox{$X^{a,b}_{j+1,x}=F(X^{a,b}_j,a_j,b_j)$, for $j\geq 0$} \\
   & & \mbox{$X_{0,x}^{a,b}=x$.} 
\eeno
(This means that we apply controls $a_j(X^{a,b}_{j,x})$ and $b_j(X^{a,b}_{j,x})$ at step $j$.)
Notice then that $X_{1,x}^{a,b}=Y_{p,x}^{a,b}$ as well as $X_{n+1,x}^{a,b}=Y_{p,X_{n,x}^{a,b}}^{a,b}$, $\forall n\geq 0$.

For given feedback controls $(a,b)\in \cA\times \cB$, we define the functional $G^{a,b}(x)$ by
$$
  G^{a,b}(x) := \max_{0\leq k\leq p-1} g(Y_{k,x}^{a,b}).
$$

A general semi-discrete problem, for a given $N\geq 1$,
is to optimize the following functional
over all sequence of controls $(a,b)\in \cA^N\times \cB^N$
where $a=(a_0,\dots,a_{N-1})$ and $b=(b_0,\dots,b_{N-1})$:
\beno
  J_0(x,a,b) & := &
  G^{a_0,b_0}(x) \bigvee G^{a_1,b_1}(X^{a,b}_{x,1}) \bigvee 
  \ \cdots  \
  \bigvee G^{a_{N-1},b_{N-1}}(X^{a,b}_{N-1,x}) 
  \bigvee (g\vee \varphi)(X^{a,b}_{N,x}).
\eeno
Fully discrete schemes will be obtained by using different ways to approximate the controls.

\paragraph{Intermediate values.}
Let us denote in the same way, for controls $a=(a_k,\dots,a_{N-1})$ and $b=(b_k,\dots,b_{N-1})$ and $x\in \R^d$
\beno
  J_k(x,a,b) & := & 
  \bigg(\max_{j=k,\dots,N-1} G^{a_j,b_j}(X^{a,b}_{j,x})\bigg)
  \bigvee (g\vee \varphi)(X^{a,b}_{N-k,x})
\eeno
and 
$$
  V_k(x):=\min_{a_k}\max_{b_k}\ \ \cdots \ \ \min_{a_{N-1}}\max_{b_{N-1}} J_k(x,a,b).
$$
Then the following DPP holds, for all $0\leq k\leq N-1$ and $x\in \R^d$:
\be\label{eq:dpp-vk}
  V_k(x):=\min_{a\in A }\max_{b \in B} \bigg( G^{a,b}(x) \bigvee V_{k+1}(F(x,a,b)) \bigg)
\ee

%% file: app-dynprog.tex
\section{old notes - Dyn Prog. Principle}

- HYPOTHESES PLUS COMPLIQUEES POUR LE CHANGMENET DE VARIABLE (non utilisé finalement)

- notes sur le DPP

\noindent{\textbf \hypfive}
\begin{em}
    The functions $\rho_k\in L^1(\R^d)$ and open sets $\mO_k\subset\R^d$, $0\leq k\leq N$, are such that
\begin{subequations} 
\be
  & &
    \mbox{$\rho_k(x)>0$ on $\mO_k$, and $\supp(\rho_k)=\overline{\mO_k}$,  $\forall k=0,\dots,N$}
    \label{eq:H5-a}
    \\[1ex]
  & &
    \mbox{$F(\mO_k,a,b)\subset \mO_{k+1}$, $\forall a\in A$, $\forall b\in B$, 
    $\forall k=0,\dots,N-1$} 
    \label{eq:H5-b}
    \\
  & &
    C_{k,\dt}:= \sup_{x\in \mO_k} \sup_{a\in A} \sup_{b\in B} \frac{\rho_k(x)}{\rho_{k+1}(F(x,a,b))}<\infty,
    \ \forall k=0,\dots,N-1.
    \label{eq:H5-c}
\ee
\end{subequations} 
  Furthermore, 
  we consider random variables $(X_k)_{0\leq k\leq N}$ on some probability space, with values in $\R^d$, 
  absolutely continuous with respect to Lebesgue's measure and admitting 
  $(\rho_k)_{0\leq k\leq N}$ as associated densities.
\end{em}

From the definitions we have $\E[\phi(X_k)] = \int_{\mO_k} \phi(x) \rho_k(x) dx$ for any measurable bounded function $\phi$.

In order to deal with the DPP \eqref{eq:dpp} we can consider an equivalent reformulation using feedback controls:

\begin{prop}
  Let $n\in \llbracket 0,N-1\rrbracket$ and $(\mO_n,\rho_n)$ as in \hypfive\ (with associated random variables $X_n$). 
  Then
  the following dynamic programming principle holds:
  for any  $(\ba_n,\bb_n)\in \cA\times \cB$ such that
  \be\label{eq:DPP-2-an}
    \ba_n(.) \in \argmin_{a\in \cA} \max_{b\in \cB} \E\bigg[G^{a,b}(X_n) \vee V_{n+1} (F^{a,b}(X_n))\bigg]
  \ee
  and 
  \be\label{eq:DPP-2-bn}
    \bb_n(.) \in \argmax_{b\in \cB} \E\bigg[G^{\ba_n,b}(X_n) \vee V_{n+1} (F^{\ba_n,b}(X_n))\bigg]
  \ee
  it holds
  \be\label{eq:DPP-2-vn} 
    V_n(x) =  G^{\ba_n,\bb_n}(x) \bigvee V_{n+1}(F^{\ba_n,\bb_n}(x)), \quad \forall x\in \mO_n.
  \ee
  In particular, we have 
  \be \label{eq:dpp-useful}
     \E[V_n(X_n)] =  \E\bigg[ G^{\ba_n,\bb_n}(X_n) \bigvee V_{n+1}(F^{\ba_n,\bb_n}(X_n))\bigg]
                  = \inf_{a\in \cA} \sup_{b\in \cB} \E\bigg[G^{a,b}(X_n) \vee V_{n+1} (F^{a,b}(X_n))\bigg].
  \ee
\end{prop}

%% file: app-disc1.tex
\section{file \texttt{app-disc1.tex}} 

{\color{blue}

{\bf 1.} There is a major difficulty in the convergence analysis. Mainly,
the fact that we no more have 
$\hV_n(x) \geq  V_n(x)$. This is because of the $\inf \sup$ formula
 instead of just $\inf$ or just $\sup$.
Here are some attemps to bypass this difficulty.

{\bf 2.} Let the following intermediary value
$\bar V_n$ be defined recursively as follows:\\
$(i)$
$$  \bar a_n \in \argmin_{a\in \numA_n} 
   \sup_{b\in \cB}
   \E[Q_{\bar V_{n+1}} (X, a(X), b(X)) ]
 $$
 where 
 $$
   Q_{V}(x,a,b):= G^{a,b}(x) \vee e^{-\ml \dt} V(F^{a,b}(x)).
 $$
 So here the $\sup$ is over $b\in\cB$ instead of $b\in \numB_n$. 
 We also denote $\bar b_n$ an optimal control in $\cB$ such that  \\
   $$ \sup_{b\in \cB}
   \E[Q_{\bar V_{n+1}} (X, \bar a_n(X), b(X)) ] =
   \E[Q_{\bar V_{n+1}} (X, \bar a_n(X), \bar b_n(X)) ].
   $$
$(ii)$ Then $\bar V_n(x)$ is defined by 
  $$
   \bar V_n(x) := Q_{V_{n+1}} (x,\bar a_n(x),\bar b_n(x)).
  $$
\begin{lem}
By the definition of
$V_n$ and $\bar V_n$ we have
$$\bar V_n(x)\geq V_n(x).$$
\end{lem}

\begin{proof}
By using the definition of $\bar b_n$, and as in Lemma~\ref{lem:min},
we have
$$
  \bar V_n(x) := \sup_{b\in B} Q_{V_{n+1}} (x,\bar a_n(x),b).
$$
Hence by recursion we obtain a formula of the type
$$ \bar V_n(x)= \sup_{(b_n,\dots,b_{N-1})\in B^{N-n}}
    J(x,(\bar a_n,\dots,\bar a_{N-1}), (b_n,\dots,b_{N-1}))
$$
(involving also functions $g$ and $\varphi$, and feedback controls 
$\bar a_k \in \numA_k$, but nothing else),
whereas 
$$ V_n(x)= \inf_{a_n\in A}\sup_{b_n\in B}\cdots 
        \inf_{a_{N-1}\in A}\sup_{b_{N-1}\in B}
    J(x,(a_n,\dots,a_{N-1}), (b_n,\dots,b_{N-1})),
$$
and the desired inequality follows.
\end{proof}

{\bf 3.} {\em Going further:}
It is then possible to bound  (as far as I see ...)
$\E[\bar V_n(X)-V_n(X)]$ in terms of 
$\E[(\bar V_{n+1}(X)- V_{n+1}(X))_+] = 
\E[\bar V_{n+1}(X)- V_{n+1}(X)]$, like in previous paper ...

The remaining problem is to compare the scheme values 
$\hV_n$ with $V_n$ or $\bar V_n$, because we do not have the inequality
$\bar V_n(x)\geq \hV_n(x)$.

\red{
{\bf 4 (XW):}
  Note
  \begin{align*}
        \inf_{a\in \numA_n} 
           \sup_{b\in \cB} Q(a,b)=  \sup_{b\in \cB} Q(\bar a_n,b) \ge \sup_{b\in \numB_n} Q(\bar a_n,b) \ge \inf_{a\in \numA_n} \sup_{b\in   \numB_n}Q(a,b)
  \end{align*}
  then if think that $\bar V_n(x)\geq \hat V_n(x)$.
  Then 
  \begin{align*}
      |V_n(x)- \hat V_n(x)| \le &  |V_n(x) -\bar V_n(x) |+ |\bar V_n(x)  - \hat V_n(x)| \\
        =  & (\bar V_n(x)  - \hat V_n(x))  +  (\bar V_n(x)  - V_n(x))
  \end{align*} 
  where the two terms are positive...
} 

{\color{black} 
{\bf 5 (OB) 29/05:}}
 Attention en général $\bar V_n(x)\ngeq \hat V_n(x)$:
 ce serait vrai si on pouvait écrire quelque chose comme
$$ \bar V_n(x)=  \inf_{a_n\in \numA_n} \sup_{b_n\in B} \cdots\ 
    J(x,(a_n,\dots,a_{N-1}), (b_n,\dots,b_{N-1}))
$$
et
$$ \hV_n(x)= \inf_{a_n\in \numA_n} \sup_{b_n\in \numB_n} \cdots\ 
    J(x,(a_n,\dots,a_{N-1}), (b_n,\dots,b_{N-1}))
$$
car comme tu l'as fait remarqué on aurait $\sup_{B} \geq \sup_{\hat B_n}$
et si on pouvait en chaque point $x$ faire des inf/sup sur un ensemble de controles dans $A\times B$ ou dans des sous-ensembles
approchés $A_n\times B_n$ (comme dans une méthode de type DF ou semi-lagrangien sur maillage...).
Malheureusement, tout ce qu'on a à disposition ce sont des définition par $\inf/\sup$ d'espérances $\E[...]$, et sur des ensembles de fonctions controles feedback.
On a éventuellement l'écriture comme si dessus mais pour des controles feedback optimaux particuliers
$$ \bar V_n(x)= 
    J(x,(\bar a_n,\dots,\bar a_{N-1}), (\bar b_n,\dots,\bar b_{N-1}))
$$
et
$$ \hV_n(x)=  
    J(x,(\ha_n,\dots,\ha_{N-1}), (\hb_n,\dots,\hb_{N-1})),
$$
et je ne vois pas comment comparer ces valeurs.

It is possible to first bound
$\E[\hV_n(X)-\bar V_n(X)]$ in terms of for instance 
$\E[(\hV_{n+1}(X)- \bar V_{n+1}(X))_+]$. But then since the last integrant is not $\geq 0$ we do not have a recursive argument for the bound.

{\bf 26/05/23:}
In order to attempt a "global" approach instead of a "local" approach (or step-by-step approach), we
would need to have something like (for the case of a 2-time-steps problem)
\be \label{eq:commut1}
   \inf_{a_0} \sup_{b_0} \inf_{a_1}\sup_{b_1} Q(...)
   =\inf_{a_0} \inf_{a_1} \sup_{b_0}\sup_{b_1} Q(...)
\ee
(for some functionnal $Q$),
so that we could consider a problem
$$ 
    \inf_{a=(a_0,a_1)} \sup_{b=(b_0,b_1)} Q(...) 
$$
by some SGDA algorithm.
(Otherwise, I see no reason why SGDA would converge
to $\inf\limits_{a_0} \sup\limits_{b_0} \inf\limits_{a_1}\sup\limits_{b_1} Q(...)$.

However, in general we do not have
\be \label{eq:commut2}
 \sup_{b_0} \inf_{a_1} Q(...)
  = \inf_{a_1} \sup_{b_0} Q(...)
\ee
which is necessary in order to show \eqref{eq:commut1} in general.
Indeed, in general it is well known that
$$
  \inf_{a} \sup_{b} \varphi(a,b) \neq \sup_{b} \inf_{a} \varphi(a,b).
$$
(To be more explicit : the cases for which the equality is true are limited
for instance : $\varphi(a,b)=\varphi(a+b)$ with $\varphi(\uparrow)$,
$\varphi(a,b)=f(a)-g(b)$ with $f,b$ convex functions. But
even for $\varphi(a,b)=\varphi(a+b)$, as soon as $\varphi$ is not increasing,
the equality is false - for instance $\varphi(a,b)=\cos(a+b)$ with $(a,b)\in [0,2\pi]^2$,
$I=\inf_a\sup_b \varphi(a,b)=1$ $>$ $J=\sup_b\inf_a\varphi(a,b)=-1$.)

\medskip

{\bf XW 28/05:} First we can suppose that  g is convex convave and using Von Neuman min max theorem we can rearrange as you did.

\blue{{\bf OB 29/05: }I am not sure, for the situation of interest $g$ is neither convex nor concave, nor a difference of two convex functions (I guess this is what you mean by convex-concave ?). 
In principle $g$ should corresponds to some level function so that $\{g\leq 0\}$
corresponds to some region of interest.
See also the example above with $\varphi(a,b)=\cos(a+b)$ which is typical, or anything of the form
$$ -\min(1,|x|-1)\quad \mbox{or}
\quad \min(1,|x|-1). $$
} 

\blue{{\bf OB 29/05: }I am not sure,
also, how to connect the $\min \max \varphi(a,b)$ pb to Von-Neumann $\min/\max$ theorem.
Assume $A$ is an $n\times k$ square matrix. Von Neumann Theorem asserts that (from Wikipedia ...)
$$
  \min_{X\in \Delta_k} \max_{Y\in \Delta_n} {}^t Y A X =
  \max_{Y\in \Delta_n}\min_{X\in \Delta_k} {}^t Y A X
$$
where $\Delta_k=\{x\in \R^k,\ x\geq 0, \sum_{i} x_i =1\}$.
So instead of considering $\min_{i}\max_{j} {}^t e_i A f_j=\min_i \max_j A_{ij}$ for some $A_{ij}=\varphi(a_i,b_j)$
(where $e_i$ are the canonical basis column vectors of $\R^k$ and $f_j$ idem for $\R^n$)
do you think about considering a relaxed value $\min_{X\in \Delta_k} \max_{Y\in \Delta_n} {}^t Y A X$ ?
}

}

%% file: app-disc2.tex
\section{file \texttt{app-disc2.tex}} 

{\bf\blue{OB/XW (September)}}
\blue{
We discuss two player games where $a$ plays first, $b$ plays second. We consider different ways to write the min-max procedure.
}


\begin{lem}
  For a given $x\in \R^d$, let 
  $$
    V_0(x) = \inf_{a_0} \sup_{b_0} \inf_{a_1} \sup_{b_1} \varphi(X_{2,x}^{a_0,b_0,a_1,b_1})
  $$
  and
  $$
    \bar V_0(x) = \sup_{b_0[.]} \sup_{b_1[.]} \inf_{a_0} \inf_{a_1} \varphi(X_{2,x}^{a_0,b_0[a_0],a_1,b_1[x_1,a_1]})
  $$
  with the notations $x_1\equiv X_{1,x}^{a,b}=F(x, a_0,b_0)$, 
  $x_2 \equiv X_{2,x}^{a,b}=F(X_{1,x}^{a,b}, a_1,b_1) \equiv 
  F( F(x,a_0,b_0), a_1,b_1))$,
  $b_0[.]:A \conv B$ and
  $b_1[.]$ denotes any function of two variables ($b_1:\R^d\times A \conv B$).
  Then 
  $$
    V_0(x)= \bar V_0(x).
  $$
\end{lem}

\newcommand{\disp}{\displaystyle}
Note here that $b=(b_0[.],b_1[.])$ is a non-anticipative (control) strategy in the sense that $b_0[.]$ depends only of $a_0$, 
and $b_1[.]$ depends only of $x_1$ and $a_1$.

Note : by convention, for any $b_1: \R^d\times A \conv B$, 
we denote indifferently $b_1[x_1,a_1]$ or $b_1[a_1](x_1)$ when $a_1\in A$ and $x_1\in \R^d$.

{\color{red}
\begin{rem}
  {\bf (OB 27/10)}
  Here this notion of non-anticipative strategy is slightly different from the classical one, 
  which states that $\mb[.]: A^N\conv B^N$ is a "discrete non-anticipative strategy" (and denote $\mb\in \bar\mG^N$)
  any $k$, $b[a]_k$ is a function only of $a_0,\dots,a_k$, or equivalently,  
  for any $0\leq k\leq N-1$, 
  $$ 
   \mbox{$\big(\forall 0\leq j\leq k$,  $a_j=\bar a_j\big)$  $\implique$
         $\big(\forall 0\leq j\leq k$,  $b[a]_j=b[\bar a]_j \big)$.
    }
  $$
  Then the discrete control problem is
  $V_0(x):=\inf_{\mb\in \bar\mG^N} \sup_{a\in A^N}  \varphi(X_{N,x}^{a,\mb[a]}).$
     
\end{rem}
}

\begin{proof}
  $x_1:=X_{1,x}^{a,b}$ (depending only on $x,a_0,b_0$) being fixed, using \eqref{eq:lem-i} 
  \be
     \inf_{a_1} \sup_{b_1} \varphi(X_{2,x}^{a_0,b_0,a_1,b_1}) 
      & = & \inf_{a_1} \sup_{b_1} \varphi(F(x_1,a_1,b_1))
            \ = \ \sup_{ b_1[.]} \inf_{a_1} \varphi(F(x_1,a_1,b_1[x_1,a_1])) \nonumber \\ 
      & = &  \sup_{b_1[.]} \inf_{a_1} \varphi(X_{2,x}^{a_0,b_0,a_1,b_1[a_1]}),
      \label{eq:interpx1}
  \ee
  where $b_1[x_1,a_1]$ is an optimal solution for $\sup\limits_{b_1} \varphi(F(x_1,a_1,b_1))$.
  Hence we get
  $$
    V_0(x) = \inf_{a_0} \sup_{b_0} \sup_{b_1[.]} \inf_{a_1} \varphi(X_{2,x}^{a_0,b_0,a_1,b_1[x_1,a_1]}).
  $$
  
  We have also, using strategies $b_0[.]$ that will depend only of $a_0$:
  $$
    V_0(x) = \sup_{b_0[.]} \inf_{a_0} \sup_{b_1[.,.]} \inf_{a_1} \varphi(X_{2,x}^{a_0,b_0[a_0],a_1,b_1[x_1,a_1]}).
  $$
  Then using the general fact that $\disp \inf_P \sup_Q  J(P,Q) \geq \sup_Q\inf_P J(P,Q)$, we have
  $$
    V_0(x) \geq \sup_{b_0[.]} \sup_{b_1[.,.]} \inf_{a_0} \inf_{a_1} \varphi(X_{2,x}^{a_0,b_0[a_0],a_1,b_1[x_1,a_1]}):= \bar V_0(x).
  $$
  Now we aim to prove the equality.
  For any given $a_0\in A$,
  let $b_0[a_0]\in B$ be the maximal value of $b_0\in B$ to reach the supremum
  \be 
    \sup_{b_0} \inf_{a_1} \sup_{b_1} \varphi(X_{2,x}^{a_0,b_0,a_1,b_1}) 
    \ \bigg(\equiv \inf_{a_1}  \sup_{b_1} \varphi(X_{2,x}^{a_0,b_0[a_0],a_1,b_1})\bigg).
    \label{eq:bigsis}
  \ee 
  In particular using these particular (non-anticipative) controls $b_0[.]$ and $b_1[.]$, and \eqref{eq:bigsis},
  we have
  \beno 
    \bar V_0(x)
    & \geq  & \inf_{a_0} \inf_{a_1} \varphi(X_{2,x}^{a_0,b_0[a_0],a_1,b_1[x_1,a_1]}) \\
    &  =    & \inf_{a_0} \inf_{a_1} \sup_{b_1} \varphi(X_{2,x}^{a_0,b_0[a_0],a_1,b_1}) \\
    &  =    & \inf_{a_0} \sup_{b_0} \inf_{a_1} \sup_{b_1} \varphi(X_{2,x}^{a_0,b_0,a_1,b_1}) := V_0(x),
  \eeno 
  which is the desired result.

\end{proof}

\begin{rem}
  \beno
    {\bar{\bar V}}_0(x)
       & :=  & \inf_{a_0} \inf_{a_1} \sup_{b_0[.]} \sup_{b_1[.]} \varphi(X_{2,x}^{a_0,b_0[a_0],a_1,b_1[a_1]})  \\
       &  =  & \inf_{a_0} \inf_{a_1} \sup_{b_0} \sup_{b_1} \varphi(X_{2,x}^{a_0,b_0,a_1,b_1}) 
  \eeno
  and therefore this value is in general different from $V_0(x)$.
\end{rem}

\paragraph
{\bf\blue{OB/XW (October 17) Generalization and algorithm.}}

{\color{blue}
\medskip\noindent
Consider the problem 
$$
  V_0(x):=\min_{a_0} \max_{b_0} \min_{a_1} \max_{b_1} \cdots \min_{a_{N-1}} \max_{b_{N-1}} \varphi(X_{N,x}^{a,b}) 
$$
where $a\equiv (a_0,\dots,a_{N-1})\in A^N$ and $b\equiv (b_0,\dots,b_{N-1})\in B^N$, and $X_{k,x}$ is defined by recursion 
by $X_{k+1,x}^{a,b}=F(X_{k,x}^{a,b},a_k,b_k)$ ($\forall 0\leq k\leq N-1$) and $X_{0,x}^{a,b}=x$.

Let $b[.]=(b_0[.],\dots,b_{N-1}[.])$ denotes the non-anticipative strategies for player $b$, 
in the sense that for all $k=0,\dots,N-1$, $b_k[.]\in \cM:=\Meas(\R^d\times A,B)$ 
$b_k[a](x_k)=b_k[a_k](x)=b_k[x,a_k]$ for all $x\in \R^d$. 
Furthermore in the recursive definition of $X_{N,x}^{a,b[a]}$ we consider that $X^{a,b[a]}_{k+1,x}=F(X_{k,x}^{a,b[a]},a_k,b_k(X_{k,x},a_k))$.
Equivalently we have $X_{N,x}^{a,b[a]}=x_N$ where 
\be
 \mbox{$x_0=x$ and $x_{k+1}=F(x_k,a_k,b_k(x_k,a_k))$, $\forall 0\leq k\leq N-1$.}\label{eq:recu-xk-ak-bk}
\ee

\begin{thm}
The following holds:
$$ 
  V_0(x) = \max_{b[\cdot]\in \cM^N} \min_{a\in A^N} \varphi(X_{N,x}^{a,b[a]}) 
$$
with the notations 
$a=(a_0,\dots,a_{N-1})$, $b[.]=(b_0[.],\dots,b_{N-1}[.])$.
\end{thm}

Following Lemma~\ref{lem:min}, we have [REMAINS TO BE CHECKED...]
\begin{lem} \label{lem:minmax}
  Let $X$ be a random variable with values in $\R^d$
  which admits a Lebesgue measurable density $\rho$ supported on $\bar\mO$ for some $\mO\subset\R^d$ 
  (such that $\rho(x)>0$ a.e. $x\in \mO$, with $\supp(\rho)=\bar\mO$),
  and such that $\E[|X|]<\infty$.
  Let a given continuous function $Q:\R^d\times A\times B\conv \R$, with Lipschitz growth
  ($\exists C\geq 0$, $\forall (x,a,b)$, $Q(x,a,b)\leq C(1+|x|)$).
  Let $F=\Meas(A,B)$ denotes the set of measurable functions from $A$ to $B$.
  It holds\\
  $(i)$ Case $N=1$:
  \be\label{eq:maxmin-I} 
    \E\bigg[\sup_{b[.]\in F} \inf_{a\in A} Q(X,a,b[a])\bigg]  
      = \sup_{b\in \cM} \inf_{a\in \cA} E\bigg[ Q(X,a(X),b(X,a(X))) \bigg]
  \ee
  where $\cM=\Meas(\R^d\times A,B)$ and $\cA=\Meas(\R^d,A)$.
  \\
  \COMMENTG{
  {\bf 27/10 (OB) A REVOIR}
  $(ii)$ Generalization ($N\geq 1$):
  \be\label{eq:maxmin-N} 
    \E\bigg[\sup_{b[.]\in F^N} \inf_{a\in A^N} Q(X,a,b[a])\bigg]  
      = \sup_{b(.)\in \cM^N} \inf_{a\in \cA^N} E\bigg[ Q(X,a(X),b(X,a(X))) \bigg].
  \ee
  where $b(x,a(x))$ is understood as appling $b_k=b_k(x_k,a_k(x_k))$ at stage $k$ in the recursive 
  definition \eqref{eq:recu-xk-ak-bk}.
  }
\end{lem}

We intend to apply the previous result to functionals such as 
$$
  Q(x,a,b())=\varphi(X_{N,x}^{a,b(.,a)})\ \equiv \ J(x,a,b(x,a)).
$$

\paragraph{Algorithm.} In view of the r.h.s. of \eqref{eq:maxmin-N}, we can propose a natural global SGDA algorithm 
that deals with only one 'min-max' expression
 \be
    & & \min_{b(.)\in {\hat B}^N} \bigg( \max_{a\in {\hat A}^N} J(x,a(x),b(x,a(x))) \bigg)
 \ee
where $\hat A$ is a NN space for functions : $\R^d \conv A$, and $\hat B$ is a NN space for function: $\R^d\times A \conv B$,
and $J(x,a(x),b(x,a(x)))$ is evaluated using the recursive definition \eqref{eq:recu-xk-ak-bk}.
}

\begin{example}
  We look for the exact solution for a 2d example where 
  $f(x,a,b)=(b-a)\scriptsize\VECT{1\\0}$, $\varphi=$"2 holes", $(a,b)\in[-1,1]$.
  Because the dynamics only advects in the $x$ direction, this is like a 1D problem.

  We first consider a one-time step evolution for a 1D problem. The problem is to compute the following
  $$V_0(x)=\min_a \max_b \varphi(F(x,a,b))$$
  where $F(x,a,b)=x+ \dt (b-a)$, with $x\in \R$. We also first assume that $\varphi(x)=\max(\varphi_1(x),\varphi_2(x))$ 
  where $\varphi_1(\downarrow)$ and $\varphi_2(\uparrow)$, and $\varphi_1,\varphi_2$ continuous.
  We assume furthermore that the minimal value of $\varphi$ is reached at some point $x_{min}$, and that $\varphi$ is symmetric
  with respect to the axis $x=x_{\min}$. 
  In that case we first remark that $\max_b \varphi(F(x,a,b))=w(x-a\dt)$ where $w(x)=\max(\varphi_1(x-\dt), \varphi_2(x+\dt))$,
  and we finally can show that
  \be\label{eq:ex1-solution} 
    V_0(x)=\min_{a\in[-1,1]} w(x-a\dt) =  
      \left\{
      \barr{ll}
        \varphi(x) &  x\leq x_{\min}-\dt\\
        \varphi(x) &  x\geq x_{\min}+\dt\\
        \varphi(x_{\min}-\dt) &  x\in [x_{\min}-\dt,x_{\min}+\dt]
      \earr
      \right.
  \ee 
  (where by symmetry $\varphi(x_{\min}-\dt)=\varphi(x_{\min}+\dt)$).
  This amounts to bound from below the $\varphi$ function, that is : 
  $$ V_0(x) \equiv \max(\varphi(x),\varphi(x_{\min}-\dt)),\quad x\in \R.
  $$
  (In the case $\varphi$ is not symmetric, there is also an expression of $V_0$ but it becomes more
  difficult to identify the minimal value; the expression for $w(x)$ is the same, but then the 
  front comming from the left, $\varphi_1(x-\dt)$, and from the right
  $\varphi_2(x+\dt)$, may not cross at $x_{min}$ but at some different point $\bar x(\dt)$ depending on time ...)

  {\color{blue} {\bf Note 1: a second proof for \eqref{eq:ex1-solution}.}
  We consider the case of $\varphi(x)=\min(1,|x|-1)$, and $\dt\leq 1/2$.
  For any $a\in[-1,1]$, choosing $b=a$ leads to $F(x,a,b)=x$ and $V_0(x)\geq \varphi(x)$. Also,
  for $|x|\geq \dt$, for instance if $x\leq -\dt$, in order to minimize one should take $a=-1$ (so that $F(x,a,b)$ goes to the right), 
  but then choosing $b=a$ will lead to $V_0(x)=\varphi(x)$.
  Finally for $|x|\leq \dt$, for instance $x\in[-\dt,0]$, to minimize one should try to reach $F(x,a,b)=0$, that is, $a=x/\dt$ (if they where no $b$ control).
  However for any value of $a$,  $b:=a-1$ (if $a\geq 0$) or $b:=a+1$ (if $a\leq 0$) leads to $F(x,a,b)=\pm \dt$ and therefore $V_0(x)\geq \varphi(\dt)$ 
  (and for any $a$ this value is reached by some $b$).
  This proves that $V_0(x) = \max(\varphi(x),\varphi(\dt))$, $\forall x\in \R$, which is the same formula as \eqref{eq:ex1-solution}.

  If further iterations in time are done, then the solution remains unchanged,
  that is, if we start with $V_1(x)=\max(\varphi(x),\varphi(\dt))$,
  then $V_0(x)=\min_a\max_b V_1(F(x,a,b))\equiv V_1(x)$.
  }

  {\color{blue} {\bf Note 2: the continuous problem.}
  The HJI equation for the 1D problem with dynamics $f(x,a,b)=(b-a)$ is
  $v_t + \max_a \min_b (-f (x,a,b) v_x ) =0$, where $a,b\in [-1,1]$, and initial data $v(0,x)=\varphi(x)$.
  We see that 
    $$ \max_a \min_b ((a-b)v_x) = \max_a (a v_x) + \min_b (-b v_x) = |v_x| - |v_x| = 0.
    $$
  Hence the HJI equation becomes $v_t(t,x)=0$ with $v(0,x)=\varphi(x)$, therefore $v(t,x)=\varphi(x)$.
  }

  Application to the 2D case. First in the case of $\varphi(x):=\min(1,\sqrt{x_1^2+x_2^2}-1)$, for $x=(x_1,x_2)$:
  for each given $x_2\in \R$, $x_1\conv \varphi((x_1,x_2))$
  has the desired 1D behavior with $x_{1,\min}=0$, therefore
  \be\label{eq:V0-ex-formula}
     V_0(x)=\max(\varphi(x),\varphi((-\dt,x_2))).
  \ee

  Finally for the case of a "2-holes" function such as 
  $\varphi(x):=\min(\varphi_0(x-e),\varphi_0(x+e))$, 
  with $\varphi_0(x)=\min(1,\sqrt{x_1^2+x_2^2}-1)$, $x=(x_1,x_2)$ and $e=(1,0)$, 
  one can show that  the same formula \eqref{eq:V0-ex-formula} is valid.
  (In particular the $0$-level set of $V_0$ is the same as for $\varphi$ if $\dt<1$ or it is empty if $\dt>1$.) 

\end{example}

{\color{red}
\begin{example}
  {\bf (OB 27/10)}
  Here $d=2$, $\varphi(x,y)=\min(0.5, \max(-0.5, \max(|x|,|y|)-1))$,
  $f((x,y),a,b)=(\cos(\pi(a-b)), a+b)$, with controls $a,b\in [-1,1]$.
  Le problème est alors 
  $$
     V_0(x) = \max_{\mb[\cdot]} \min_{a} \varphi(X_{N,x}^{a,\mb[a]}).
  $$
  avec typiquement $X_{1,x}^{a,b}=F(x,a,b) = x + \dt f(x,a,b)$.

  Dans le code HJB j'ai résolu
  $ v_t + \max_a\min_b( -f (x,a,b)\cdot \nabla_x v)=0. $
  avec $v(0,x)=\varphi(x)$, et j'affiche $v(T,x)$ pour $T=1$.
\end{example}
}

%% file: app-algo.tex

\section{algos : [file \texttt{app-algo.tex}]}

We can consider three algorithms related to the previous fonctionnal, depending on how we choose the controls.
\begin{itemize}
    
  \item Algorithm 1 \COMMENTG{\texttt{'Global2Controls'}} :
    utilizes the same control over all time steps, that is, sequences of the form 
    $(a,\dots,a)\in \cA^N$ and $(b,\dots,b)\in \cB^N$ for given controls $(a,b)\in \cA\times \cB$.
  \item Algorithm 2 \COMMENTG{\texttt{'GlobalWithTime2Controls'}} :
    allows for different controls  $(a_0,\dots,a_{N-1})\in \cA_{t,x}$ and $(b_0,\dots,b_{N-1})\in \cB_{t,x}$,
    yet only one neural network (in time and space) will be used for encoding $(a_0,\dots,a_{N-1})$ 
    (as $a_k(x)=a(t_k,x)$ where $a\in \cA_{t,x}:=\Meas(\R^{d+1},A)$)
    as well as one neural network for encoding $(b_0,\dots,b_{N-1})$
    (as $b_k(x)=b(t_k,x)$ where $b\in \cB_{t,x}=\Meas(\R^{d+1},B)$)
  \item Algorithm 3 \COMMENTG{\texttt{'Local2Controls'}} :
    allows for different controls  $(a_0,\dots,a_{N-1})\in \cA^N$ and $(b_0,\dots,b_{N-1})\in \cB^N$,
    and the optimization is performed backward from $t_{N-1}$ to $t_0$ (see below).
\end{itemize}

In the same way, the notation
 \be\label{eq:argminmax-ab-notation-AB}
     (a^*,b^*) \in \argmin_{a\in\cA} \argmax_{b\in \cB} \E\big[  Q(X_n,a,b)  \big]
 \ee
  will be used when the following two assertions are satisfied:
  \begin{subequations} \label{eq:argminmax-ab-AB}
  \be
    & & a^* \in \argmin_{a\in \cA} \max_{b\in \cB} \E\big[  Q(X_n,a,b) \big] 
    \\
    & & b^* \in \argmax_{b\in\cB} \E\big[  Q(X_n,a^*,b)  \big].
  \ee
  \end{subequations}

%% file: app-algo-old.tex
\section{Algorithms old \texttt{[app-algo-old.tex]}}

\medskip\noindent
\black
{\bf Algorithm \blue{2} (using possibly different controls at each time step : $(\ha_1,\dots,\ha_n)$ and $(\hb_1,\dots,\hb_n)$)}

Let $\hat \cA$ be an approximation space for feedback controls.
Let $\hat \cB$ be an approximation space for feedback adverse controls.
Let $\Dt>0$ be a time step, $p\in \N$ with $p\geq 1$, an  $h=\frac{\dt}{p}$ (a substep).
Let $\hV_N:=g$. We iterate backwards, for $n=N-1,\dots,0$:
\begin{quote}
  $(i)$
  $ (\ha_n(.),\hb_n(.)) \in \argmin_{a\in \hat\cA} \max\limits_{b\in \hat \cB} \E[ \mG^{a,b}(\hV_{n+1})(X) ] $\\
  where 
  $$\mG^{a,b}(V)(x) := G(x,a,b) \vee e^{-\ml \dt} V(F(x,a,b)), \quad
    \mbox{with}\ G(x,a,b) :=\max\limits_{0\leq k\leq p-1} e^{-\ml k h} g(Y^{a,b}_{x,k})
  $$
  and
  $Y^{a,b}_{x,0}=x$, $y^{a,b}_{x,k+1}=F_h(Y^{a,b}_{x,k},a(x),b(x))$ for $k\geq 0$
  (characteristic with
  controls $a(x)$, $b(x)$ fixed at the foot of the characteristic)
  and $F(x,a,b):=Y^{a,b}_{x,p}$.

  $(ii)$
  $\hV_{n}(x) = \mG^{\ha_n,\hb_n}(\hV_{n+1})(x). $
\end{quote}
Step $(i)$ can be approximated by an SGDA (Stochastic Gradient Descent-Ascent) algorithm.

Hence after $n$ time steps, with $\ha=(\ha_0,\dots,\ha_{n-1})$ and $\hb=(\hb_0,\dots,\hb_{n-1})$ 
\rd{
\begin{align*}
  \hV_{0}(x)= & J_0(x,\ha,\hb) = G(x,\ha_0,\hb_0) \vee e^{-\ml \dt} G(X^{\ha,\hb}_{x,1},\ha_1,\hb_1) \vee 
  \cdots \vee e^{-\ml n\dt} G(X^{\ha,\hb}_{n,x},\ha_{n},\hb_{n}) 
\end{align*}
}
\black{{\bf OB.} Sorry what I had in mind was :
\beno
  \hV_{0}(x)= & J_0(x,\ha,\hb) = G(x,\ha_0(x),\hb_0(x)) \vee e^{-\ml \dt} G(X^{\ha,\hb}_{x,1},\ha_1(X_{x,1}),\hb_1(X_{x,1})) \vee 
  \cdots \vee e^{-\ml n\dt} G(X^{\ha,\hb}_{n,x},\ha_{n}(X^{\ha,\hb}_{n,x}),\hb_{n}(X^{\ha,\hb}_{n,x})) 
\eeno
}
where $X^{a,b}_{x,0}=x$, $X^{a,b}_{x,k+1}=F(X^{a,b}_{x,k},a_k(x),b_k(x))$, $k\geq 0$.

\medskip\noindent
{\bf Algorithm \blue{1}} 
(using same controls at each time step : $(\ha,\dots,\ha)$ and $(\hb,\dots,\hb)$)
\rd{
Ok.  But Why not define for $a$,$b$ functions of x
\begin{align}
    J(x,a, b) := G(x,a(x),b(x)) \vee e^{-\ml \dt} G(X^{\ha,\hb}_{x,1}, a(X^{\ha,\hb}_{x,1}) , b(X^{\ha,\hb}_{x,1}))\vee 
  \cdots \vee e^{-\ml n\dt} G(X^{\ha,\hb}_{n,x},a(X^{\ha,\hb}_{n,x}),b(X^{\ha,\hb}_{n,x})) 
\end{align}
And solve
\begin{align*}
(\ha(.),\hb(.)) \in \argmin_{a\in \hat\cA} \max\limits_{b\in \hat \cB} \E[J(X,a, b)] 
\end{align*}
Supposing that we have a saddle point by for the global problem with J
\black{{\bf OB:} OK!}
}

%% file: app-num.tex
\section{old - Numerical examples [file \texttt{app-num.tex}]} \label{sec:other-contexts}

\becolblue

{\bf Example 1.}
We consider in $\R^2$ the two-dimensional dynamics
$$f(x,a,b)= a e_1 + b e_2 \equiv \VECT{a\\b}$$
We look for the stable region under controls $a\in A=[-1,1]$ and uncertainties $b\in[-\frac{1}{2},\frac{1}{2}]$:
$$ \mO^\cK_\infty:=\{ x\in \R^2,\ \exists a[\cdot]\in \mG,\ \forall b(.)\in \cB,\ \forall t\geq 0,\ y_x^{a[b],b}(t)\in \cK\},$$
(where $\mG$ is the set of non-anticipative controls with values in $A$, and $\cB$ is ...)
where $y(t)=y_x^{a,b}(t)$ is the solution of
\beno
    & & \dot y (t) = f(y(t),a(t),b(t)),\quad t\geq 0,\\
    & & y(0)=x.
\eeno
\rd{Is the control depending on $t$ ?
Is it 
\beno
    & & \dot y (t) = f(y(t),a(y(t)),b(y(t))),\quad t\geq 0,\\
    & & y(0)=x.
\eeno
}
\black{{\bf OB:} we will discuss this (As far as I understand : the first allows disconstinuous - only measurable controls and is the correct general math formulation; 
the second formulation assumes Lipschitz feedback functions for well-posedness, 
and is in general the one we use for the code and approximations)}

Here the state constraints $\cK$ are composed of a given box (rectangular) domain $\cK_0:=\prod_{i=1}^d [A_i,B_i]$,
and we impose a rectangular obstacle $\cO:=\prod_i]\ma_i,\mb_i[$ so that $\cK:=\cK_0\backslash\cO$. 

The box domain is  modelized by $g_{\cK_0}(x)=\min(\eps_0, \max\limits_i \max(A_i-x_i,x_i-B_i))$
(so that $g_{\cK_0}(x)\leq 0$ $\equivalent$ $x\in \cK_0$), 
and the rectangular obstacle by 
$g_{\cO^\complement}(x)=\min(\eps_0, - \max\limits_i \max(\ma_i-x_i,x_i-\mb_i))$
(so that $g_{\cO}(x)\leq 0$ $\equivalent$ $x\in\cO^\complement$, where 
$\cO^\complement=\R^2\backslash \cO$).
Then we set $g(x)=\max(g_{\cK_0}(x),g_{\cO^\complement}(x))$,
so that $g(x)\leq 0$ $\equivalent$ $x\in \cK:=\cK_0\backslash\cO$.

\begin{rem}
  \black{\bf OB (correctif).} 
  If we do not impose the box constraint (the only condition is to avoid the obstacle $\cO$), then we can take simply 
  $g(x):=g_{\cO^\complement}(x)$.
\end{rem}

\eecol